\documentclass[bj,preprint]{imsart}
\usepackage{amsthm, amsmath, natbib}
\usepackage{amsfonts, amssymb}
\usepackage{graphicx, color, latexsym}
\usepackage{accents}
\usepackage{tikz}
\usepackage{tikz-qtree}
\usepackage{pgfplots}
\usepackage[]{natbib}
\usepackage{bookmark}

\theoremstyle{plain} 
\newtheorem{thm}{Theorem}
\newtheorem{rmk}{Remark}

\newtheorem{cor}{Corollary}
\newtheorem{lem}{Lemma}
\theoremstyle{remark}

{}

\def\1{1\!{\rm l}}
\newcommand{\leqa}{\lesssim}
\newcommand{\geqa}{\gtrsim}

\newcommand{\EM}{\ensuremath}

\newcommand{\al}{\alpha}
\newcommand{\be}{\beta}

\newcommand{\ga}{\gamma}

\newcommand{\ka}{\kappa}

\newcommand{\si}{\sigma}

\newcommand{\veps}{\varepsilon}

\newcommand{\cA}{\EM{\mathcal{A}}}
\newcommand{\cB}{\EM{\mathcal{B}}}
\newcommand{\cC}{\EM{\mathcal{C}}}
\newcommand{\cD}{\EM{\mathcal{D}}}

\newcommand{\cF}{\EM{\mathcal{F}}}

\newcommand{\cH}{\EM{\mathcal{H}}}
\newcommand{\cI}{\EM{\mathcal{I}}}

\newcommand{\cL}{\EM{\mathcal{L}}}

\newcommand{\cN}{\EM{\mathcal{N}}}

\newcommand{\cS}{\EM{\mathcal{S}}}

\newcommand\munderbar[1]{%
  \underaccent{\bar}{#1}}

\definecolor{blendedblue}{rgb}{0.2,0.2,0.7}

\definecolor{commentpaul}{rgb}{0,0.6,0}
\definecolor{amethyst}{rgb}{0.6, 0.4, 0.8}

\DeclareMathAlphabet{\mathpzc}{OT1}{pzc}{m}{it}

\newcommand{\RR}{\mathbb{R}}

\newcommand{\given}{\,|\,}

\newcommand{\vep}{\varepsilon_n}

\newcommand{\mockalph}[1]{}

\newcommand{\eps}{\varepsilon}

\newcommand{\bi}{\begin{enumerate}[label=\roman*)]}
\newcommand{\ei}{\end{enumerate}}
\newcommand{\ba}{\begin{array}{rcl}}
\newcommand{\ea}{\end{array}}

\begin{document}

\def\spacingset#1{\renewcommand{\baselinestretch}%
{#1}\small\normalsize} \spacingset{1}

%%%%%%%%%%%%%%%%%%%%%%%%%%%%%%%%%%%%%%%%%%%%%%%%%%%%%%%%%%%%%%%%%%%%%%%%%%%%%%
\begin{frontmatter}

\title{Leveraging tails for adaptation}
%\title{A sample article title with some additional note\thanksref{T1}}
\runtitle{Leveraging tails for adaptation}
%\thankstext{T1}{A sample of additional note to the title.}

\begin{aug}
\author[A]{\fnms{Sergios} \snm{Agapiou}\ead[label=e1]{agapiou.sergios@ucy.ac.cy}}
\author[B]{\fnms{Isma\"el}~\snm{Castillo}\ead[label=e2]{ismael.castillo@sorbonne-universite.fr}}
\author[C]{\fnms{Paul}~\snm{Egels}\ead[label=e3]{paul.egels@sorbonne-universite.fr}}
%%%%%%%%%%%%%%%%%%%%%%%%%%%%%%%%%%%%%%%%%%%%%%
%% Addresses                                %%
%%%%%%%%%%%%%%%%%%%%%%%%%%%%%%%%%%%%%%%%%%%%%%
\address[A]{Department of Mathematics and Statistics,
University of Cyprus, Nicosia, Cyprus.
\printead{e1}}

\address[B]{Sorbonne Université, LPSM; 4, place Jussieu, 75005 Paris,
France.
\printead{e2}}

\address[C]{Sorbonne Université, LPSM; 4, place Jussieu, 75005 Paris,
France.
\printead{e3}}
\end{aug}
\begin{abstract}
%The text of your abstract. 200 or fewer words. \re{do not forget anon0 when submitting}
We consider contraction of Bayesian posterior distributions in nonparametric settings where coefficients of a function over a basis or dictionary are given priors with $p$--exponential tails, including Laplace tails $(p=1)$ and heavier tails $(p<1)$.
It is shown that contraction rates improve as $p$ decreases and that  full adaptation to smoothness, up to logarithmic factors, is obtained in an appropriate $p\to 0$ regime. As applications, we consider both series priors in white noise regression and shallow ReLU neural networks in random design regression. In particular, we show that overparametrised shallow ReLU networks can adapt to any regularity $0\le \beta\le 2$. Through a simulation study, we show strong empirical agreement with the behavior predicted by our theory. % 114 words
\end{abstract}

\begin{keyword}
\kwd{Frequentist analysis of Bayesian procedures, adaptation, nonparametric regression, Bayesian neural networks, overparametrization.}
\end{keyword}

\end{frontmatter}

\section{Introduction}
A central goal in nonparametric statistics is \emph{adaptation}: the ability of an estimator to perform simultaneously and optimally across a wide variety of settings with little to no tuning. When inference is carried out over a class of functional spaces, it is desirable that the estimator automatically adapts to unknown features of these spaces, such as smoothness, geometry, sparsity or other finer structural properties. A large body of literature has focused on adaptation: Lepski's method \cite{lepski90, lepski91}, thresholding \cite{djkp95} and model selection \cite{bbm99} are amongst the most well-known non-Bayesian approaches. Bayesian methods, on the other hand, have a natural ability to achieve   adaptation, as we discuss in more detail below, by choosing  prior distributions that are flexible enough to achieve this task (one possibility is for instance to draw certain prior parameters at random in a hierarchical Bayes fashion). %Another classical problem is adaptation to sparsity in high-dimensional settings \cite{brt09, sucandes16}.  

Recently, motivated by the remarkable empirical success of deep learning methods, there has been a growing interest in understanding how neural networks can automatically learn structural parameters, such as smoothness of functions or `effective' dimensions, for instance in  regression settings exhibiting a compositional structure  as in \cite{jsh20, kl21} or for 
data lying on geometric structures (e.g.~\cite{nakada2020adaptive}). 
While the above  works prove that optimal convergence rates can indeed be achieved by {\em appropriately} choosing the network's architecture, this choice still typically depends on parameters unknown to the statistician. In principle a solution to this {\em adaptation} problem could be to compare empirically risks for different architectures (by a method such as model selection or cross-validation); this would be particularly demanding computationally, and in practice overparameterized networks are used. These have a width or a depth larger than what traditional bias--variance trade--offs would suggest as optimal ones, but still empirically perform remarkably well. It is therefore particularly desirable to understand statistical properties of overparameterized networks, in particular, whether they can achieve adaptation to smoothness or to other structural parameters.

%A particularly desirable property is the ability of methods to attain convergence rates that automatically take into account `effective' dimensions or smoothness parameters when such geometric or structural features are present in the data but unknown to the statistician. A central goal is therefore to develop theoretical guarantees for methods used in practice; in particular, understanding the behavior of overparametrized estimators is especially important, since many neural network-based architectures employed in applications are overparametrized. 
%\ma{(what do we mean by overparametrized in the context of projection estimators?)}

 In this work, we take a step in this direction by showing, in a simple regression setting, that overparameterized prior distributions, with similar tails to those arising as outputs of typical neural networks with random Gaussian weights, have natural adaptation properties. Our results are relevant both for adaptation properties of classical projection estimators in regression as well as for neural network-type estimates --here for simplicity we consider shallow networks--. We now briefly review adaptation methods from a Bayesian perspective.
 
In the Bayesian setting, Gaussian processes (henceforth GPs) are among the most celebrated and widely used priors in machine learning and statistical applications (see e.g.~\cite{rw06,liu2020gaussian}). Thanks to pioneering work by \cite{vz08} building on the general contraction rate theory of \cite{ggv00}, contraction rates of posterior distributions for GPs are now fairly well understood from a theoretical perspective. While {\em per se} Gaussian processes are not adaptive to smoothness (see \cite{ic08}), a large body of literature has shown that GPs can be made adaptive to (homogeneous) smoothness \cite{vz09, Svz13, ckp14, ksvv16, RS17}  as well as to geometry in manifold settings \cite{RosaAdaptive,YangDunsonManifold, tang2024adaptive}, provided they are properly rescaled (depending on the observed data) or their parameters are suitably estimated or drawn from a hyperprior. 

However, several drawbacks are associated with Gaussian process approaches. First, achieving adaptation  requires an additional computational layer in order to adjust scale or regularity parameters,  which can render GP regression computationally demanding, especially in high-dimensional and data-rich regimes, see for example \cite{abps14}. A number of works aim to mitigate this complexity by developing scalable approximations of GPs, for instance through sparse variational inference \cite{nieman2025adaptive} or local approximations (see e.g.~\cite{szabo2025vecchiagaussianprocessesprobabilistic} for further references). Second, and more fundamentally, recent negative results show that even properly rescaled GPs may fail to adapt when finer notions of adaptivity are required, such as adaptation to compositional structures \cite{JSHGPcomp} (see also \cite{abraham2023deep} in the context of inverse problems) or inhomogeneous smoothness \cite{awLapInv}. In order to obtain stronger forms of adaptation, %and go beyond the strictly Gaussian case, 
a number of works advocate and study deep Bayesian methods, such as Bayesian deep GPs \cite{damianou2013deep,fsh23,cr25} and Bayesian deep (Gaussian) neural networks \cite{KimKong}. These approaches exhibit stronger adaptive properties compared to standard GPs, notably with respect to compositional structures. Nevertheless, with the currently available theory, they still require hyperparameter (or hyperprior) estimation. Contributions on posterior rates for Bayesian deep networks include \cite{polson2018posterior} for spike--and--slab priors on weights (see also \cite{cherief-abdellatif20a, baietal20, ohn2024adaptive} for variational Bayes counterparts) and \cite{lee2022asymptotic, kong2023masked, KimKong}; a review on Bayesian deep neural networks can be found in \cite{arbel2023primer}.

%Another line of work leaves the Gaussian framework entirely by introducing heavier--tailed priors.  %which appear to possess remarkable adaptive properties. 
It turns out that replacing the Gaussian distribution with heavier-tailed distributions can lead to improved contraction rates. One example is the Laplace--Besov prior, which replaces Gaussian coefficients by double--exponential ones and has recently been studied in a variety of nonparametric settings, including regression, density estimation, and inverse problems  %where some form of inhomogeneity arises 
\cite{awLapInv,GiordanoLapDens,dolera2024strong}. These works build on a more general theory for $p$--exponential priors, corresponding to densities proportional to $\exp({-|x|^p/p})$ and whose tails interpolate between Laplace ($p=1$) and Gaussian ($p=2$), which was initiated in \cite{AgapiouPEXP} and later complemented by \cite{AgapiouSavva}. Interestingly, although the contraction rates obtained in \cite{AgapiouPEXP} for such $p$--exponential priors ($1 \le p < 2$) are not minimax optimal, for $1\le p<2$ they are still polynomially faster than the corresponding Gaussian rates ($p=2$). Moreover, when scaling and/or regularity hyperparameters are drawn at random, these works establish adaptation properties of $p$--exponential priors ($1 \le p \le 2$) with respect to inhomogeneous smoothness in an $L_2$-sense (see also \cite{ace} for further discussion on adaptation to inhomogeneous smoothness). However, such $p$-exponential priors still require some form of hyperparameter estimation in order to achieve adaptation, and proof techniques crucially rely on log-concavity of the prior density, which only holds for $p\ge 1$.
%However, interestingly, it was shown (e.g. in \cite{AgapiouPEXP}) that when complete adaptivity is not ensured by such added estimation; the $p$--exponential posterior ($p <2)$, although not minimax optimal, still contracts at a (polynomially) faster rate than the corresponding Gaussian posterior $(p=2)$.  

Very recently, \cite{AC} initiated the theoretical study of even heavier--tailed priors, for instance with polynomially decaying tails such as the Student--$t$ distribution. One main advantage of such priors is that they only require a fixed universal deterministic (in particular, data--independent) rescaling, so that no hyperprior or hyperparameter estimation is needed to achieve adaptation. In \cite{ace}, these heavier--tailed priors were shown to possess strong adaptive properties with respect to non--homogeneously smooth Besov functions, this time in any $L_r$--norm ($r \ge 1$).
% , \ma{in contrast to the results in the $p$--exponential case (do we need this? "this time" points to this already}. 
When applied to overparameterized deep Bayesian neural networks, \cite{ce25}  showed that they yield posteriors that adapt to compositional structures and manifold geometry, again without requiring hyperparameter estimation. %We also point out that one example of such a heavier--tailed prior, namely the Horseshoe prior \cite{hs1,hs2}, has long been used in practice and is known to perform well in high--dimensional sparse settings \cite{vkv14,vsv17}. \sbl{distracting maybe, remove?}%\ma{shal we make clearer that HS is for high-dim not non-parametric settings?}

The primary goal of the present paper is to understand adaptation properties of $p$--exponential type priors over the full range of tail indices $0 < p \leq 2$, and in particular in the (formally defined below) limiting regime $p \to 0$. There are two main motivations for investigating the ``small $p$ regime" for $p$--exponential priors. The first is that the existing results on contraction rates for $p$--exponential priors indicate that the rate may (further) improve as $p$ decreases below $1$, suggesting that adaptation may be obtained for such priors by formally letting $p\to 0$ {\em without} the need to sample hyperparameters. The second motivation comes from the study of deep learning methods, in particular Bayesian neural networks with Gaussian priors on network weights. Indeed, several recent works show that, conditionally on the input, the output of deep neural networks with i.i.d.~Gaussian weights exhibits heavy--tailed behavior, \cite{pmlr-v97-vladimirova19a, Zavatone,noci2021precise}. More precisely, such outputs are of (generalized) sub--Weibull type with index $L/2$ (see e.g.~\cite{vladimirova2020sub}), which is closely related to the $p$--exponential distribution with $p = 2/L$, where $L$ denotes the depth of the network. These results suggest that, for deep neural networks with i.i.d.~Gaussian weights, increasing depth naturally induces marginally $p$-exponential outputs, in the regime where $p$ approaches $0$ inversely proportionally to the network depth.

Our study, on the one hand, complements the work of \cite{AgapiouPEXP}, 
%\cite{abraham2023deep} 
and builds a bridge between the very heavy--tailed (polynomial) case of \cite{AC}, where posteriors exhibit powerful adaptive properties, and the lighter heavy--tailed priors such as the Laplace--Besov case, which are more commonly used in practice. On the other hand, this will enable us to analyse {\em overparameterized} shallow ReLU neural networks with $p$--exponential weights, and show that the latter satisfy a remarkable adaptation to smoothness property.

%We argue that this investigation also provides insight into the adaptive behavior of ``deeper'' Gaussian methods, in particular deep Bayesian neural networks with Gaussian priors. Indeed, several recent works \cite{Zavatone,pmlr-v97-vladimirova19a, noci2021precise} show that, conditionally on the input, the output of deep neural networks with i.i.d.~Gaussian weights exhibits heavy--tailed behavior. More precisely, such outputs are of (generalized) sub--Weibull type (see e.g.~\cite{vladimirova2020sub}) and are thus closely related to a $p$--exponential distribution, with $p = 2/L$, where $L$ denotes the depth of the network. These results suggest that, for deep neural networks with i.i.d.~Gaussian weights, increasing depth naturally induces heavier--tailed behavior, corresponding to $p$ approaching $0$.

We now summarize our main contributions:
\begin{enumerate}
    \item In a white-noise regression setting, we study series priors with independent $p$--exponentially distributed coefficients ($p>0$) and derive an upper bound on the contraction rate of $\rho$--posteriors ($0 < \rho < 1$) toward the unknown regression function. This bound considerably broadens previously known results of \cite{AgapiouPEXP} for $1 \le p \le 2$. Notably, the obtained rate becomes near-optimal as the tail parameter $p$ approaches $0$, that is, as the prior becomes heavier-tailed. Also, previous works for $p\in[1,2]$ left open the question whether the obtained rates could be improved; here we provide a matching lower bound, valid also for standard posteriors ($\rho = 1$).
    
%    it was unclear from all previous works on $p$--exponential priors whether the rates, while (for $p<2$) better than Gaussian, were sharp, as for fixed $p$ the existing rates were still slightly above minimax. We provide a matching lower bound, valid also in the case of the usual posteriors ($\rho = 1$), thereby proving for the first time that these contraction rates are indeed sharp.
    
    \item In  random design nonparametric regression, we study a class of {\em overparameterized} ReLU shallow neural network priors, where the number of neurons is taken explicitly much larger than %what is, \ma{this may attract some questions} 
   the oracle number that achieves the optimal minimax rate. An independent $p$--exponential prior ($0 < p \le 1$) is placed on the weights of the hidden layer, and an upper bound on the contraction rate of the corresponding $\rho$--tempered posterior is obtained. Although different in nature from the series prior setting, these neural network rates  share a similar type of improvement as $p$ approaches $0$, with a rate that is near-minimax for small $p$;    
    \item In both of the above settings, we formally show that taking the limit $p \to 0$ %\ma{(in the infinitely informative data limit, $n\to\infty$, where $n$ is noise precision-amplitude in the white noise model, and sample size in nonparametric regression)} 
    as the sample size $n$ goes to infinity, 
    leads to fully adaptive posteriors over smoothness classes. These results can be linked to earlier findings of \cite{AC} for series priors and \cite{ce25} for neural networks. In particular, in the latter case, our results provide further insight into the adaptivity of deep Gaussian Bayesian neural networks (here small $p$  may be interpreted as mimicking the effect of increasing depth in a network; see also the discussion in Section~\ref{sec : disc}).
    
    \item We provide numerical experiments confirming that in both frameworks adaptation automatically occurs for the resulting estimators as the tail--index $p$ approaches $0$. In particular, ReLU shallow networks with overparametrised width of order $n$ and weights with $p$--exponential distributions and small deterministic scalings indeed achieve adaptation to smoothness in practical experiments, closely matching the predicted behaviour from our theory. %effectively having a  heavier--tailed priors. 
    %\sbl{[Sergios, you may wish to complement this a bit? (one or two sentences would do I think]}
\end{enumerate}

{\em Frequentist analysis of posterior distributions}. We consider a family of probability distributions 
$(P_f^{(n)})$ parametrized by $f \in \cF$, where $\cF$ is a (possibly infinite--dimensional) parameter space and $n \ge 1$ is an integer representing the informativeness of the observed data. %We assume \pe{usual assumption so that the posterior is defined}. 
Given data $X^{(n)} \sim P_{f_0}^{(n)}$ (sometimes simply denoted $X$) generated from a `true' $f_0 \in \cF$, we estimate $f_0$ from $X^{(n)}$ using a Bayesian procedure. Starting from a prior distribution $\Pi$ on $\mathcal{\cF}$ and a parameter $\rho \in (0,1]$, the $\rho$--tempered (or fractional) posterior distribution $\Pi_\rho[\cdot\given X]$ by, for any measurable set $B$ (\cite{gvbook, cstf}),
\begin{equation}\label{def : rhopost}
    \Pi_\rho[B\given X] = \frac{\int_B \exp(\rho  \ell_n(f,X)) \, d\Pi(f)}{
\int \exp(\rho  \ell_n(f,X)) \, d\Pi(f)},
\end{equation} 
where $\ell_n(f,X) := \log p_f^{(n)}(X)$ denotes the log-likelihood. For $\rho=1$, the fractional posterior distribution  coincides with the usual posterior $\Pi[\cdot\given X]$ given by Bayes' formula. In this paper, for technical simplicity we focus on $\rho$--tempered posteriors with fixed $\rho<1$, although we expect most results to carry over to the case $\rho=1$ (see also Section \ref{sec : disc} for more details).

\bigskip

\emph{Contraction rates of $\rho$--posteriors}.
For any positive loss function $d$ on $\cF \times \cF$, we say that the $\rho$--posterior contracts around $f_0$ at the rate $\varepsilon_n \to 0$ in $d$--loss if 
\begin{equation}\label{def : contrate}
E_{f_0}\Pi_\rho(d(f,f_0)\le M\varepsilon_n \, | X) \to 1, \qquad \text{
as $n \to \infty$},\end{equation} with $M>0$ a sufficiently large constant and $E_{f_0}$ the expectation under $P_{f_0}^{(n)}.$

\bigskip

{\em $p$--exponential distributions.} The prior distributions $\Pi$ we consider below are based on the $p$--exponential distribution on $\RR$: for $p>0$, this distribution has  density  given by, for  $t \in \RR$,
\begin{equation} \label{def : p-exp dist}
    h_p(t) \propto \exp \left \{ -\frac{|t|^p}{p} \right \}.
\end{equation} 
Such distribution generalizes the Gaussian ($p=2$) and Laplace $(p=1)$ distributions, with possibly heavier tails as the parameter $p$ gets closer to $0$. More generally, for any $p>0$, we say that a density function $h$ has $p$--tails (or $p$--exponential tails) if 
\begin{align}
h  & \text{ is symmetric about $0$ } \qquad (h(-t)=h(t) \text{ for all }t),
\label{conds} 
\end{align}
if for some constants $c_0,c_1>0$, one has
\begin{align}
h(t)  \ge c_0 e^{-c_1 t^{p}}\qquad t\ge 0,\label{condt}
\end{align}
and, denoting by $\overline{H}(x):=\int_x^{+\infty}h(t) \,dt$ the survival function associated with $h$, if for $q \in (0,p]$ and positive constants $d_0,d_1, M_0$,  
\begin{equation}
\overline{H}(x)  \le d_0 e^{-d_1 x^q}\qquad x\ge M_0.\label{condu}
\end{equation}
For example, the $p$--exponential density $h_p$ defined in \eqref{def : p-exp dist} satisfies conditions \eqref{conds}--\eqref{condt}--\eqref{condu} for any $q < p$ (see e.g. Lemma \ref{lem : sandwich} for the case $p<1$). Other commonly used prior distributions such as Weibull and Generalized--Weibull (see e.g. \cite{vladimirova2020sub}) distributions also have $p$--tails in the above sense.% \pe{Although condition \eqref{condu} also depends on a parameter $q$, we remark here that our results will hold for any $q \in(0,p]$. The influence of $q$ on the rate is at most logarithmic in $n$ and in any practical application one can easily relate $q$ to $p$.}

\bigskip
{\em Outline.} In Section \ref{sec : series}, we state our results for $p$--exponential series priors in white noise regression, both for given $p$ and in a regime $p\to 0$, together with a corresponding matching lower bound. Section \ref{sec : SNN} focuses on results for overparameterized shallow ReLU neural network posteriors. A simulation study illustrating both settings can be found in Section \ref{sec:simulations}, while the discussion in Section \ref{sec : disc} puts our results in perspective. Part of the proofs of the main results can be found in Section \ref{sec : proof}. % and a detailed discussion in Section \ref{sec : disc}. 
The Supplementary material contains the remaining proofs, a number of additional results as well as some technical Lemmata.

\section{Series priors} \label{sec : series}
In this section, for technical simplicity we focus  on the prototypical nonparametric Gaussian white noise model: for $f\in L^2[0,1]$ set of squared-integrable functions on $[0,1]$, one observes 
\[ dZ^{(n)}(t)=f(t)dt+dW(t)/\sqrt{n},\qquad t\in[0,1],\]
where $W$ is standard Brownian motion and $n\ge 1$.  For $(\varphi_k)_{k\ge1}$  an orthonormal %(for the canonical euclidean structure $\langle \, \cdot \, \mid \, \cdot \, \rangle$)
basis of $L^2([0,1])$, for the canonical inner product, one denotes $f_{k}:= \langle  f  \mid  \varphi_k  \rangle$ its basis coefficients. The white noise model above induces, once projected into the basis $(\varphi_k)$, observations in the so-called Gaussian sequence model (see the book by \cite{GineNickl})
\begin{equation}\label{def : gwn}
    X_k := f_{k} + \frac{1}{\sqrt{n} }\xi_k ,\qquad k\ge 1,
\end{equation}
where $\xi_k$ are i.i.d. $\cN(0,1)$ random variables. The observation sequence from the model \eqref{def : gwn} will be denoted $X := X^{(n)} \sim P_f^{(n)}$ and associated with the log-likelihood
\begin{equation}\label{def : likelihood gwn}
  \ell_n(f,X) := -\frac{n}{2} \sum_{k\ge 1} (X_k-f_k)^2. 
\end{equation}

{\emph{Definition of the prior}.} We define a prior $\Pi := \Pi(p,\al)$ on $f \in L^2[0,1]$, identified as the sequence of its (square summable) coefficients $(f_k)_{k \ge 1}$ by setting
\begin{equation} \label{def : prior series}
f_k = \sigma_k \zeta_k,
\end{equation}
where $\zeta_k$ are i.i.d.\,random variables with a density $h$ satisfying the $p$--tails conditions \eqref{conds}--\eqref{condt}--\eqref{condu} for some $p>0$. A possible choice of scaling sequence $(\sigma_k)_{k \ge 1}$ is, for $\al>0$,
\begin{equation} \label{def : alphsig}
\sigma_k = k^{-1/2-\al}.
\end{equation}

\emph{Regularity assumption on $f_0$, targeted rate}. Equipped with the previously defined prior distribution $\Pi = \Pi(p,\alpha)$ and from the likelihood formula $\eqref{def : likelihood gwn}$, one defines (for any $\rho<1$) the $\rho$--posterior $\Pi_\rho[ \cdot \given X]$ using \eqref{def : rhopost}.  We study these $\rho$--posteriors under the assumption $X^{(n)} \sim P_{f_0}^{(n)}$, where $f_0$ belongs to the hyperrectangle, for some $\beta,L >0$,
\begin{equation}\label{def : Holder ball}
    \cF^{\beta}(L) := \left\{ f = (f_k) \, : \, \underset{k \ge 1}{\max}\, \left( |f_k|\, k^{\beta+1/2}\right) \le L\right\}.
\end{equation}
The contraction rate of $\Pi_\rho[ \, \cdot \given X]$ around $f_0$ will involve an interplay between the prior parameters $(p,\alpha)$ and the true function smoothness parameter $\beta$. We define
\begin{equation}\label{def : gamma}
    \gamma = \gamma(p,\alpha,\beta) := \beta + \frac{p}{2}(\alpha - \beta),
\end{equation}
and the associated rate
\begin{equation} \label{def : rate}
 \veps_n = \veps_n (p,\al,\be) :=
 \begin{cases}
\,  n^{-\frac{\be}{2\ga+1}}=n^{-\frac{\be}{2 \be+p(\al-\be) +1}},& \qquad \al>\be,\\
\,  n^{-\frac{\al}{2\al+1}},& \qquad \al\le\be.
\end{cases}
\end{equation}

\subsection{Contraction rate for $p$--tails series priors.} 
For a prior \eqref{def : prior series}--\eqref{def : alphsig} with $p$--tails distributions, the rate  \eqref{def : rate} turns out to be the  $L^2$--contraction rate of the associated $\rho$--posterior. For $g\in L^2[0,1]$, let $\|g\|_2^2=\int_0^1 g(u)^2 du$. 
% in $L^2$--distance  $\| f - f_0\|_2 := \Big(\int_0^1 (f(x)-f_0(x))^2 \, dx \Big)^{\frac12}$.
\begin{thm}[Upper bound]\label{thm : conc series}  Let $p>0$ and  $\alpha,\beta >0$. Suppose $f_0 \in \cF^\be(L)$ for some $L >0$, assume $X^{(n)} \sim P_{f_0}^{(n)}$ from the model \eqref{def : gwn}. Then, for any $\rho \in (0,1)$, starting from the prior $\Pi = \Pi(p,\alpha)$ defined in \eqref{def : prior series}--\eqref{def : alphsig}, as $n \to \infty$, we have
    \[ E_{f_0} \Pi_\rho \left[ ||f-f_0||_2 \le M \varepsilon_n \given X\right] \to 1, \]
    where $\veps_n$ is given in \eqref{def : rate} and $M>0$ is a large enough constant.
\end{thm}
The proof can be found in Section \ref{proof : conc series}. % \re{To reduce the size of the file we can send the case $\al \le \be$ to the supplement and reduce the size of the proof of Theorem \ref{thm : p to zero} by referencing the proof of Theorem \ref{thm : conc series}} 
The rate obtained in Theorem \ref{thm : conc series} depends on the choice of the prior smoothness parameter $\alpha$ compared to the true smoothness $\be$. In particular, there is an elbow in the rate at $\alpha = \beta$. In the undersmoothing case $\alpha < \be$, the $p$--tails posterior contracts at the (slower than minimax) rate $n^{-\al/(2\al +1)}$; in the matching case $\al = \be$, the minimax rate is attained. In the oversmoothing case $\al > \be$, the rate is $n^{-\be/(2 \be + p(\al-\be) +1)}$ and improves as $p$ gets smaller. In particular, this rate (available for all $p>0$) matches the one obtained over the range $1 \le p \le 2$ in \cite{AgapiouPEXP}. When $p=2$, this rate corresponds to the known (sharp) rate of contraction of Gaussian processes obtained in \cite{vz08} and \cite{ic08}. 

\begin{rmk}
   Although Theorem \ref{thm : conc series} assumes $f_0$ belongs in the hyperrectangle $\cF^\beta(L)$ the same result holds true over (the richer) Hilbert--Sobolev Balls: this claim is formally proved for the oversmoothing prior ($\al > \be$) with $p<1$ in Theorem \ref{thm : conc series sobolev} of the Appendix.
\end{rmk}

\begin{thm}[Lower bound] \label{thm : lower bound}
%\re{Check notation in the proof.} 
Let $p \in (0,1]$, %\pe{[check, we should be able to get $p = 1$ too]} 
$L>0$ and $\alpha>\beta>0$. Assume the prior on $f$ is defined as in \eqref{def : prior series}--\eqref{def : alphsig} for the specific choice of density $h = h_p$ as in \eqref{def : p-exp dist}. Then there exists a function $f_0\in\cF^\beta(L)$ such that, if $X^{(n)} \sim P_{f_0}^{(n)}$ from the model \eqref{def : gwn}, then for any $\rho\in(0,1]$, as $n\to \infty$,
\[ E_{f_0}\Pi_\rho[\|f-f_0\|_2 < m \cdot \veps_n \given X] \to 0,\]
where $\veps_n$ is given in \eqref{def : rate} and $m>0$ is a small enough constant.
\end{thm}
The proof of this result can be found in  Appendix \ref{proof : lower bound}.
Theorem \ref{thm : lower bound} shows that in the heavier than Laplace case ($p <1$), the contraction rate $\veps_n$ obtained in Theorem \ref{thm : conc series} is tight, in the sense that there exists some function $f_0 \in \cF^\be(L)$ towards which the $p$--tails posterior (note that Theorem \ref{thm : lower bound} allows also for the standard posterior $\rho=1$) cannot contract at a faster rate than $\veps_n$. We believe such lower bound can be obtained also for lighter tails $1<p<  2$ with similar proof techniques, but here have focused  on the (harder) case $p \le 1$. In the next section, we further explore the rate improvement noted above in a regime $p\to 0$.  

\subsection{Adaptation with varying tails.} 

We showed that whenever $p$ decreases the contraction rate of the $p$--exponential process posterior improves. In this section, in order to obtain minimax contraction rate and smoothness adaptation, we design priors for which $p$ naturally decreases towards $0$.

Let $(\sigma_k)_{k \ge 1}$ be a positive sequence, $(p_k)_{k\ge1}$ be a sequence such that $0 < p_k \le 1$ for all $k \ge 1$. Consider a prior $\Pi$ on coefficients $(f_k)_{k \ge 1}$ with
\[ f_k = \sigma_k \zeta_k,\]
where $\zeta_k$ are independent $p_k$--exponential random variables with respective densities $h_{p_k}$ defined in \eqref{def : p-exp dist}. For some $\beta >0$, consider the following quantities:
\begin{equation}\label{def : rk and zk}
     r_k :=  \frac{  k^{-\beta - 1/2}}{\sigma_k} \qquad \text{and} \qquad z_k := \frac{r_k^{p_k}}{p_k}.
\end{equation}
The next Theorem  provides a result under generic conditions; practical choices of $p_k$ are considered  in Corollary \ref{thm : cor p to zero}. 
For any positive number $s >0$, we denote 
\begin{equation}\label{def : Ncutoff}
N_s := \lfloor n^{\frac{1}{2s +1}}\rfloor.
\end{equation}
\begin{thm}\label{thm : p to zero} 
    Let $f_0 \in \cF^\be(L)$ for some $\beta >0$ and $L \ge 1$. Assume $X^{(n)} \sim P_{f_0}^{(n)}$ from the model \eqref{def : gwn}. Consider the prior $\Pi$ defined above this statement and assume that $(\sigma_k)_{k\ge1}$ and $(p_k)_{k \ge1}$ are chosen such that, for $n$ large enough,
    \begin{align}
       % \label{eq : compatib cond 1} &\ora{\forall k \leq N_\beta, \qquad \sqrt{p_k}\exp\{1/{p_k}\} \geq \sigma_k,} \\
        \label{eq : compatib cond 2} &\forall k > N_\beta, \qquad r_k^{p_k} \geq 8.
    \end{align}
    \text{Furthermore, assume there exists a constant $\eta >1$, such that}
    \begin{equation} \label{eq : sum cond}
        \sum_{k \geq 1} \exp \{-z_k/4\} < + \infty \qquad \text{and} \qquad \sum_{k \leq N_\beta} z_k \lesssim N_\beta \log^{\eta}n .
    \end{equation} 
    Then, for any $\rho \in (0,1)$, there exist a constant $\eta' > 2\eta$, such that, as $n \to \infty$, we have 
    \[ E_{f_0} \Pi_\rho \left[   ||f-f_0||_2 \leq n^{- \frac{\beta}{2 \be +1}} \log^{\eta'} n \given X\right] \to 1.\]
\end{thm}
The proof of Theorem \ref{thm : p to zero} can be found in Section \ref{proof : p to zero}. Theorem \ref{thm : p to zero} provides prior conditions on the scaling $\sigma_k$ and tail parameters $p_k$ for which the posterior contracts at minimax rate $n^{-\be/(2\be +1)}$ up to logarithmic factors. Compatibility condition \eqref{eq : compatib cond 2} ensures that the prior is oversmoothing in a way (roughly $\sigma_k \lesssim k^{-\be -1/2}$) and that $p_k$ does not go too fast to $0$; an examination of Condition \eqref{eq : sum cond} suggests that a natural choice for $1/p_k$ is a logarithmic--type growth  in terms of $k$. This is confirmed in the next Corollary \ref{thm : cor p to zero}, whose proof can be found in Section \ref{proof : cor p to zero}, which provides two natural examples for which these conditions are satisfied.

\begin{cor}\label{thm : cor p to zero}
    The conditions of Theorem \ref{thm : p to zero} are met for the following choices of %sequences 
    $(p_k)$, $(\sigma_k)$
    \begin{enumerate}
        \item Polynomial decay of the scaling. For $\alpha > \beta >0$, $p_1 = p_2 =1$,
        \begin{equation}\label{eq:seriesvarp-a}\sigma_k = k^{-1/2 - \alpha} \qquad \text{and}\qquad \forall k \geq 3,\quad p_k = (\log k)^{-1}\log \log k.\end{equation}
        \item Faster than polynomial decay. For $\gamma >0$, $p_1 = p_2 =1$ and $c > 2.1 > \log 8$, 
        \begin{equation}\label{eq:seriesvarp-ot}\sigma_k = \exp(-\log^{1+\gamma} k) \qquad \text{and} \qquad \forall k \geq 3,\quad  p_k = c(\log k)^{-(1+\gamma)}.\end{equation}
%\ora{I believe this can be replaced by $3$ (or even $2.1>\log{8}$)}
    \end{enumerate}
\end{cor}
 The first natural choice matches the oversmoothing polynomial choice of sequence $\sigma_k$ in Theorem \ref{thm : conc series}. Compared to Theorem \ref{thm : conc series} (fixed $p$) for which the rate was polynomially slower than minimax, now choosing $p_k$ going to $0$ leads to contraction at quasi--minimax rate (up to log factors). With the choice \eqref{eq:seriesvarp-a} one obtains one--sided adaptation on the range $\al > \be$. In order to obtain full minimax adaptation, we use the second choice of sequence $\sigma_k = \exp(-\log^{1 + \ga}k)$ which ensures oversmoothing thanks to its faster than polynomial decay. This automatic oversmoothing choice matches the ones introduced in \cite{AC, ace} and similarly shows full adaptation without being restricted by hyperparameters (the choice of $\ga$ is free and  need not depend on $\be$).% knowledge is needed on $\be$ in order to choose $\ga$).

\section{Overparameterized shallow neural network priors}\label{sec : SNN}
In this section we focus on random--design nonparametric Gaussian regression. Given an integer $n \ge 1$, we observe $n$ i.i.d. pairs of random variables $(X_i,Y_i) \in [0,1] \times \RR$, with % with% \pe{[remark about the multivariate case]} from the model 
\begin{equation}\label{def : random design}
    Y_i = f(X_i) + \xi_i,
\end{equation}
where $(X_i)$ are i.i.d. from a distribution $P_X$ on $[0,1]$ and independently $\xi_i$ are i.i.d. $\cN(0,1)$ % random 
variables. For simplicity, we work in dimension $1$, but all our results extend to an input space $[0,1]^d, d\ge 2,$ in a straigthforward way. %A single observation $(X_i,Y_i)$ under \eqref{def : random design} has distribution denoted by $P_{f}$ such that the i.i.d. $n$--sample $((X_i,Y_i))_{1 \le i \le n}$ is distributed according to the product measure $P_{f}^n$ and associated with the log--likelihood
The sample  $((X_i,Y_i))_{1 \le i \le n}$ has distribution denoted $P_f^n$, where $P_f$ is the law of a single pair $(X_i,Y_i)$; the corresponding log-likelihood is
\[ \ell_n(f,(X,Y)) := -\frac{1}{2} \sum_{i =1}^n (Y_i - f(X_i))^2.\]
%From the log--likelihood formula above and 
Given a prior $\Pi$ on $f:[0,1] \to \RR$ to be defined below,  the $\rho$--posterior $\Pi_\rho[ \, \cdot \, | \, X,Y ]$ is formed using \eqref{def : rhopost}. We study it 
%these posteriors %under the assumption that
 assuming $((X_i,Y_i))_{1 \le i \le n}=:(X,Y)$ is distributed from $P_{f_0}^n$, where the true  $f_0$ belongs to a H\"older ball: denoting by $D^j f$ the $j$--th derivative of $f$, define
\begin{equation}\label{def : holdersmooth}
\cH^\beta(L) := \left\{f : [0,1] \to \RR \; :\; \max_{0\le i\leq\munderbar{\beta}}||D^if||_\infty  + \underset{x\neq y}{\sup}\frac{|D^{\munderbar{\beta}}f(x)-D^{\munderbar{\beta}}f(y)|}{|x-y|^{\beta-\munderbar{\beta}}} \leq L \right\},
\end{equation}
for $L \ge0$, $\beta \in (0,2]$ and where $\munderbar{\beta}\in\{0,1\}$ is the largest integer strictly smaller than $\beta$. %In this setting 
We aim to characterize the contraction rate of $\Pi_\rho[ \, \cdot \, \mid X,Y]$ around $f_0$ using the population loss
\begin{equation}\label{eq:sqloss} ||f-f_0||_{2,P_X}^2 := \int (f-f_0)^2 \, d P_X.\end{equation}

\subsection{Definition of the prior}
\textit{Shallow neural networks.} Functions in $\cH^{\beta}(L)$ can naturally be approximated by shallow neural networks (henceforth SNN). The realization of a shallow ReLU neural network with $M\ge 1$ neurons is a function% that can be written as
\begin{equation}\label{def : shallowNN}
    f_M : x\in[0,1] \mapsto \sum_{k=0}^{M-1} w_k (v_k \times x + a_k)_+ +b,
\end{equation}
where $b \in \RR$ and $w_k,v_k, a_k \in \RR$ for all $k \in \{0, \dots , M-1\}$. The ReLU activation function $x \mapsto (x)_+ := \max(0,x)$ is piecewise linear, so that the function $f_M$ is piecewise affine.% and continuous.
%therefore the functions defined by equation \eqref{def : shallowNN} are piecewise affine. 

To construct an approximation of a smooth function $f \in \cH^\beta(L)$ with an SNN, one can consider a uniform subdivision of $[0,1]$ in $M$ intervals $I_k := [k/M,(k+1)/M)$, and approximate $f$ by the piecewise affine function $f^\star_M$ changing slope and interpolating $f$ precisely at the boundary points of $I_k$. Lemma \ref{lem : approx shallow} (see Section \ref{proof : affine approx}) formalizes this intuition and controls the approximation error. 
For any $s \in \RR$, let us denote $N_s := 2^{m_s}$, where $m_s$ is the closest integer solution to $2^{m_s}=n^{1/(2s +1)}$, for $n\in\mathbb{N}$. 

{\em Ideal `oracle' approximator.} Suppose first, to fix ideas, that the smoothness parameter $\be$ of $f_0\in\cH^{\beta}(L)$ is known. In this case one could set $M=N_\beta$ and 
\begin{equation}\label{approxSNN}
f^\star_{N_\beta}(x)=f_0(0)+\sum_{k=0}^{N_\beta-1}w_{0;k}\left(x-\frac{k}{N_\beta}\right)_+.
\end{equation} 
Lemma \ref{lem : approx shallow} implies $||f_0-f^\star_{N_\beta}||_\infty\leq 2LN_\beta^{-\beta}$, for $\be\in(0,2]$, which is the optimal approximation error for $\beta$--H\"older functions. 

%\textit{Priors on overparameterized shallow neural networks.}
%In this section we consider priors over functions $f:[0,1]\to\RR$ constructed using overparameterized SNNs (that is with larger than `optimal' number of neurons). 
%Consider a truth $f_0\in\cH^{\beta}(L)$. According to Lemma \ref{lem : approx shallow} with $M=N_\beta$, there exists an approximating shallow network, 
%\begin{equation}\label{approxSNN}
%\hat{f}_{N_\beta}(x)=f_0(0)+\sum_{k=0}^{N_\beta-1}w_{0;k}\left(x-\frac{k}{N_\beta}\right)_+,
%\end{equation} 
%such that $||f_0-\hat{f}_{N_\beta}||_\infty\leq 2LN_\beta^{-\beta}$, that is, for $n\to\infty$ the shallow network $\hat{f}_{N_\beta}$ uniformly approximates the truth at the typical minimax rate for (direct) function estimation problems over  $\cH^{\beta}(L)$.

\textit{Priors on overparameterized shallow neural networks.} We now define an {\em overparameterized} prior  on shallow networks (that is, with possibly much larger than oracle number of neurons). Let us choose a number of neurons $N_\alpha \ge N_\beta$ (equivalently $\alpha \le \beta $; typically one can think of the choice $N_\al=n$, or $N_\al=n^{1-\delta}$ for some small $\delta>0$) and set %draw the weights and bias in the second layer only, 
\begin{equation}\label{SNN-prior}f = \sum_{k = 0}^{{N_\alpha-1}} w_k (x - a_k)_+ + b, \qquad w_k \overset{iid}{=} \sigma_n\zeta_k, \qquad a_k=k/N_\alpha, \qquad b\sim \pi_b, \end{equation}
for some deterministic $\sigma_n>0$ to be chosen, $\zeta_k$ independent and identically distributed random variables with density $h$ satisfying conditions \eqref{conds}--\eqref{condu}, and where the prior $\pi_b$ on the bias is any distribution with symmetric, continuous and strictly positive density on $\RR$, independent of the prior on the weights $(w_k)$ (note also that, for simplicity,  we have taken the shifts $a_k$ to be deterministic). 
%Notice that, since $N_\al \ge N_\be$, this SNN prior has equal or larger width than the approximating network in \eqref{approxSNN} (in this sense the number of neurons is larger than optimal and the prior is `overparameterized').

\subsection{Contraction results for overparameterized SNN priors}

For any $\rho \in (0,1)$ and $P,Q$ probability measures, the $\rho$--R\'enyi divergence is defined as
\begin{equation}\label{def : renyi}
    D_\rho(P,Q) := \frac{1}{\rho -1} \log \int \left( \frac{dP}{dQ} \right)^\rho \, dQ.
\end{equation} 
The next result examins contraction rates of $\rho$--posterior distributions under SNN priors. %($\rho\in(0,1)$) %We first establish rates in R\'enyi divergence, which under an additional clipping condition immediately imply rates in the more typical square loss \eqref{eq:sqloss}. 
\begin{thm}\label{thm:SNN}
Let $\beta\in(0,2]$ and $f_0\in \cH^{\beta}(L)$. Consider data $(X_i,Y_i)_{1 \le i \le n}$ generated from $P_{f_0}^n$ in model \eqref{def : random design}. Let $0 \le \alpha \le \beta$ and $\Pi$ be the overparameterized SNN prior defined in \eqref{SNN-prior}, with $0 < p \le1$, $q>0$ and some $\sigma_n$ to be specified below. For any $\rho \in (0,1)$ and  $D_\rho$ % the R\'enyi divergence defined 
as in \eqref{def : renyi}, there exists a large enough constant $M >0$, such that, as $n \to \infty$, 
\[ E_{f_0} \Pi_\rho \left[ \left\{ f \, : \, \frac1n D_\rho(P_f^n,P_{f_0}^n) \ge M \varepsilon_n^2 \right\} \mid X,Y\right] \to 0,\]
where, letting $\veps_n^\ast=n^{-\frac{\beta}{1+2\beta}}$ and $\veps_n^+=n^{-2/5}$ (equal to $\veps_n^\ast$ for $\beta=2$), $\veps_n$ is given by
\begin{enumerate}
    \item[i)] (Oracle $\sigma_n$) for $\sigma_n=N_\alpha^{-\frac{2}{2+p}}N_\beta^{\frac{2}{2+p}-\beta}\log^{-\frac{2}{q(2+p)}}(n)$,
\[\veps_n=\veps_n^\ast(N_\alpha N_\beta)^{\frac{p}{2+p}}\log^{-\frac{p}{q(2+p)}}(n);\]
    \item[ii)] (Non-oracle $\sigma_n$) for $\sigma_n=\veps_n^+/N_\alpha$
    \[\veps_n=\veps_n^\ast \: n^{\,p\:\left(\frac{1-\beta}{2+4\beta}+\frac15+\frac1{2+4\alpha}\right)}.\]
\end{enumerate} 
\end{thm} 

\begin{cor}\label{cor:SNN}
    In the setting of Theorem \ref{thm:SNN}, assume further $||f_0||_\infty \le F$, for some $F>0$. Define $T_F : f \mapsto -F \vee(f \wedge F)$ and consider the ``clipped'' posterior $\Pi_\rho^F [ \,  \cdot \, \given X,Y ] = \Pi_\rho [ \,  \cdot \, \given X,Y ] \circ T_F^{-1} $. As $n \to \infty,$
    \[ E_{f_0} \Pi_\rho^F \left[ \left\{ f \, : \, ||f-f_0||_{2,P_X} \ge M' \varepsilon_n \right\} \mid X,Y\right] \to 0,\]
    where $M'>0$ is a large enough constant (depending on $F$) and $\veps_n$ is defined as in Theorem~\ref{thm:SNN}.
\end{cor}
The proofs of Theorem \ref{thm:SNN} and Corollary \ref{cor:SNN} can be found in Section \ref{proof:SNN}. These results give conditions on the deterministic prior scaling $\si_n$ to obtain posterior contraction rates in $L_2(P_X)$--loss (or in R\'enyi divergence if clipping is not applied) with overparameterized SNN $p$-tailed priors. The first choice of $\sigma_n$ is said to be oracle as it depends on the unknown `true' smoothness $\beta$, while the second choice does not. Although both lead to polynomially slower than minimax rates (here the minimax rate is $\veps_n^*$), similarly to the series priors in Section \ref{sec : series},  this extra polynomial factor decreases as $p$ goes to $0$.

\begin{rmk}\label{rem:SNN1}
The case of lighter tails is also interesting (since it includes the case of Gaussian priors on the weights and bias) and is examined in Theorem \ref{thm:SNNhighb} in the supplement. As expected, the rate becomes worse as $p$ increases.  
%We comment on the case $p>1$ which is interesting since it includes the case of Gaussian priors on the weights and bias. There are two regimes depending on whether $\beta\le 1+\frac1{p\vee1}$ or $\beta> 1+\frac1{p\vee1}$. In the former, the situation is identical to Theorem \ref{thm:SNN}. In the latter, there is a slight variation, \ma{see Theorem \ref{thm:SNNhighb} in the supplement}. In either case, the rates get polynomially worse the larger $p$ is.
\end{rmk}

%Similarly to the series prior case, 
To obtain a minimax adaptive contraction rate, we now devise an overparameterized SNN prior with suitably chosen deterministic prior scalings $\sigma_n$ and decaying tail parameter $p_n$.

\begin{thm}\label{thm:SNNvarp}
Let $\beta\in(0,2],$ and $f_0\in \cH^\be(L)$. Consider data $(X_i,Y_i)_{1 \le i \le n}$ generated from $P_{f_0}^n$ in model \eqref{def : random design}. Let $\Pi$ be the overparameterized SNN prior defined as in \eqref{SNN-prior}, for $\alpha=0$ (hence the width is $N_\alpha\asymp n$), $\sigma_n=n^{-t}$ with $t>2.5$, and $\zeta_k$ independent and identically distributed according to $p_n$--exponential distributions (defined in \eqref{def : p-exp dist}) with $p_n=2/\log{n},\;n\ge8$. For any $\rho \in (0,1)$, there exists a large enough constant $M>0$, such that, as $n \to \infty$,
\begin{equation*}
    E_{f_0} \Pi_\rho \left[ \left\{ f \, : \, \frac1n D_\rho(P_f^n,P_{f_0}^n) \ge M \varepsilon_n^2 \right\} \mid X,Y\right] \to 0,
\end{equation*}
for any $\varepsilon_n\ge \varepsilon_n^\ast \sqrt{\log{n}},$ where $\varepsilon_n^\ast=n^{-\frac{\beta}{1+2\beta}}$ and $D_\rho$ is R\'enyi divergence. Assuming further $||f_0||_\infty \le F$ and considering the clipped posterior $\Pi_\rho^F[ \, \cdot\, | X,Y]$ as in Corollary \ref{cor:SNN}, for a large enough constant $M'>0$, as $n \to \infty$,
\begin{equation*}
    E_{f_0} \Pi_\rho^F \left[ \left\{ f \, : \, ||f-f_0||_{2,P_X} \ge M' \varepsilon_n \right\} \mid X,Y\right] \to 0.
\end{equation*}
\end{thm}

The proof can be found in Section \ref{proof:SNNvarp}. Theorem \ref{thm:SNNvarp} shows that, for an overparameterized (with $\alpha=0$) SNN prior with $p_n$--exponential weights rescaled by a sufficiently small polynomial factor $\sigma_n$ and sufficiently heavy tails  $p_n = 2 / \log n$, we obtain a  minimax fully {\em adaptive} contraction rate (up to a $\sqrt{\log n}$ factor). This choice of tail parameter $p_n =2 /\log n$ on an SNN $p$--exponential prior can be (informally) thought of as having a deep neural network prior with i.i.d. Gaussian weights and depth $L = \log n$ (see Section \ref{sec : disc} for more discussion).

{While the last result is theoretically appealing since it establishes adaptation over the full range $\beta\in(0,2]$ up to logarithmic factors, neural network priors of width $n$ can become computationally demanding for large sample sizes. To alleviate this, one can instead use networks of slightly smaller width $N_\alpha, \;0\le\alpha<2$ (and small $\alpha$ e.g. $\al=.5$), still yielding adaptation over the  range $\beta\in[\alpha,2]$. Also, the condition $t>2.5$ on the  scalings $\sigma_n$, can be relaxed, facilitating posterior sampling. See Remarks \ref{rem:SNN-varp} and \ref{rem:SNN-varptau} in the supplement for details.}

Moreover, the proof of Theorem~\ref{thm:SNNvarp} carries over, with minor modifications, to the choice
\begin{equation}\label{eq:varpotsnn}
\sigma_n=\exp\!\bigl(-a(\log n)^{1+\gamma}\bigr), \qquad
p_n=\frac{b}{(\log n)^{1+\gamma}},
\end{equation}
for \(\gamma>0\) and \(a,b>0\) such that \(ab>\log 8\), corresponding to the second regime of Corollary~\ref{thm : cor p to zero}. See Remark~\ref{rem:SNN-other} in the supplement for details. %\ma{for width $n$?}

\section{Simulations}\label{sec:simulations}
In this section we provide a numerical simulation study. Even though results with series priors in Section \ref{sec : series} were formulated in the white noise model, we compare the performance of the various priors considered in the following random design regression setting:
\begin{equation}\label{rdregression}y_i=f(x_i)+\xi_i, \quad i=1,\dots, n,\end{equation}
where $f:[0,1]\to\mathbb{R}$ unknown, the $x_i\in[0,1]$ are i.i.d. with uniform distribution on $[0,1]$ and $\xi_i$ are independent $N(0,\sigma^2)$ for $\sigma~=~1/4$ (the latter value has no special meaning and is chosen for easier visualization in the pictures below).

We define a true function through the series expansion 
\[f_0(x)=\sum_{k=1}^\infty f_{0,k}\varphi_k(x),\]
for $\varphi_k(x)=\sqrt2 \cos((k-1/2) \pi x))$ and $f_{0,k}=(3/2)k^{-3/2}\sin{k}$. This function has regularity $\beta=1$ in the sense of \eqref{def : Holder ball}. We generate noisy observations $Y=\{y_1,\dots, y_n\}$ according to the above random design regression model, for $n=400$ and $n=4000$. We consider the following priors:
\begin{enumerate}
\item[-] series priors as in \eqref{def : prior series}, \eqref{def : alphsig} defined over the same basis as the truth, $\{\varphi_k\}$, with $p$--exponential tails for $p=2, 1, 1/2, 1/4$ and $\alpha=2$;
\item[-] series priors defined over $\{\varphi_k\}$, with $p_k$--exponential tails varying with frequency $k$ (and $n$), either as \eqref{eq:seriesvarp-a} for $\alpha=2$ or as \eqref{eq:seriesvarp-ot} for $\gamma=1/2$;
\item[-] Cauchy HT$(\alpha)$ for $\alpha=2$ and OT for $\gamma=1/2$ series priors defined over $\{\varphi_k\}$. These are series priors as in \eqref{def : prior series} but with Cauchy-distributed $\zeta_k$, with $\sigma_k=k^{-1/2-\alpha}$ or $\sigma_k=\exp(-\log^{1+\gamma}{k})$, respectively. We use them as a benchmark as they have been shown to be partially adaptive for $\alpha\ge\beta$ and fully adaptive for any $\gamma>0$, respectively, \cite{AC, ace};
\item[-] shallow neural network priors defined as in \eqref{SNN-prior} with $\alpha=1/2$, for $p=2, 1, 1/2, 1/4$ and oracle choice of $\sigma_n$, see Theorem \ref{thm:SNN}(i);
\item[-] shallow neural network priors defined as in \eqref{SNN-prior} with $\alpha=1/2$, for $p=2, 1, 1/2, 1/4$ and $\sigma_n=\varepsilon_n^+/N_\alpha$ as in Theorem \ref{thm:SNN}(ii);
\item[-] shallow neural network priors defined as in \eqref{SNN-prior} with $\alpha=1/2$ and $p$ varying with $n$, $p_n=1/\log{n}$, $\sigma_n=n^{-7/5-0.01}$, and with $\alpha=0$ (width--$n$) for $p_n, \sigma_n$ as in \eqref{eq:varpotsnn}, with $\gamma=1/2$, $a=1/4$ and $b=4\log{8}+0.01$.
\end{enumerate}

Note that compared to the shallow neural network priors, the series priors are somewhat {\em favored} by the fact that the basis $\{\varphi_j\}$ on which they are defined coincides with the one on which the true function is defined, which is not the case for the ReLU prior. In particular, both the true function and draws from the studied series priors are restricted to take the value zero at the right edge ($x=1$). Similarly, shallow neural networks with the {\em oracle} choice of $\sigma_n$ use knowledge of the regularity $\beta$ of the truth $f_0$, which is unrealistic in practice. For the series priors, we chose $\alpha=2$ so that we are in the oversmoothing prior regime $\alpha>\beta$, in which the differences in performance depending on $p$ arise. For shallow neural network priors, we chose $\alpha=1/2$ so that we are in the overparameterized regime studied in Section \ref{sec : SNN}. Note that the first varying--$p$ shallow network prior used  (corresponding to $\alpha=1/2, p_n=1/\log{n}, \sigma_n=n^{-7.5-0.01}$) is not strictly covered by the results of Section \ref{sec : SNN}, however our theory suggests that it should perform well, at least asymptotically for large $n$ (see Remarks \ref{rem:SNN-varp} and \ref{rem:SNN-varptau} in the supplement). For the second such prior, $a$ and $b$ were chosen so that $p_n$ did not become too small for the sample sizes $n$ considered, while still satisfying the conditions under which our theory holds. Smaller values of $p_n$ necessitated a smaller step size in the posterior sampling algorithm used and hence more iterations (see the next paragraph for details). As this prior already incurred a higher computational cost due to its width $n$, compared with the width $\sqrt{n}$ of the other neural network priors considered, the resulting increase in the number of iterations would have made posterior sampling computationally expensive.  Furthermore, in Section \ref{sec : SNN}, for convenience we assumed $p$-exponential priors on the weights as in \eqref{def : p-exp dist}, while in our implementation we use $p$-exponential tailed  distributions in the broader sense of \eqref{conds}--\eqref{condu}. To summarize what we expect from our theory (and which is fully empirically confirmed by the results below), the two Cauchy priors and the varying $p$-exponential tail series and neural network priors are expected to be optimal, while for the rest of the priors,  smaller $p$ is better.

Regarding the implementation, we discretized the unit interval  using 200 uniformly spaced points, while we truncated the spectral series (for the truth or the series priors) up to $k=200$.
For sampling the posterior we employed the (whitened) preconditioned Crank-Nicholson (wpCN) algorithm, \cite{cdps18, crsw13}, which is a derivative-free Metropolis-Hastings algorithm robust with respect to dimension (truncation level). The vanilla version of pCN is suitable for Gaussian priors, while the whitened version is suitable for non-Gaussian priors admitting a transformation $f=T(\xi)$, where $\xi$ is a sequence of i.i.d. standard normal variables  and $f$ is a random draw from the non-Gaussian prior of interest, see Algorithm 2 in \cite{cdps18}. For example, to get series priors with $p$-exponential tails for $p>0$, we employed the transformation $T(\xi)=|\xi|^{2/p-1}\xi$ which transforms $\xi\sim N(0,1)$ to a random variable with $p$-exponential tails, while for the shallow neural network priors this transformation is composed with the map taking the sequence of weights and biases to function realizations as in \eqref{SNN-prior}. For all considered priors, we initialized the Markov chains using draws from the prior. In all runs, we tuned the proposal step size (i.e., the scaling of the proposals) to achieve an acceptance rate of approximately 30\%. For all priors, the wpCN algorithm is less efficient compared to the Gaussian series prior case, in the sense that smaller step sizes are required to maintain this target acceptance rate. To keep the total distance explored by the sampler roughly constant across experiments, we scaled the total number of iterations inversely proportional to the step size. In all runs, we retained 20,000 samples. Depending on the prior used, a thinning factor proportional to the scaling applied to the total number of iterations was used.
 
To compare the performance of the considered priors quantitatively, we averaged errors over 100 realizations of the data, for each of $n=400$ and $n=4000$. Figure \ref{fig-data} shows one such realization for each choice of $n$. In particular, we employed two types of errors. The first one is the $L_{2}$--error of the posterior means, hence after averaging we estimate the error $E_{f_0}||\hat{f}-f_0||_{2}$, 
for $\hat{f}$ the posterior mean. The second error  estimates
\[E_{f_0}E_{\Pi[\cdot|Y]}||f-f_0||_{2},\]
where the inner expectation is estimated by taking the average of the $L_{2}$--errors of the (thinned) Markov chain samples after burn-in, and the outer by averaging over the 100 data realizations. The latter error captures the contraction of the whole posterior around $f_0$. 

The computed average errors are presented in Table \ref{tab:prior_errors}. We also computed standard deviations which were of lower order compared to the averages, hence for ease of readability we did not include them. The errors appear to be consistent with our theory. Overall, the best performers are the priors with heavier tails (either small $p$ or Cauchy) and the (diminishing) varying tail priors. Shallow neural network priors with the more practically relevant choice of $\sigma_n=\varepsilon_n^+/N_\alpha$, for small $p$ have similar performance to the corresponding shallow neural network priors with oracle $\sigma_n$. %\ma{DELETE since no longer true for the latest prior: In general, non-oracle shallow network priors perform slightly worse than the corresponding series priors, particularly for small or varying $p$. This may be due to not using the information that the unknown function vanishes at $x=1$.}

Figures \ref{fig-post400} and \ref{fig-post4000} show posterior means and 95\% credible intervals for the various priors, and for $n=400$ and $n=4000$, respectively, for the data realizations shown in Figure \ref{fig-data}. The conclusions are aligned to the ones in the previous paragraph.

\begin{figure}
    \centering
     \includegraphics[width=0.37\textwidth]{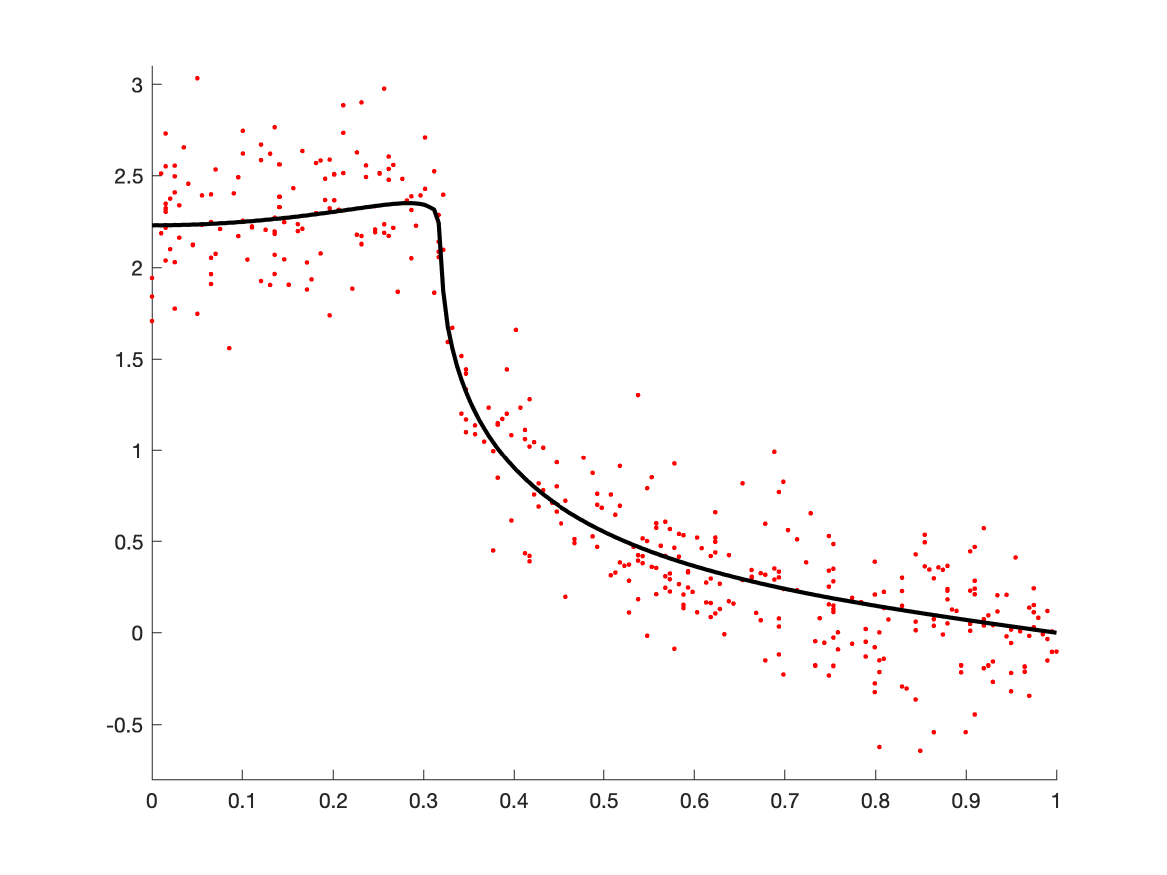}
    \includegraphics[width=0.37\textwidth]{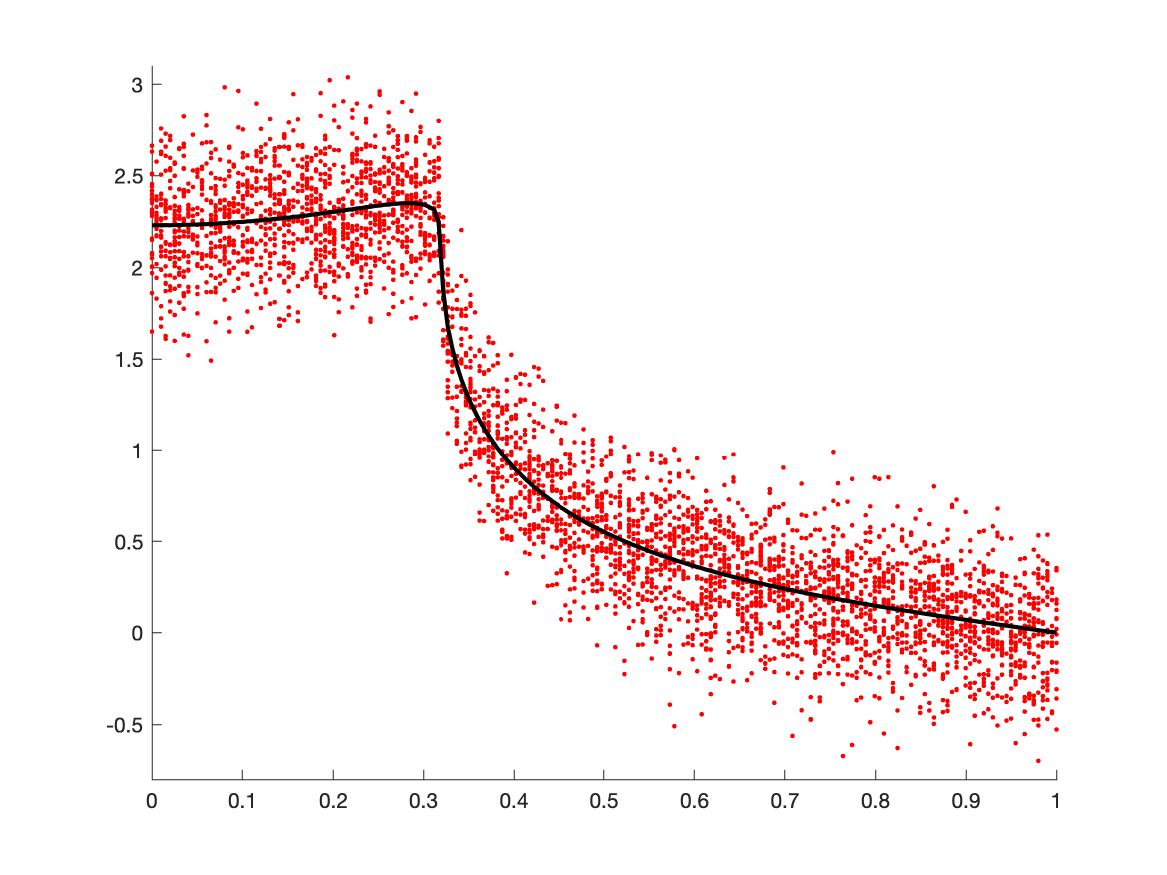}
    \caption{Observed random design regression data, with $n=400$ (left) and $n=4000$ (right). True function $f_0$ in black solid, red points noisy observations according to model \eqref{rdregression}.}
    \label{fig-data}
\end{figure}

\begin{figure}
    \centering
     \includegraphics[width=0.24\textwidth]{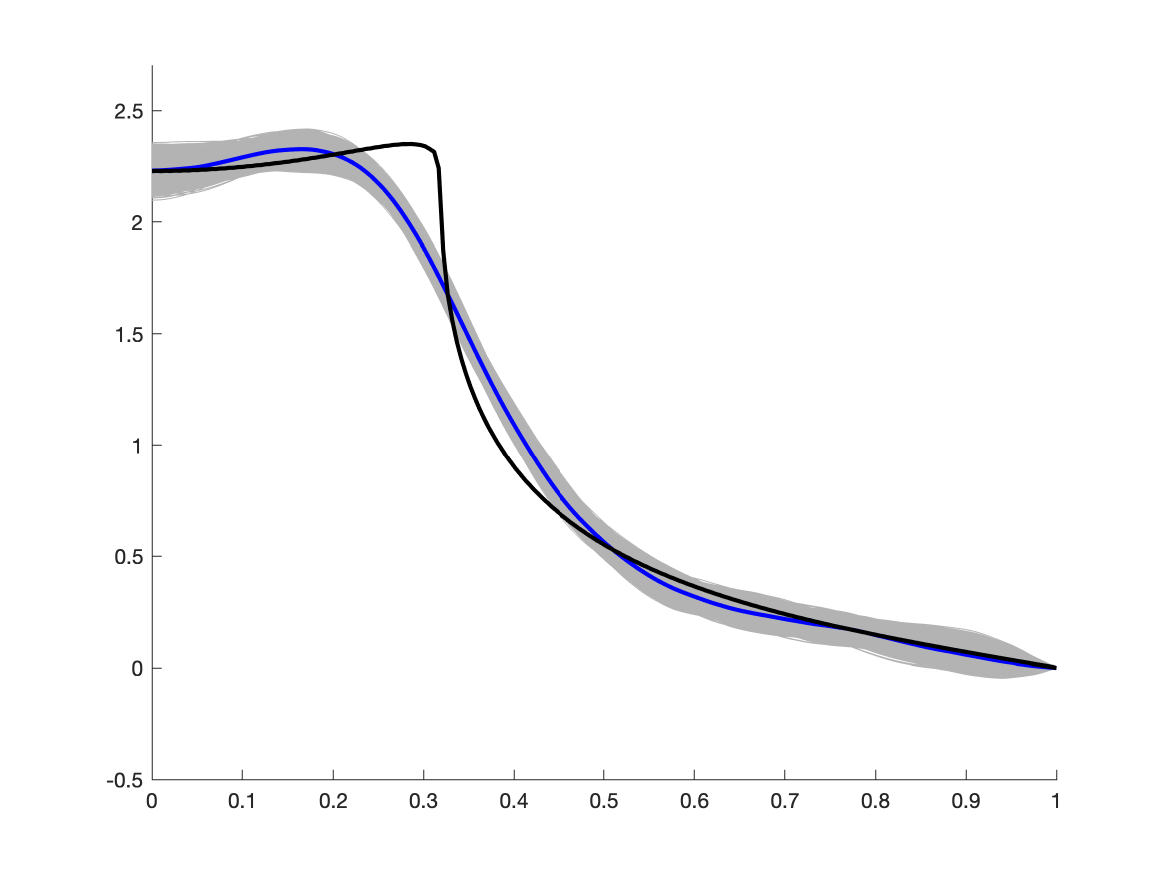}
 \includegraphics[width=0.24\textwidth]{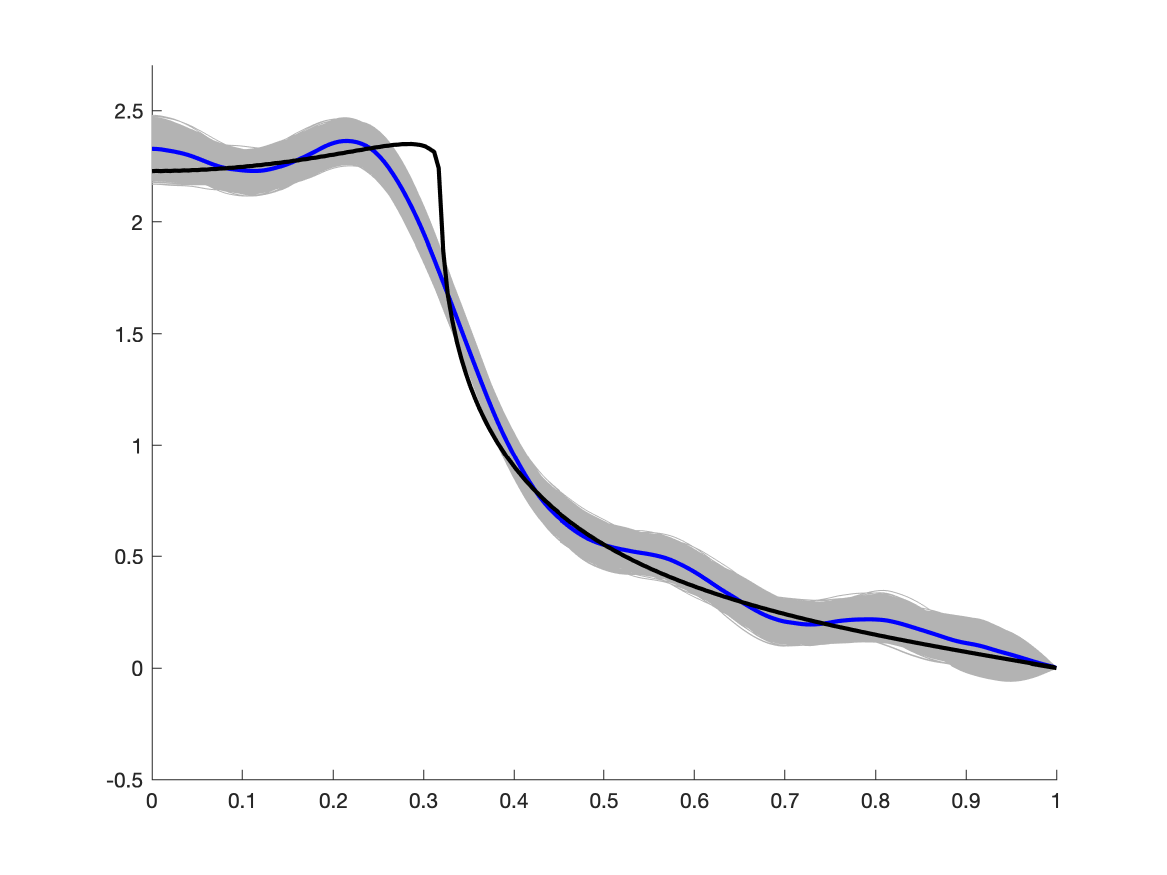} 
 \includegraphics[width=0.24\textwidth]{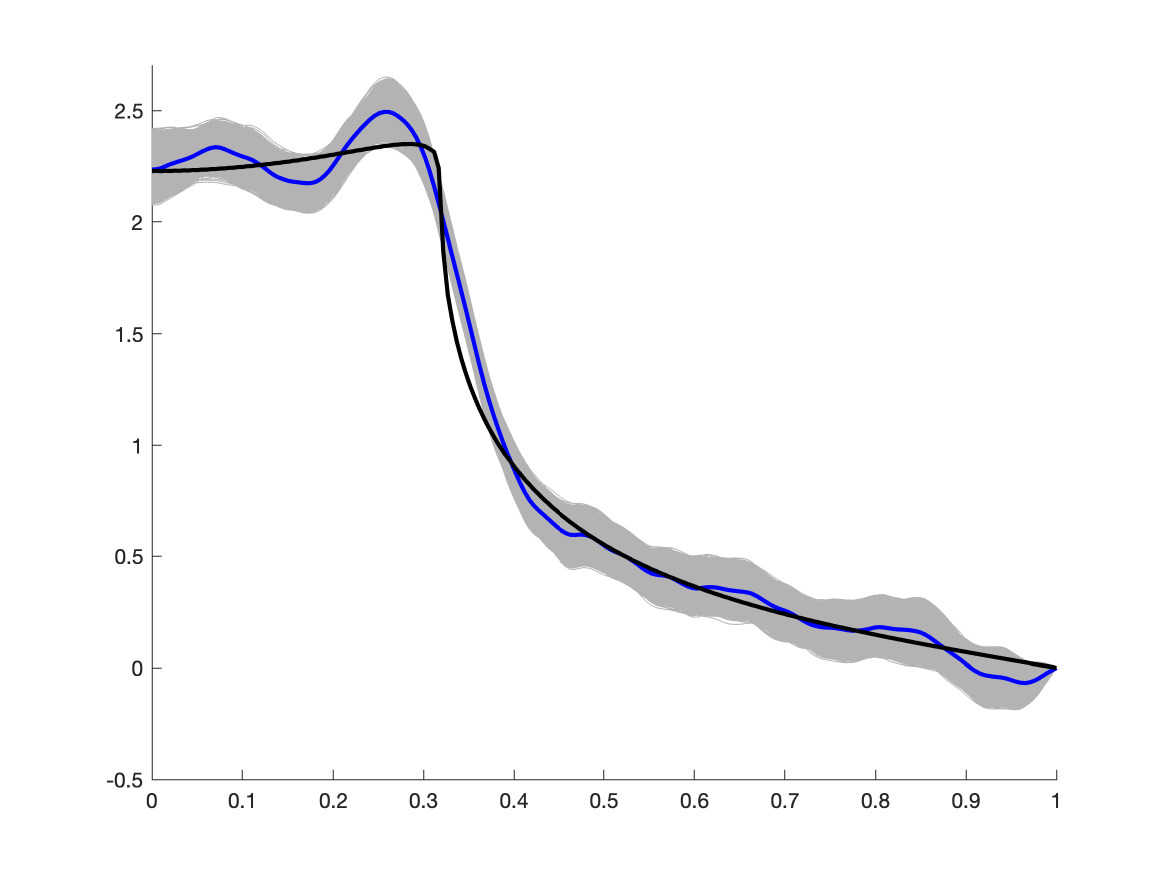} 
 \includegraphics[width=0.24\textwidth]{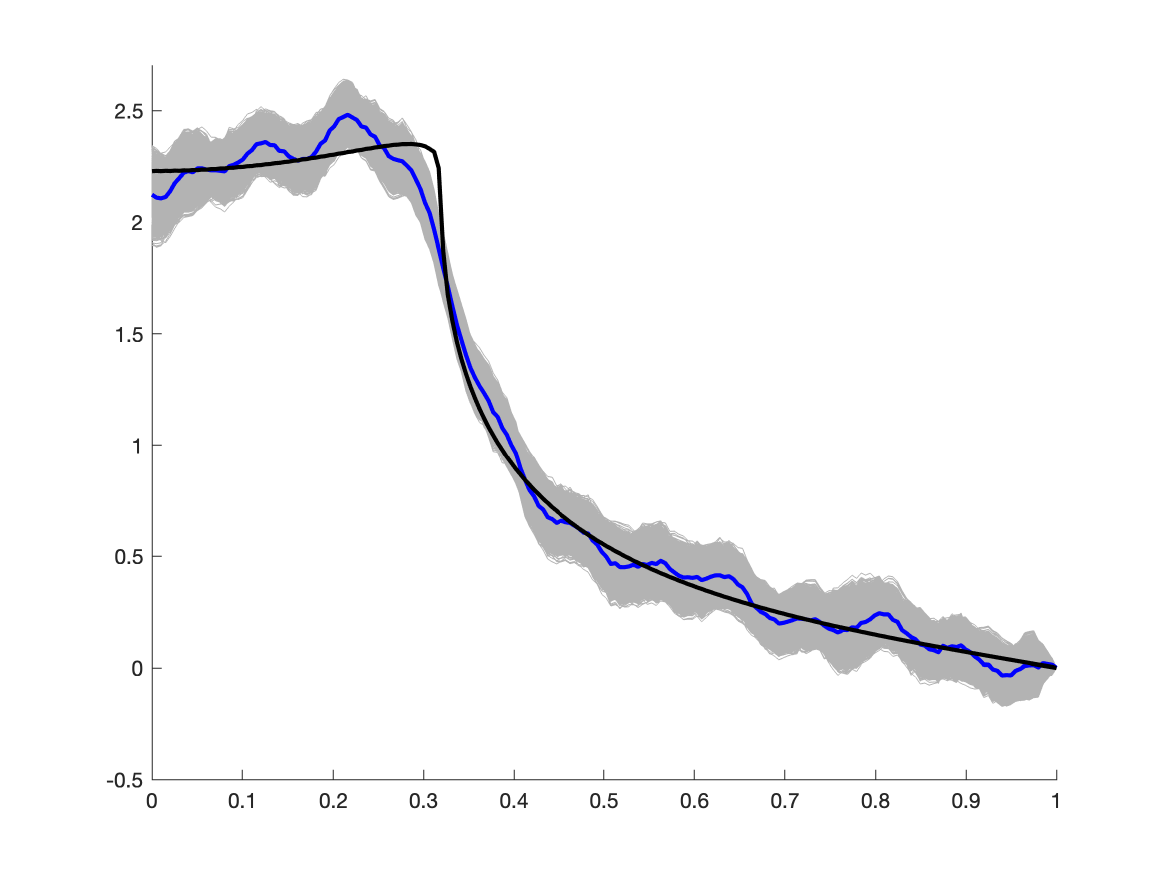} \\
 \includegraphics[width=0.24\textwidth]{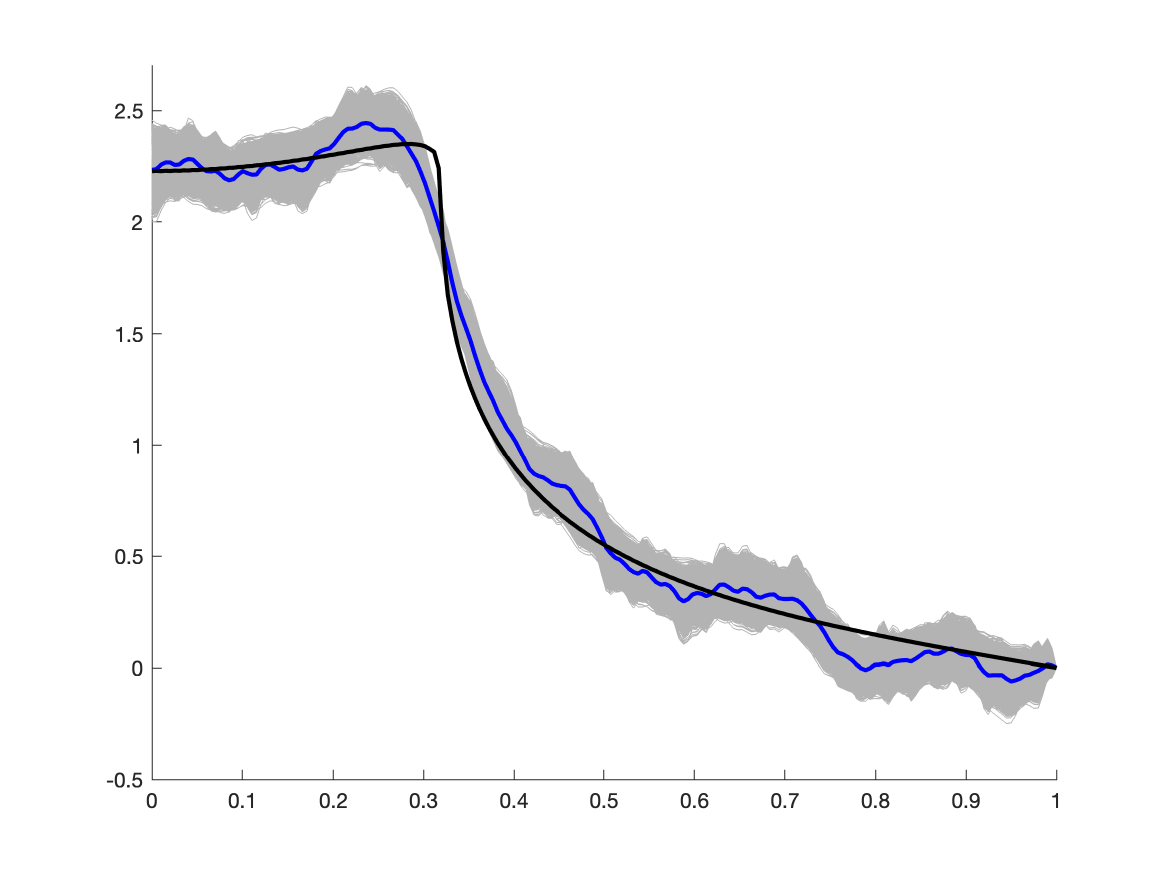} 
 \includegraphics[width=0.24\textwidth]{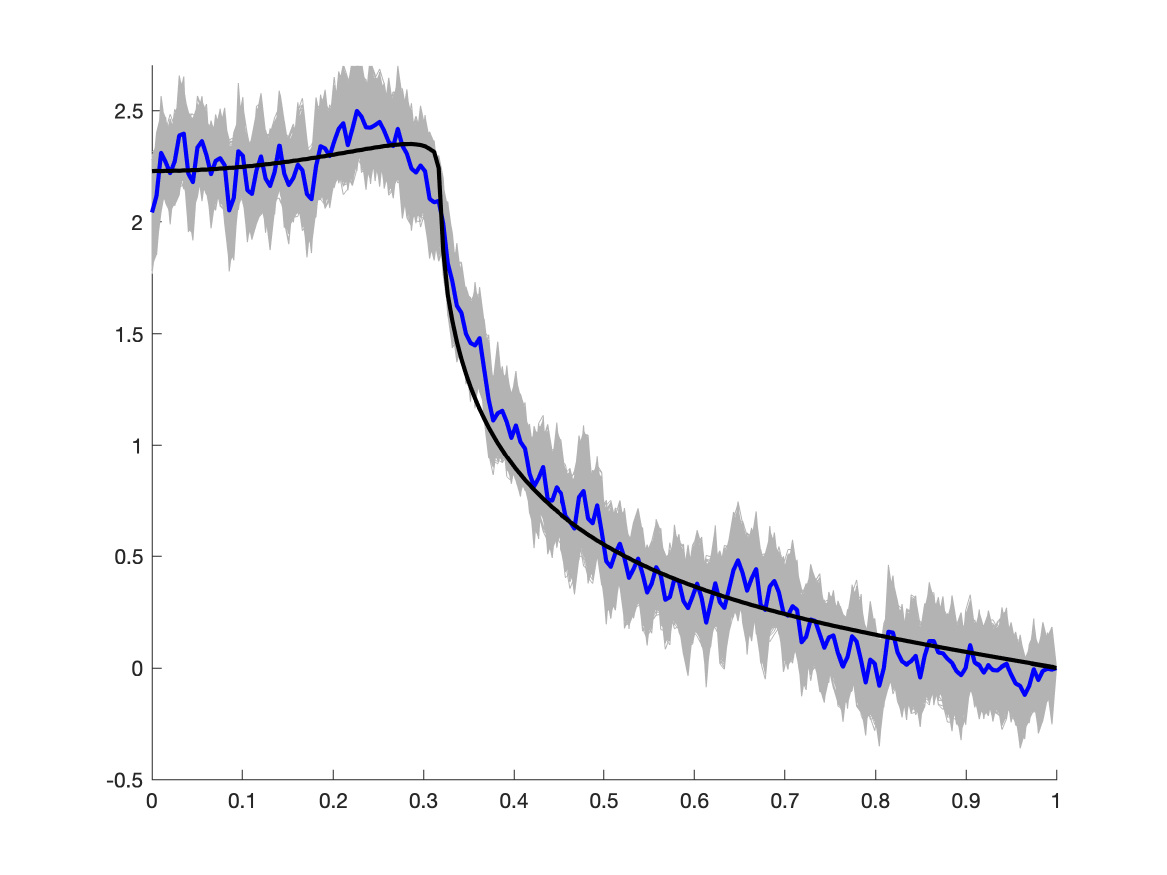} %\includegraphics[width=0.97\textwidth]{./plots/varp_gamma1_step0.0075_n400} 
 \includegraphics[width=0.24\textwidth]{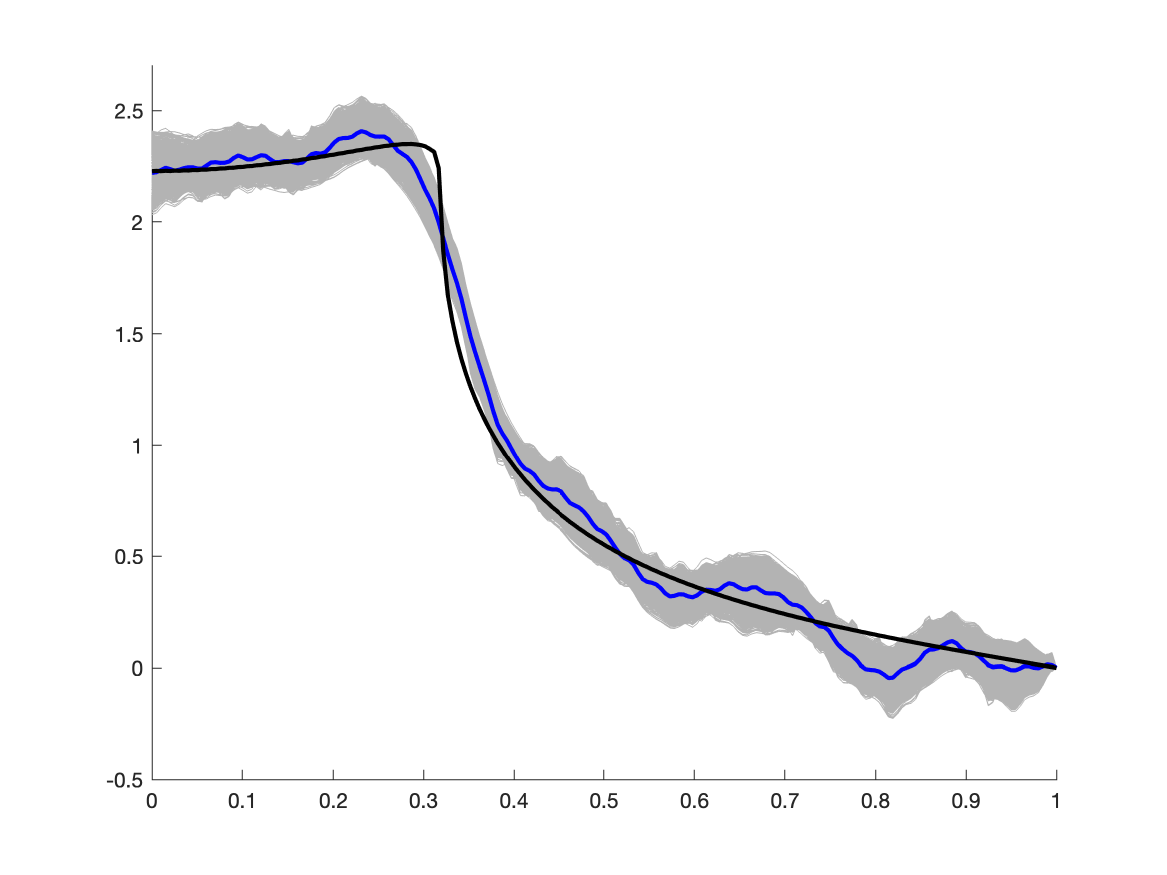}
  \includegraphics[width=0.24\textwidth]{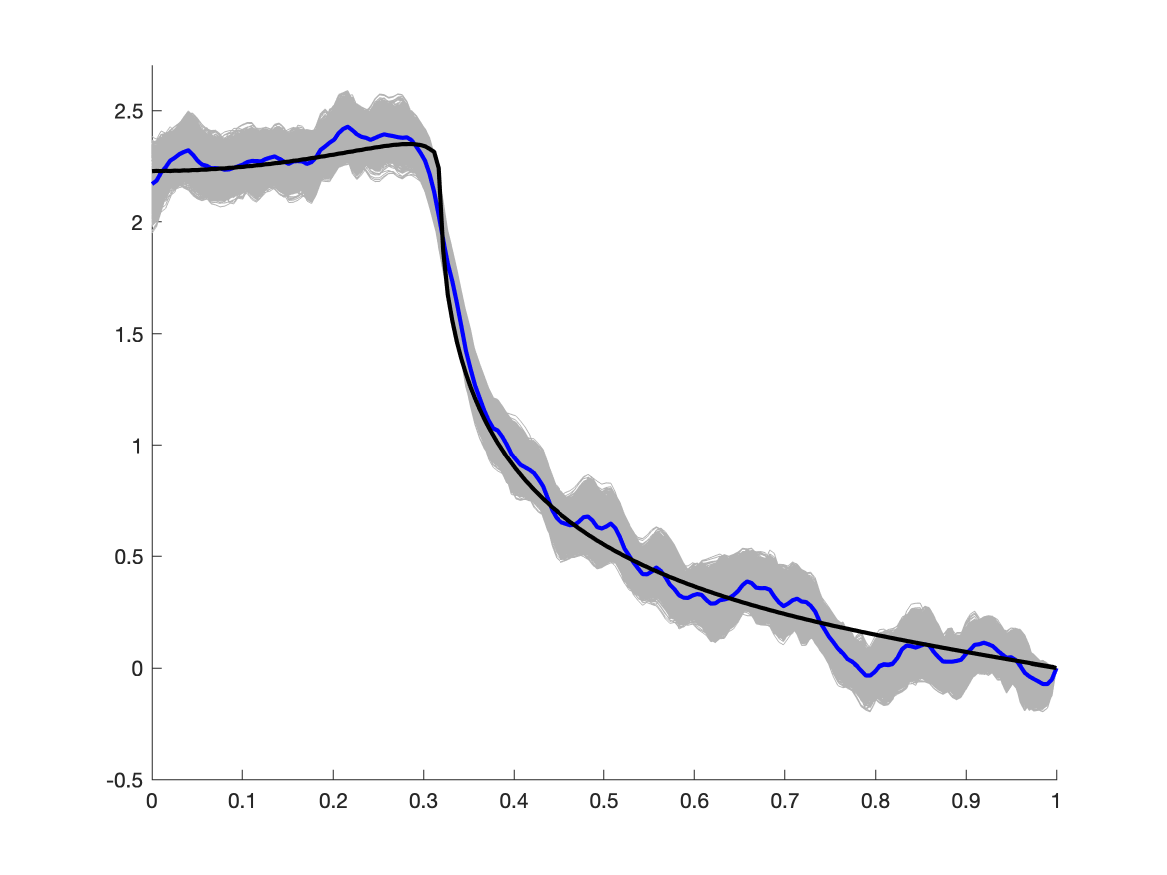}\\
\includegraphics[width=0.24\textwidth]{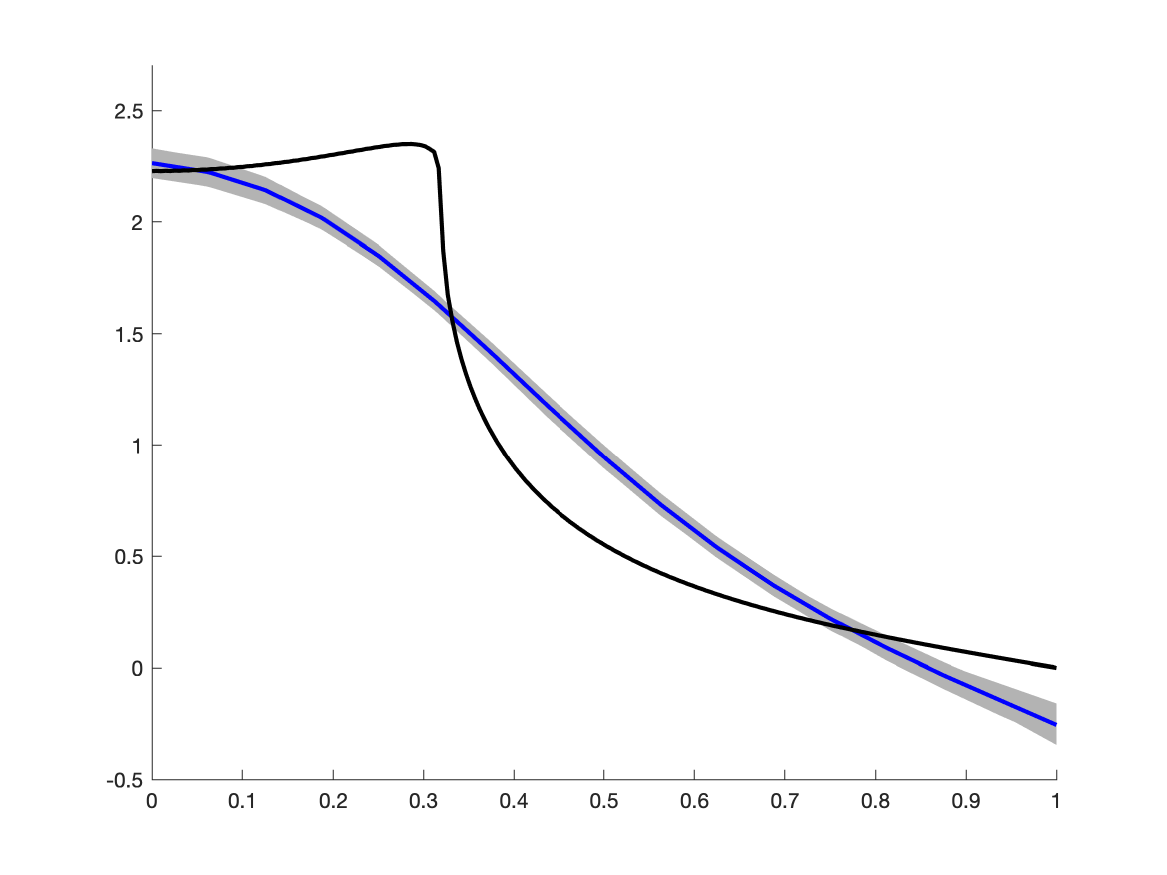}
 \includegraphics[width=0.24\textwidth]{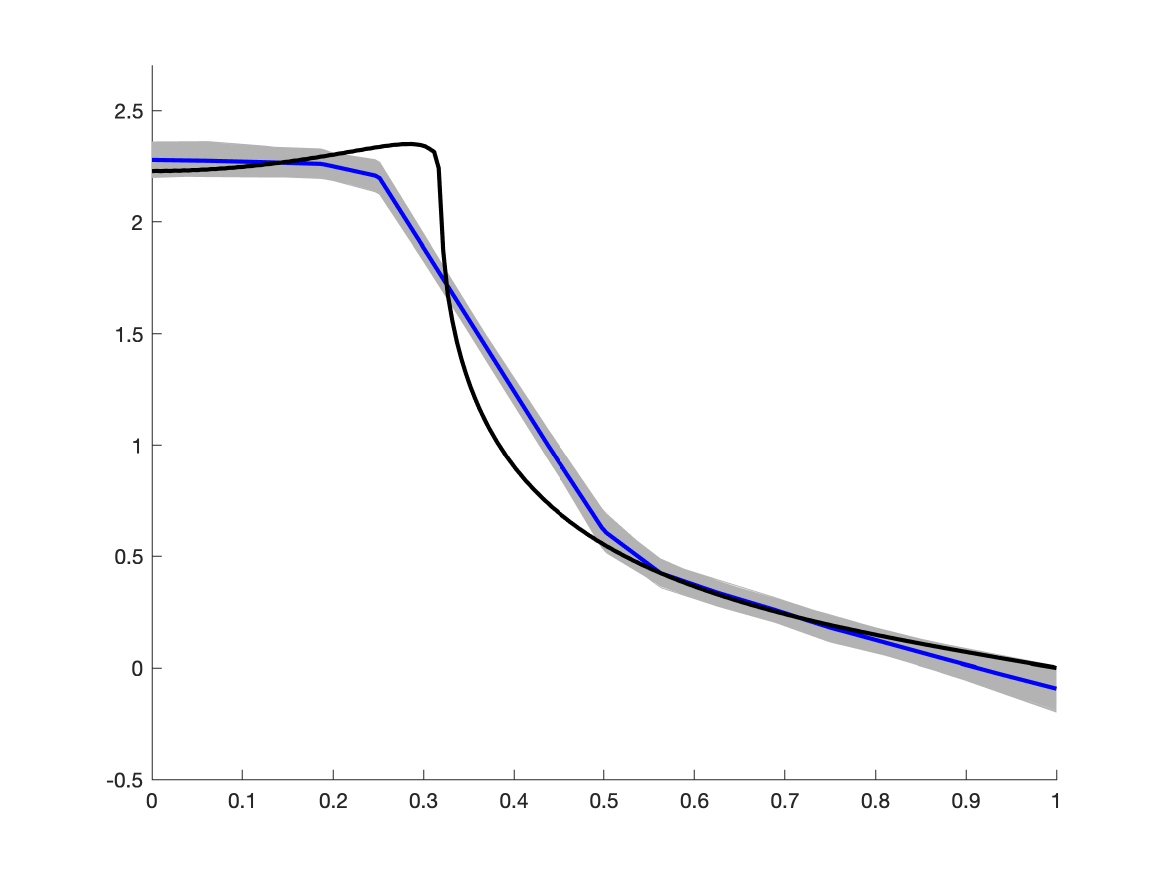} 
 \includegraphics[width=0.24\textwidth]{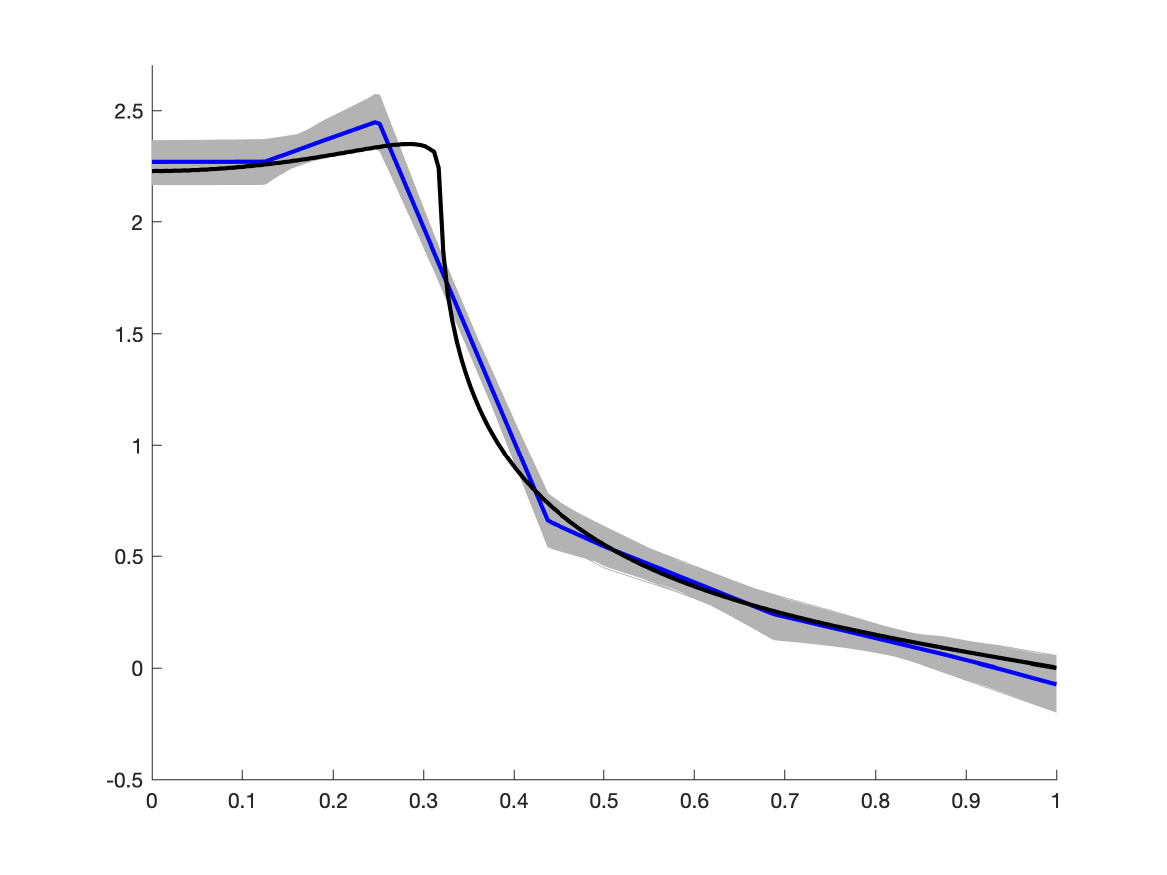} 
 \includegraphics[width=0.24\textwidth]{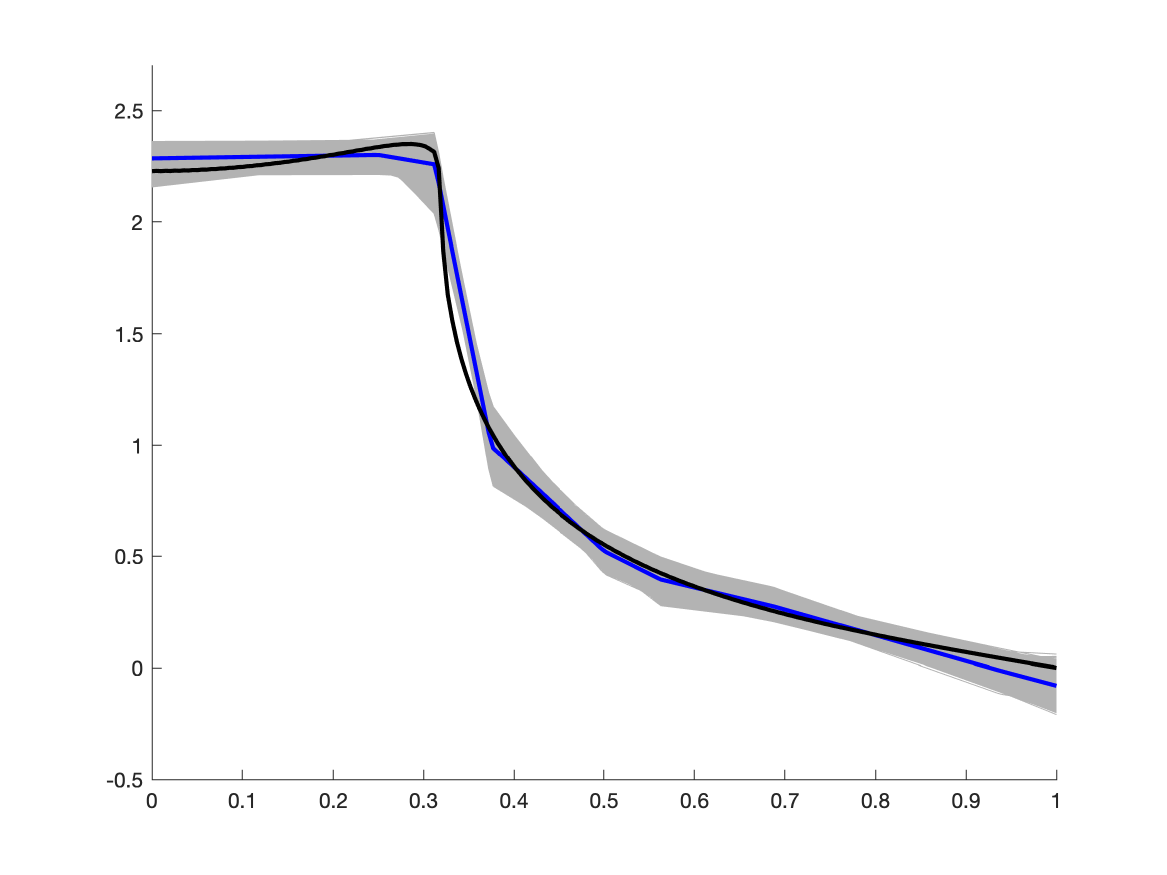}\\
 \includegraphics[width=0.24\textwidth]{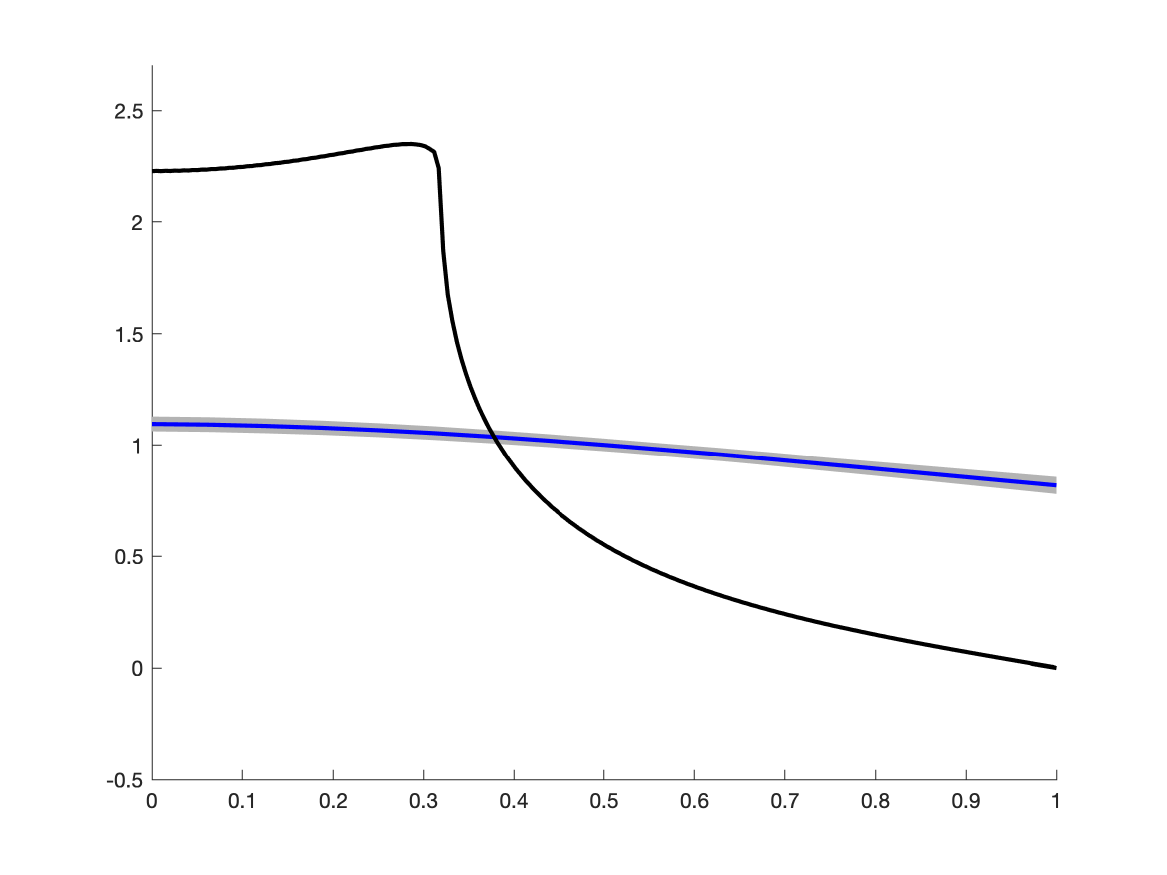}
 \includegraphics[width=0.24\textwidth]{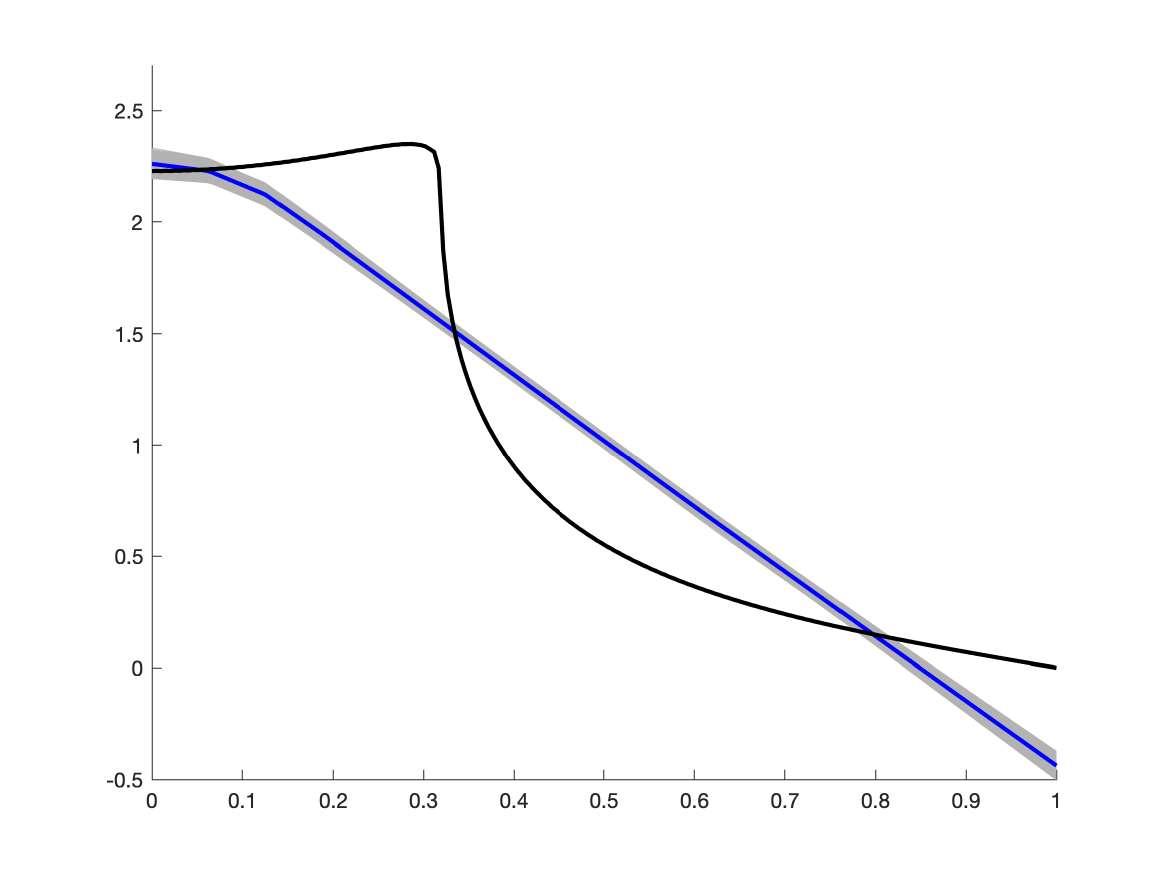} 
 \includegraphics[width=0.24\textwidth]{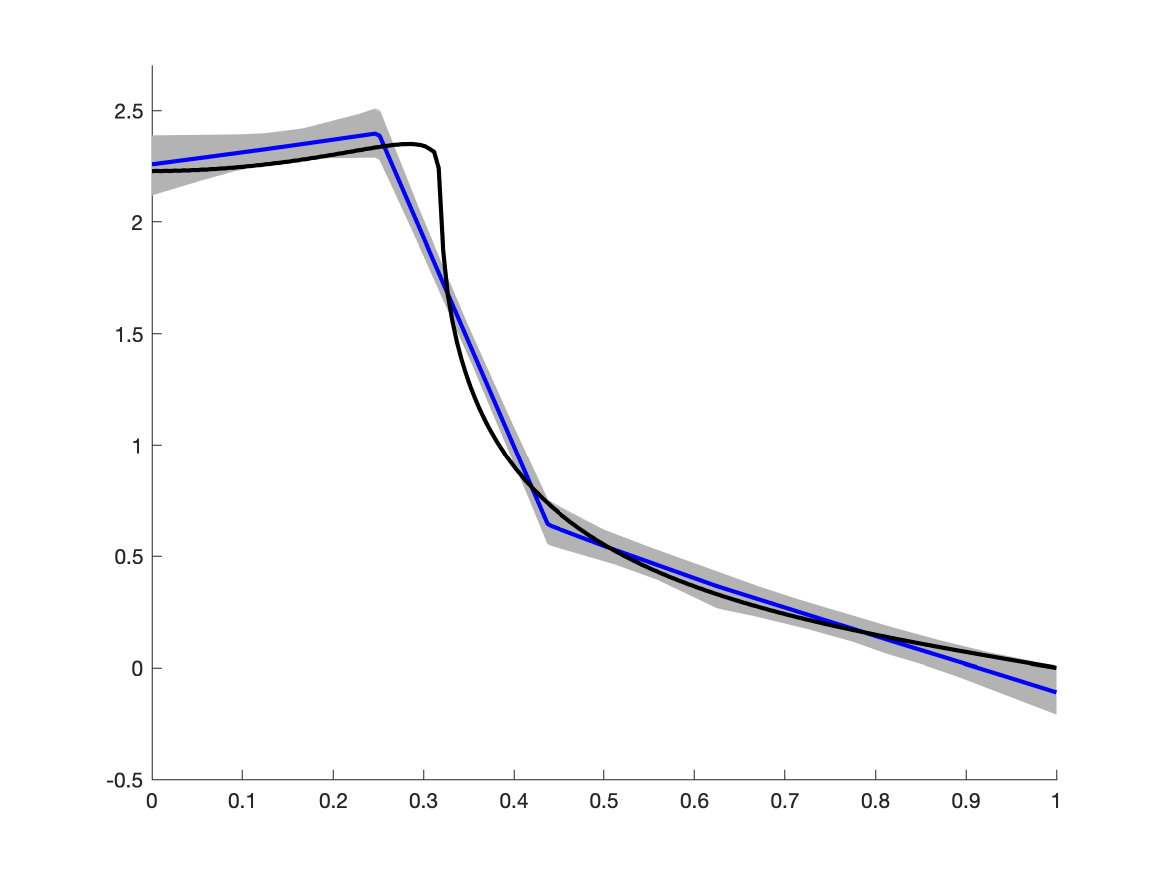} 
 \includegraphics[width=0.24\textwidth]{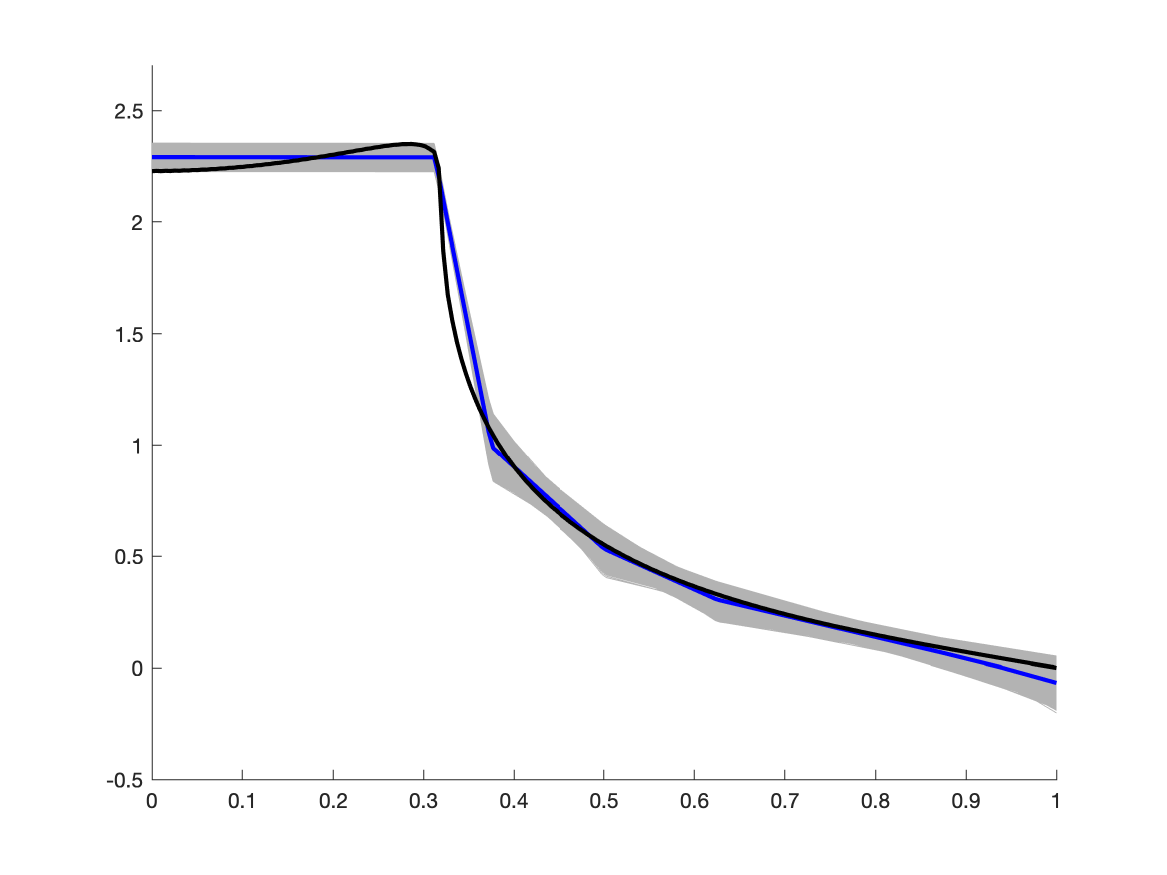}\\
  \includegraphics[width=0.24\textwidth]{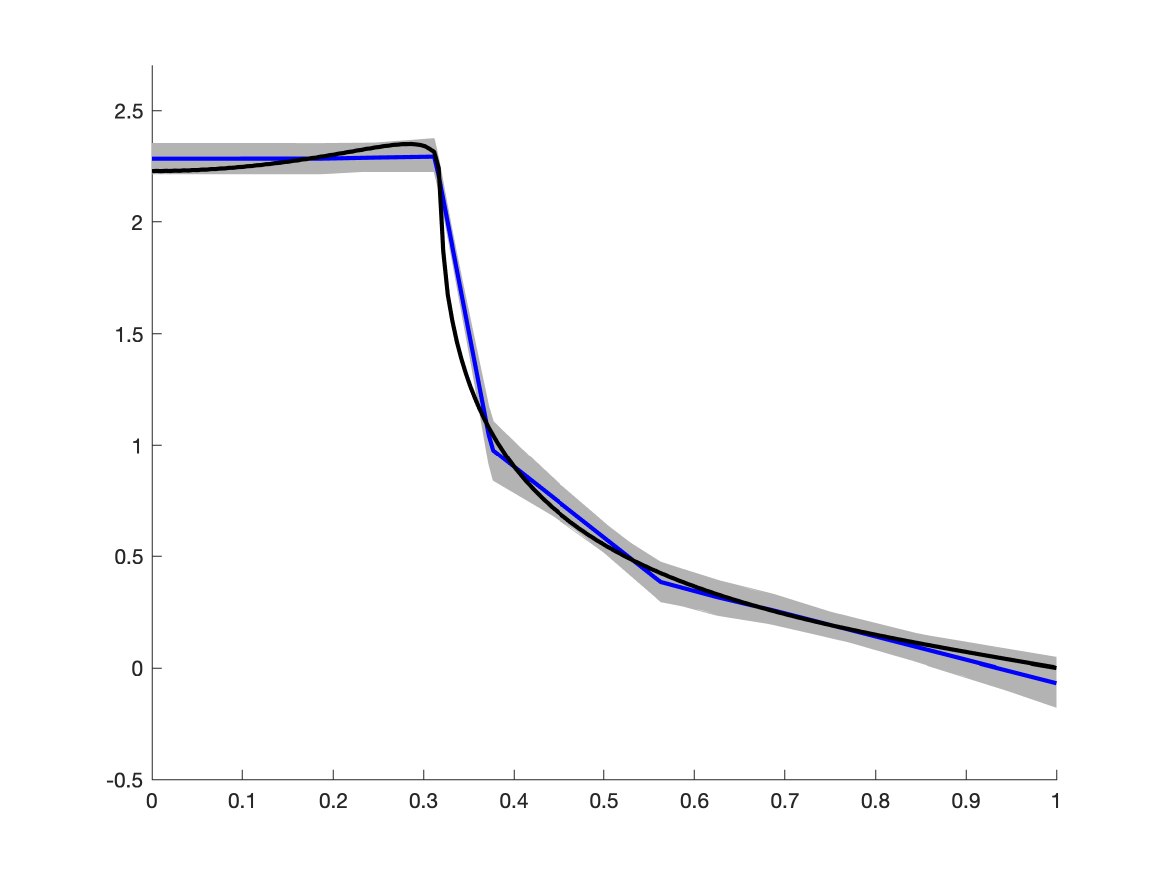}
    \includegraphics[width=0.24\textwidth]{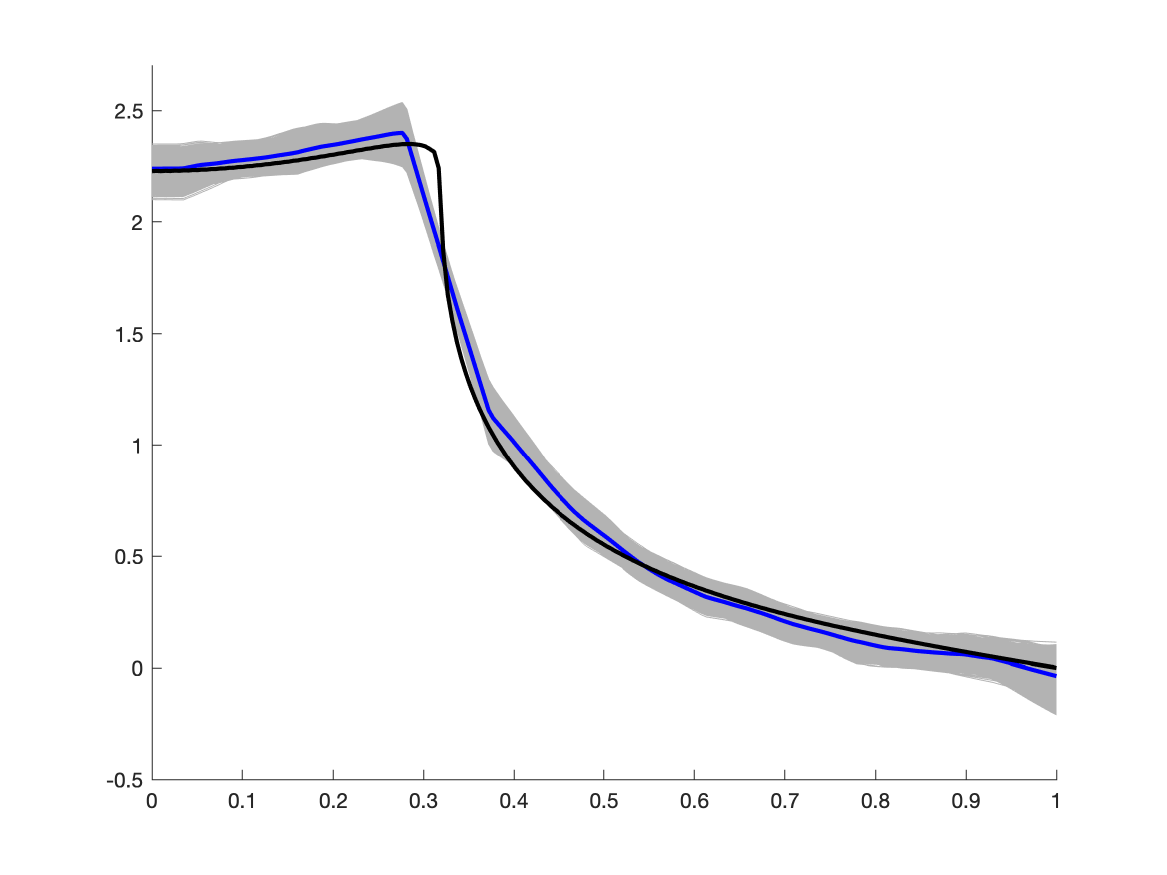}
  
\caption{Random design regression: true function (black), posterior mean (blue), 95\% credible regions (grey), for $n=400$. Top row: $p$-exponential series priors with $\alpha=2$ and $p=2, 1, 1/2,1/4$ left to right. Second row: series priors with varying $p$-tails as in \eqref{eq:seriesvarp-a} with $\alpha=2$ and as in \eqref{eq:seriesvarp-ot} with $\gamma=1/2$, Cauchy HT($\alpha$) with $\alpha=2$ and Cauchy OT with $\gamma=1/2$, left to right. Third row: shallow neural network priors as in \eqref{SNN-prior} with $\alpha=1/2$, oracle choice of $\sigma_n$ and $p=2,1,1/2,1/4$ left to right. Fourth row: shallow neural network priors as in \eqref{SNN-prior} with $\alpha=1/2$, $\sigma_n=\veps_n^+/N_\alpha$ and $p=2,1,1/2,1/4$ left to right. Bottom row: shallow network priors as in \eqref{SNN-prior} with $\alpha=1/2$, $\sigma=n^{-7/5-0.01}$,  $p_n=1/\log{n}$ and with $\alpha=0$, $\sigma_n=\exp(-(\log{n})^{3/2}/4)$, $p_n=(4\log{8}+0.01)/(\log{n})^{3/2}$.}
    \label{fig-post400}
\end{figure}

\begin{figure}[h]
    \centering
     \includegraphics[width=0.244\textwidth]{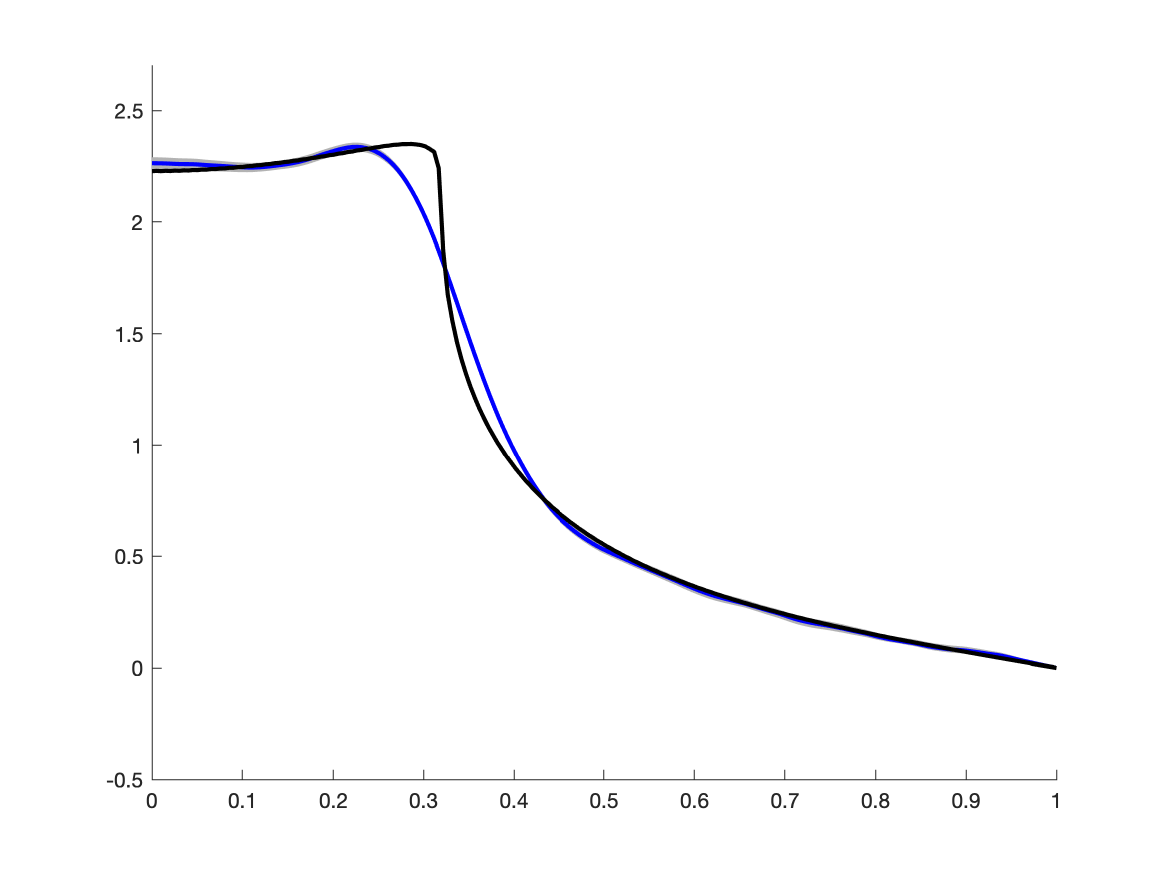}
 \includegraphics[width=0.24\textwidth]{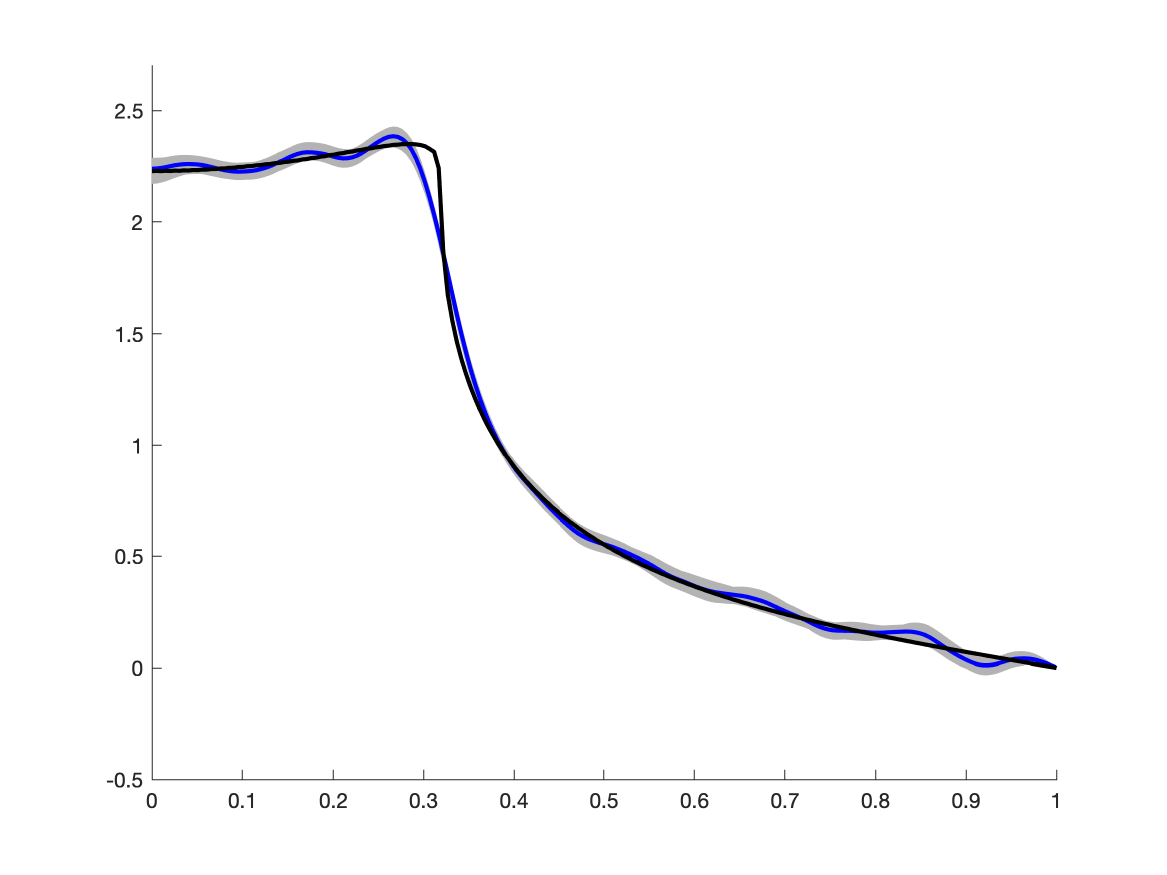} 
 \includegraphics[width=0.24\textwidth]{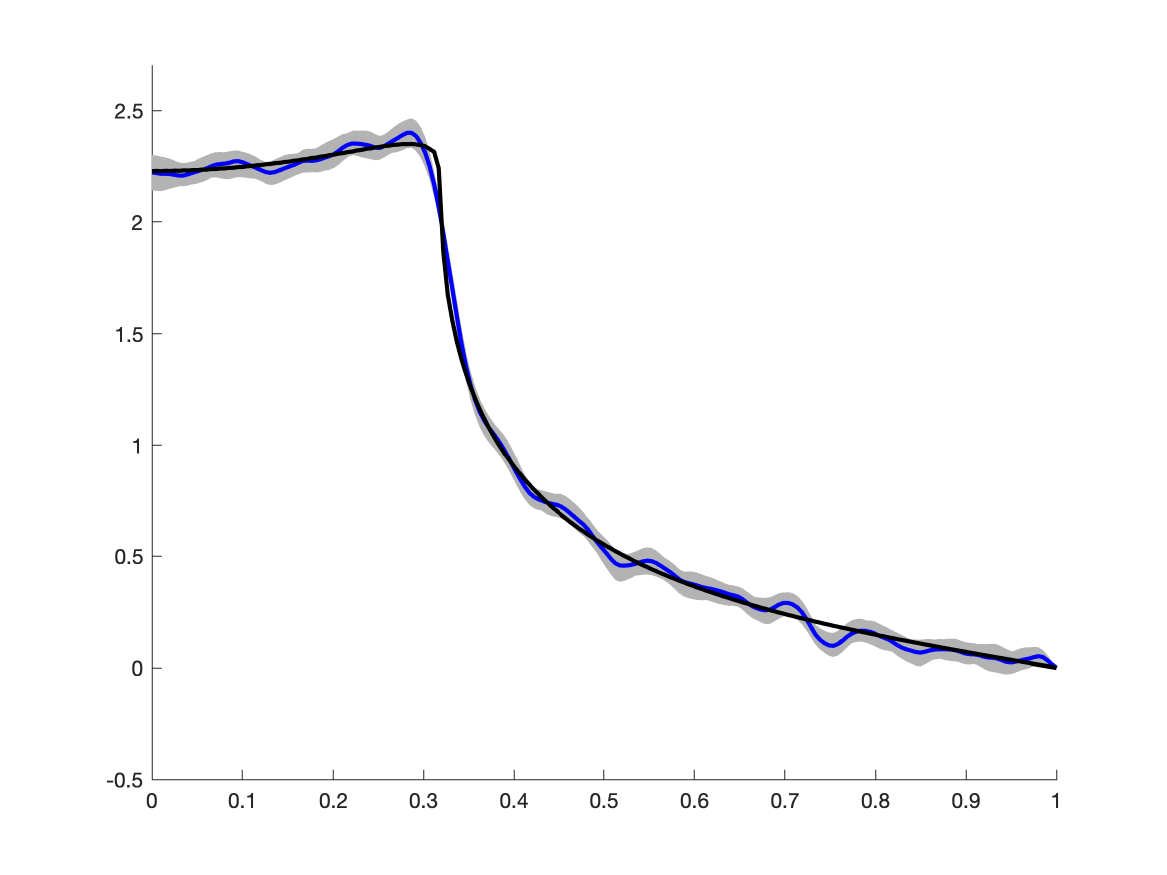} 
 \includegraphics[width=0.24\textwidth]{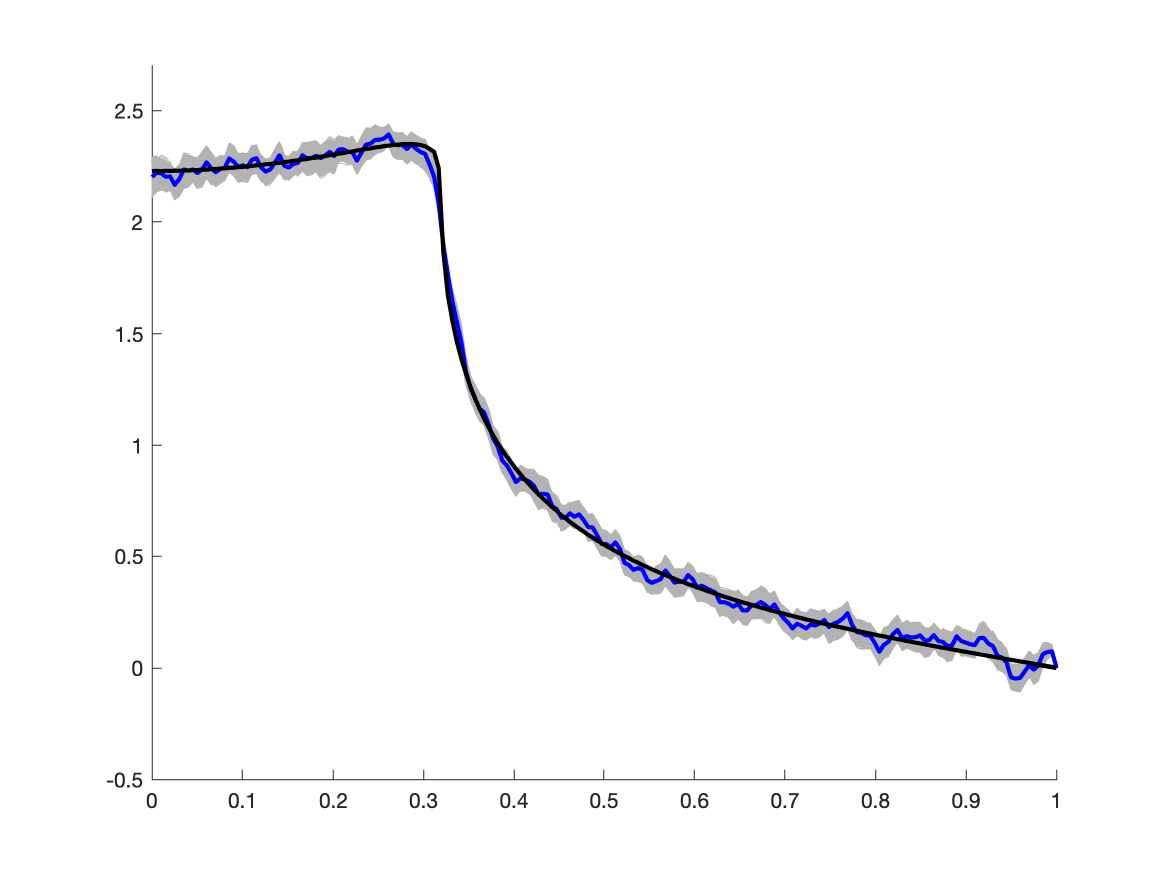} \\
 \includegraphics[width=0.24\textwidth]{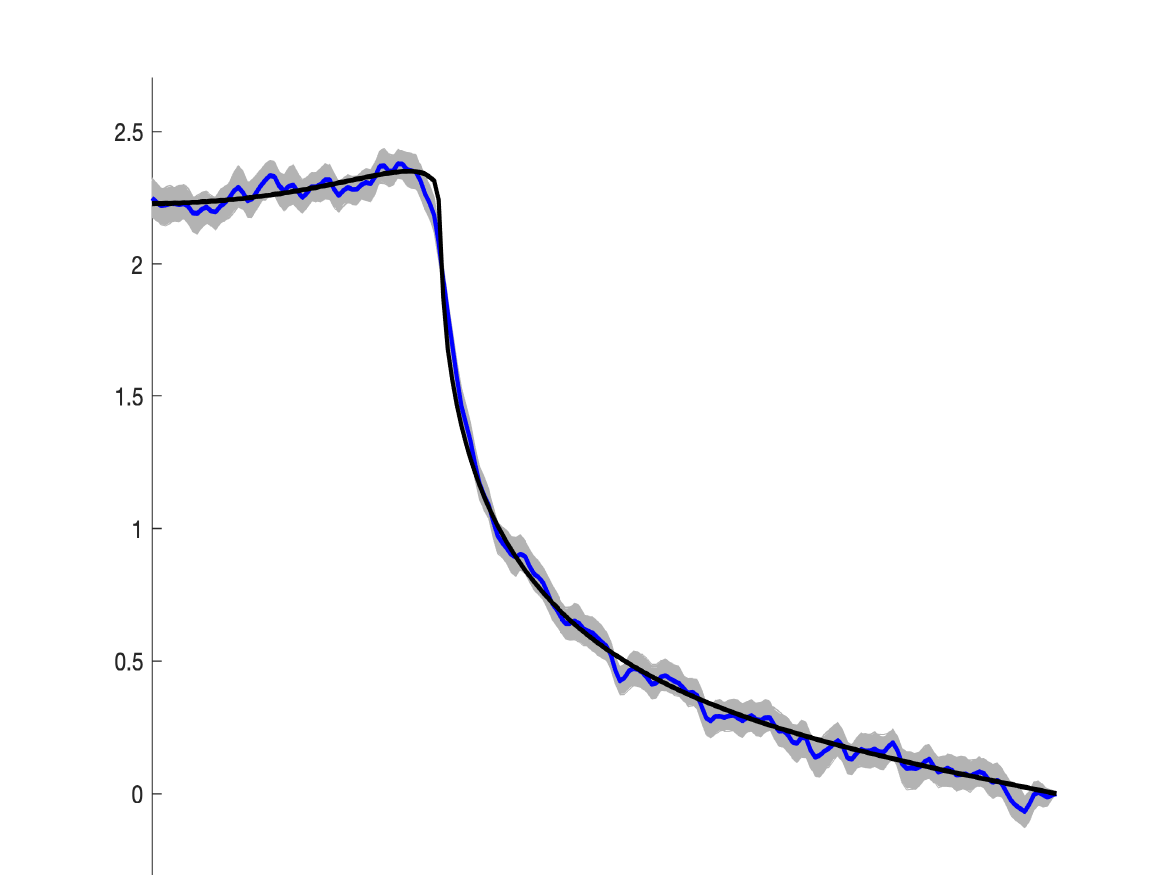} 
 \includegraphics[width=0.24\textwidth]{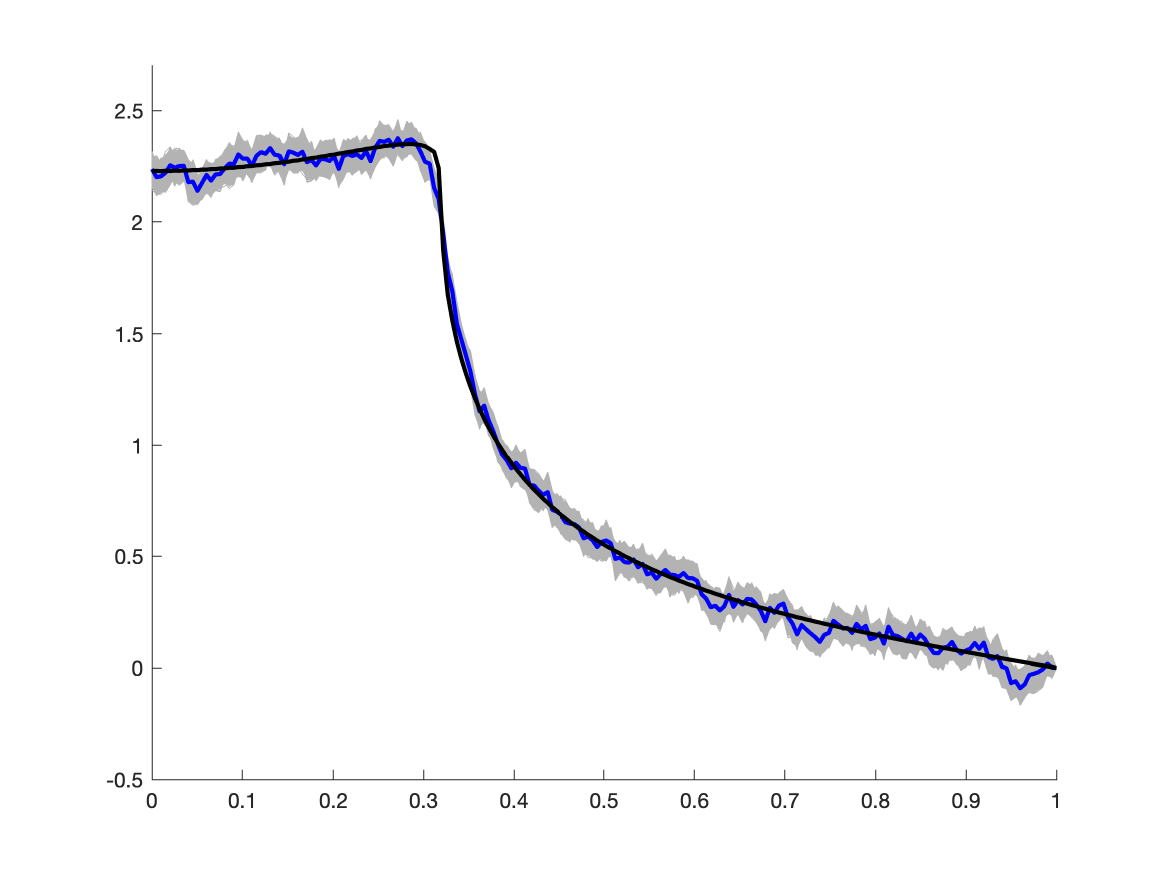} 
 \includegraphics[width=0.24\textwidth]{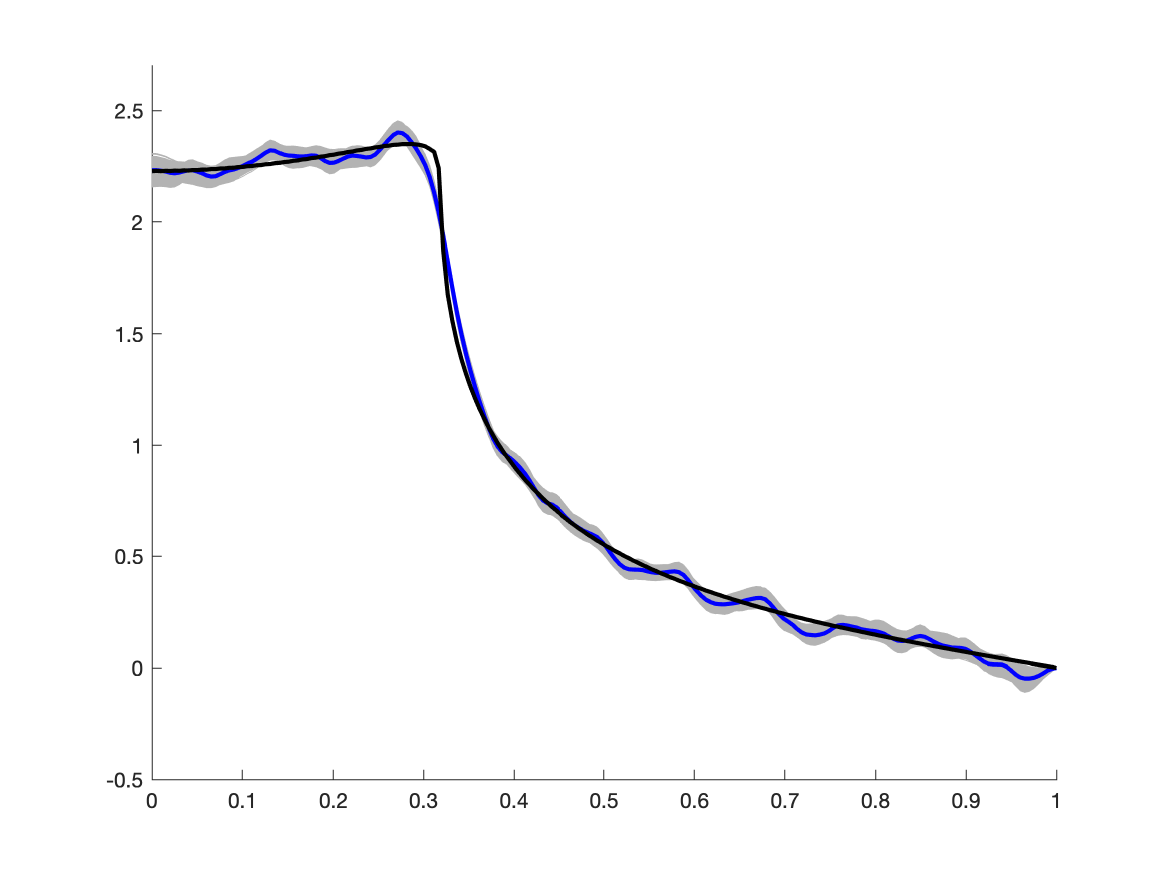}
  \includegraphics[width=0.24\textwidth]{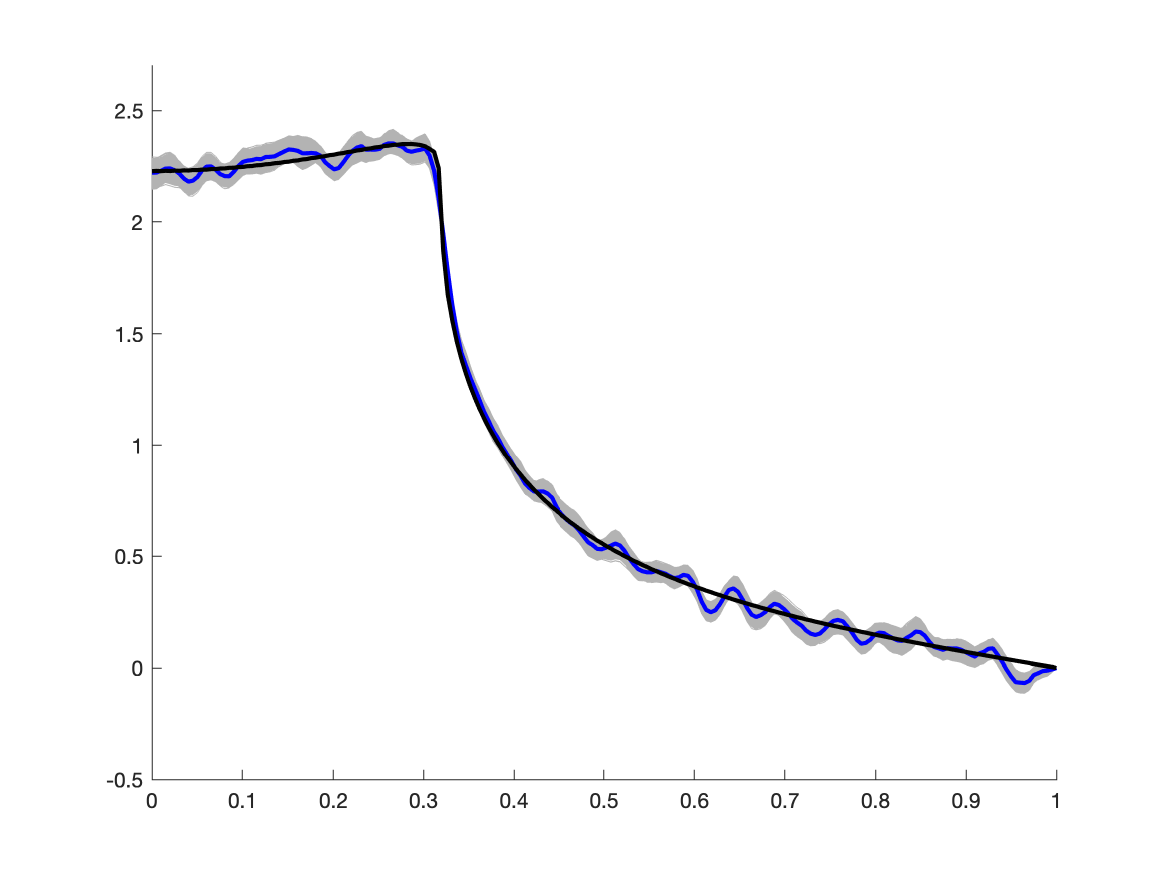}\\
\includegraphics[width=0.24\textwidth]{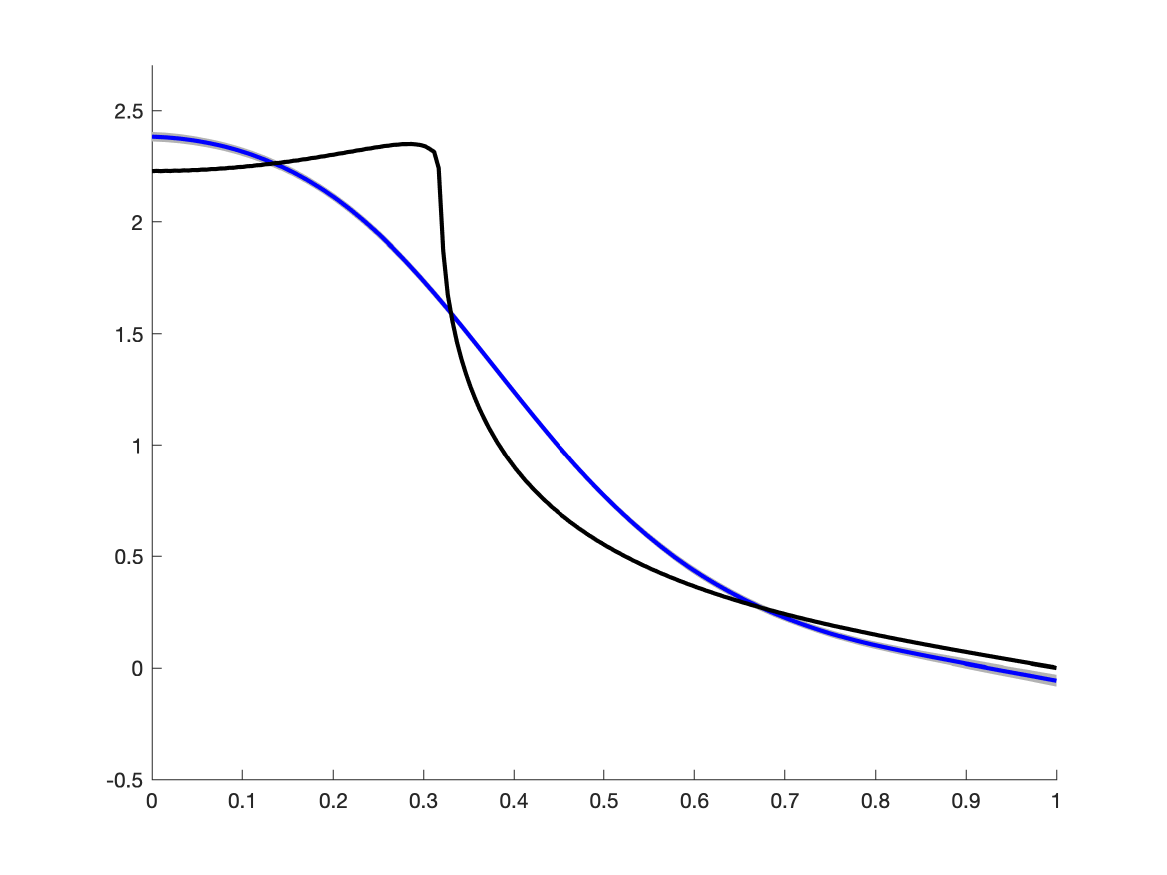}
 \includegraphics[width=0.24\textwidth]{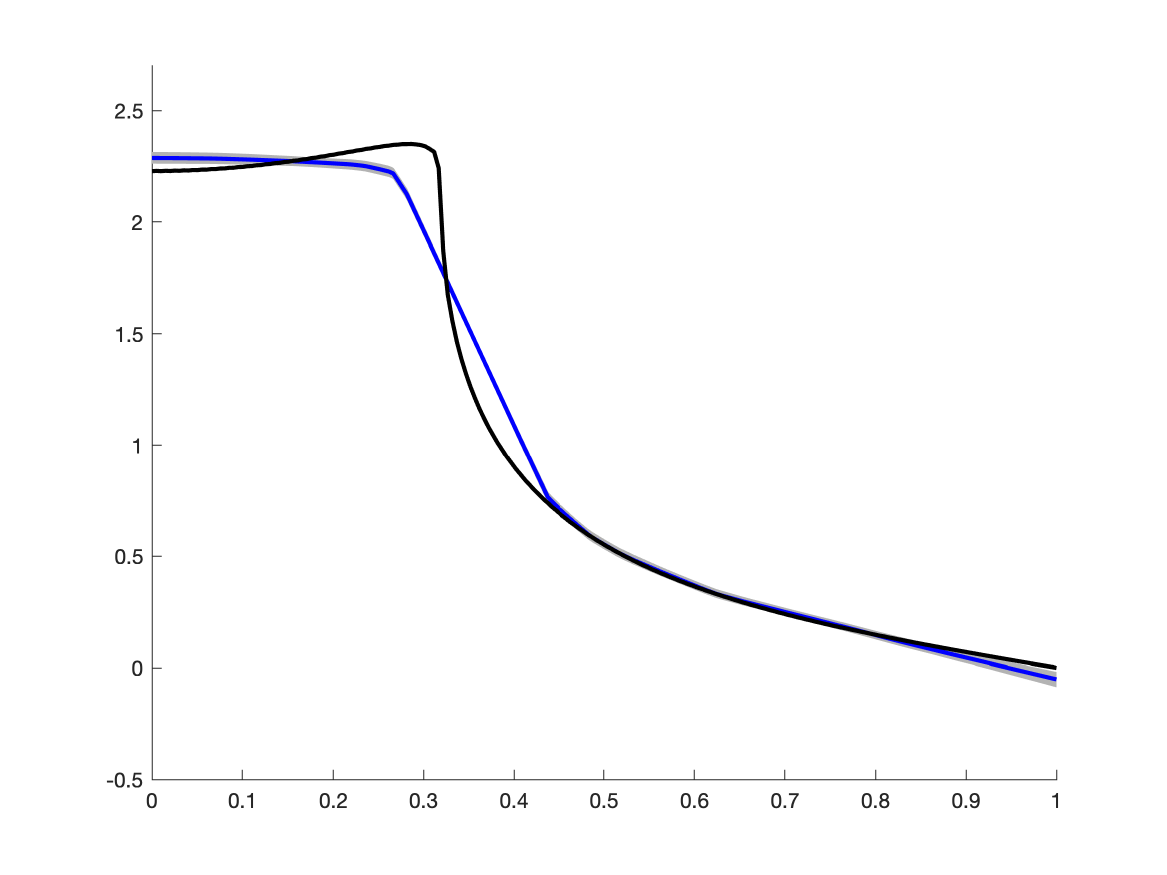} 
 \includegraphics[width=0.24\textwidth]{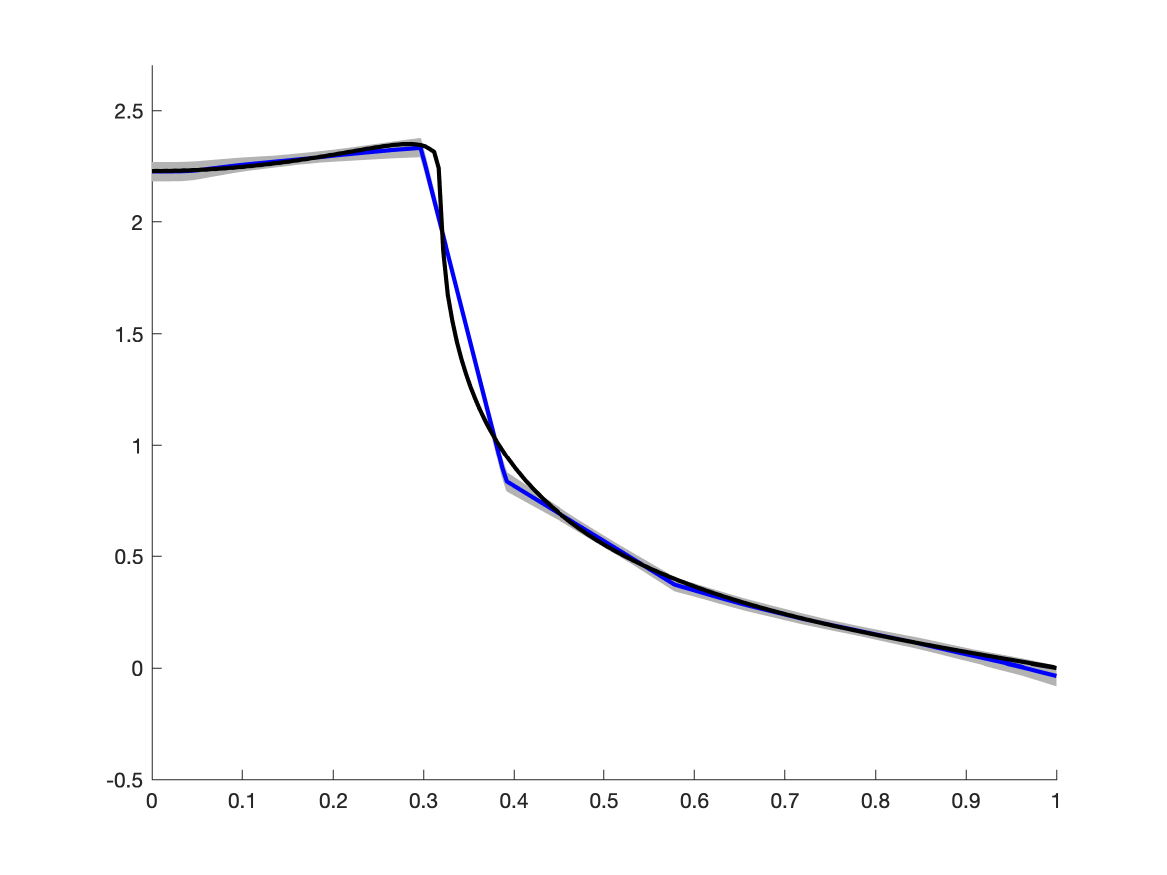} 
 \includegraphics[width=0.24\textwidth]{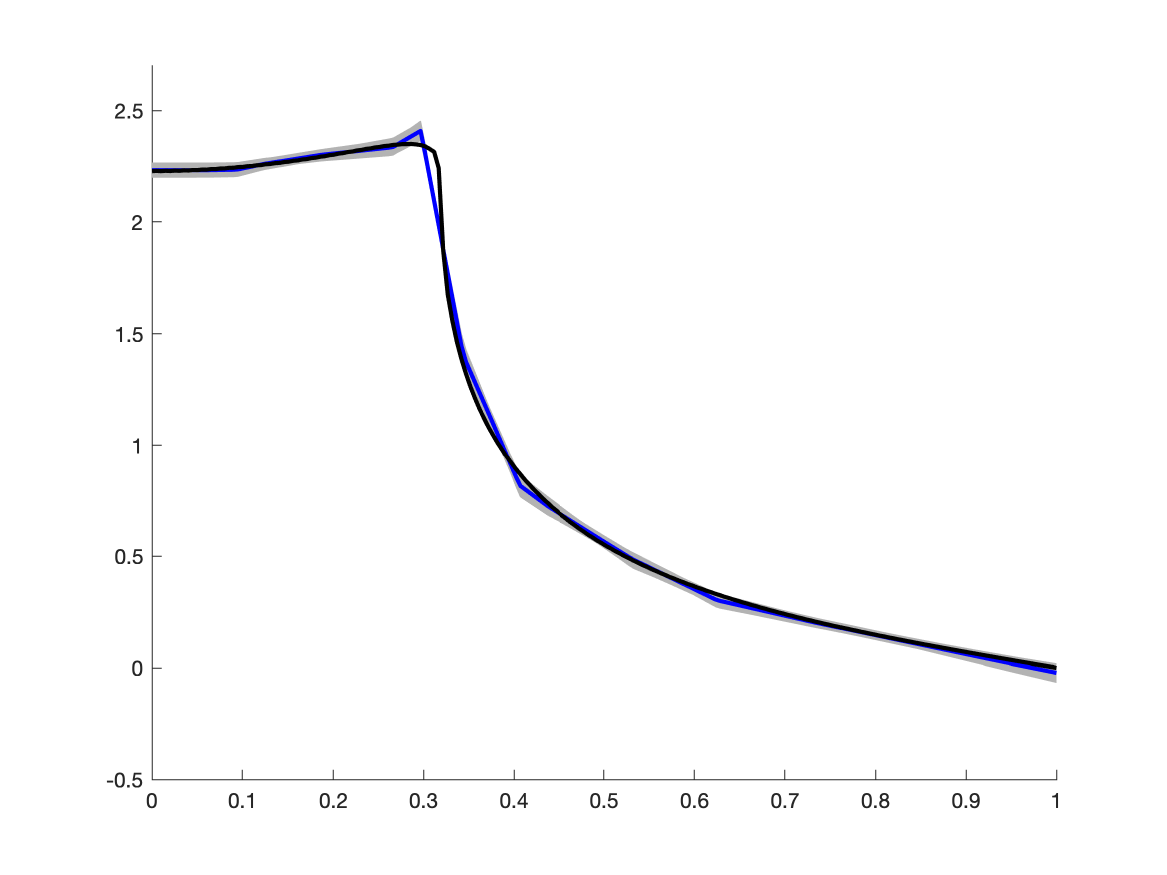} \\
 \includegraphics[width=0.24\textwidth]{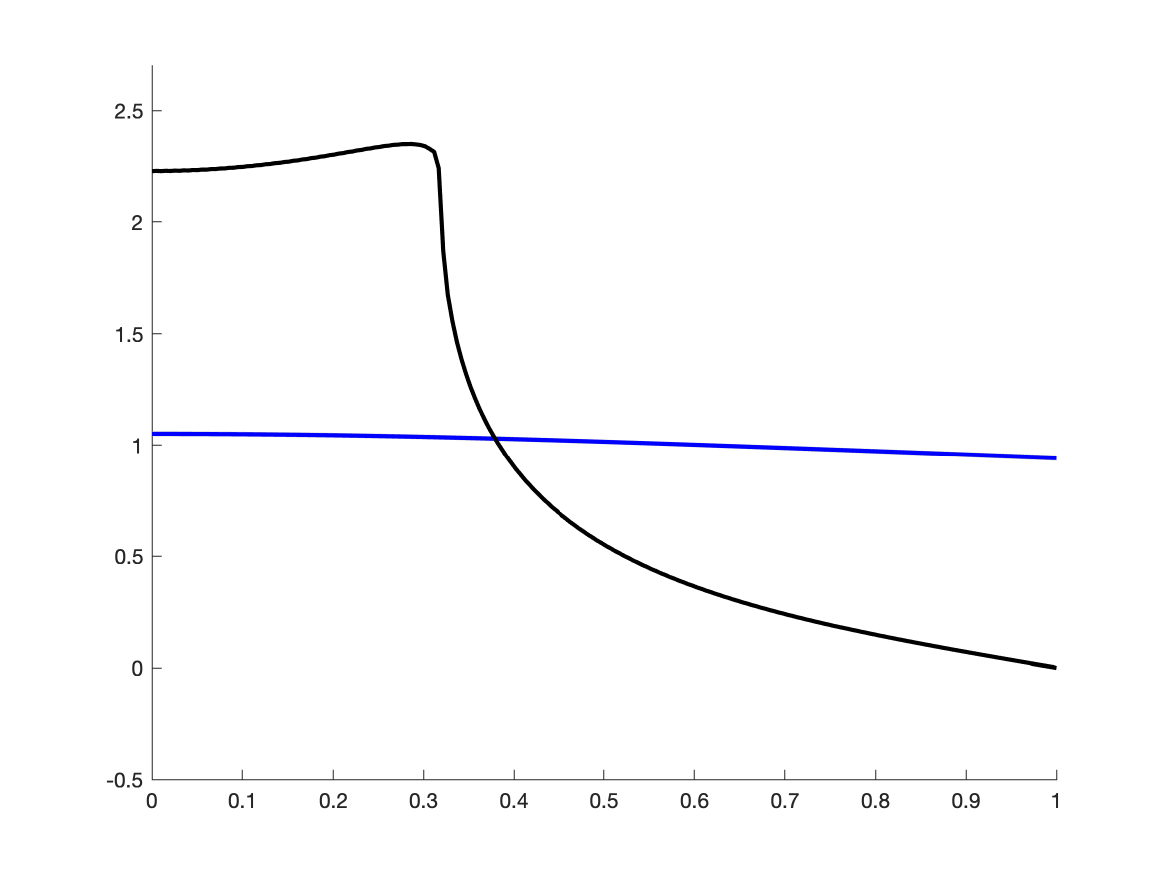}
 \includegraphics[width=0.24\textwidth]{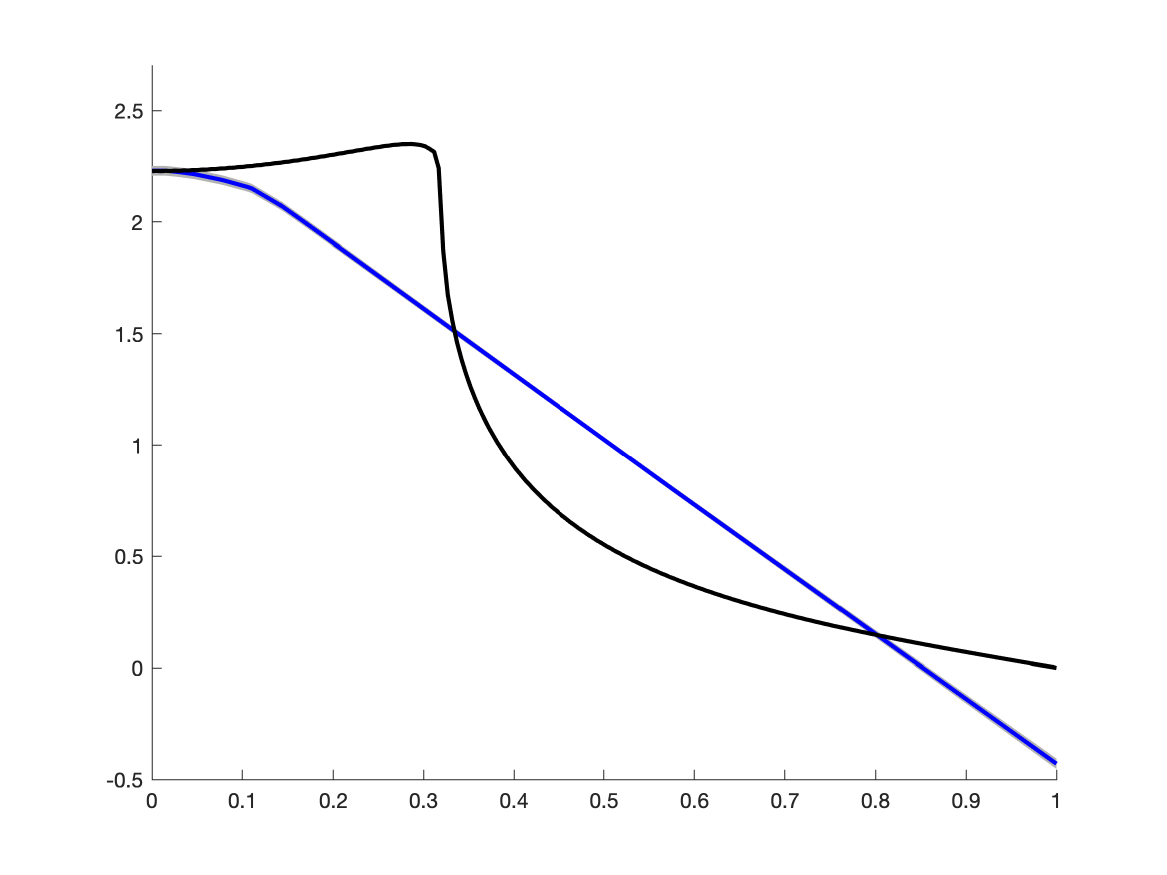} 
 \includegraphics[width=0.24\textwidth]{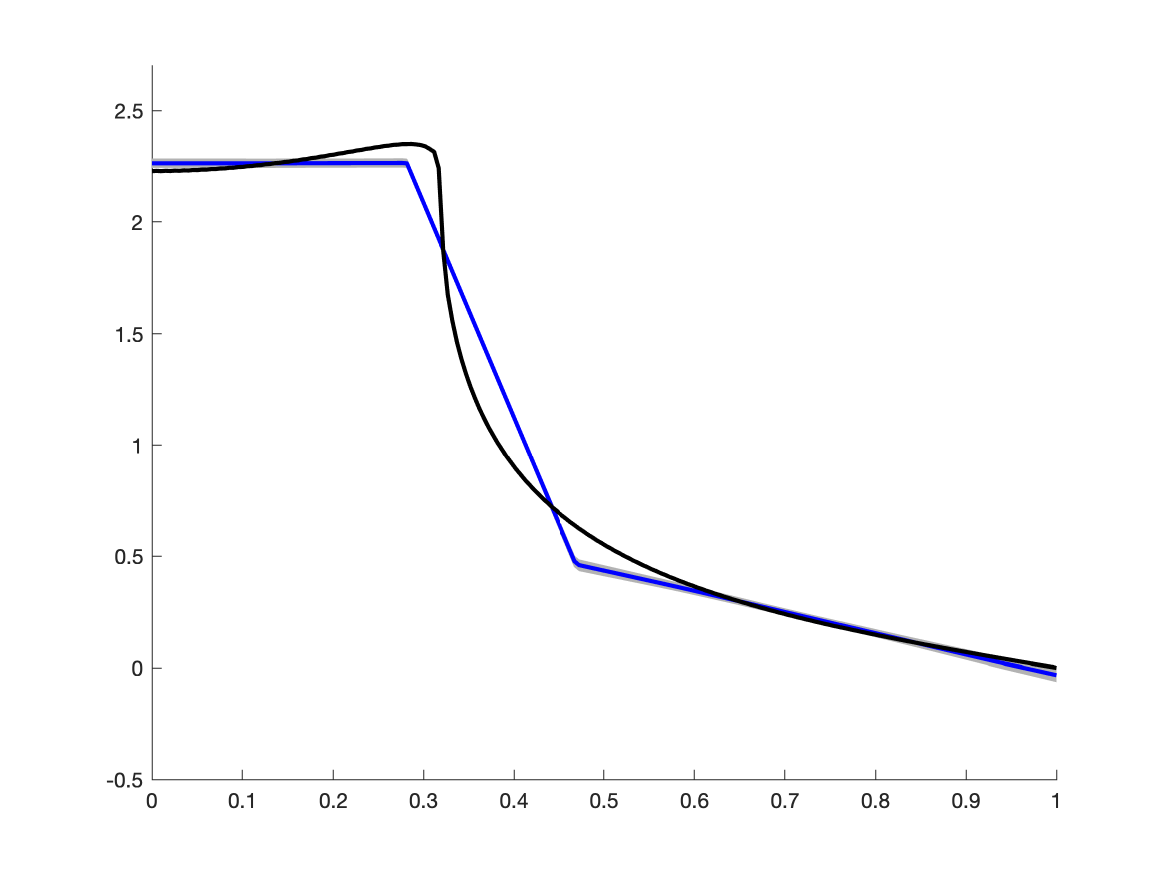} 
 \includegraphics[width=0.24\textwidth]{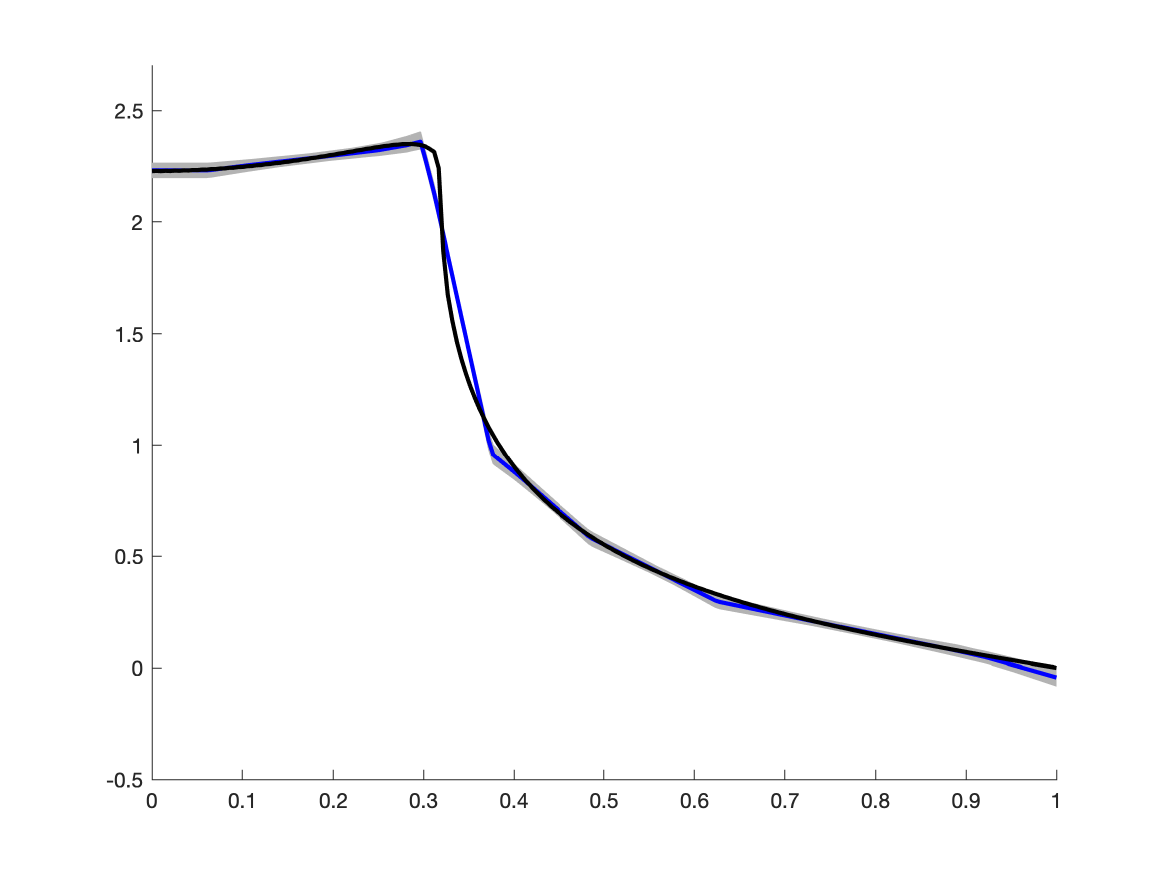}\\
   \includegraphics[width=0.24\textwidth]{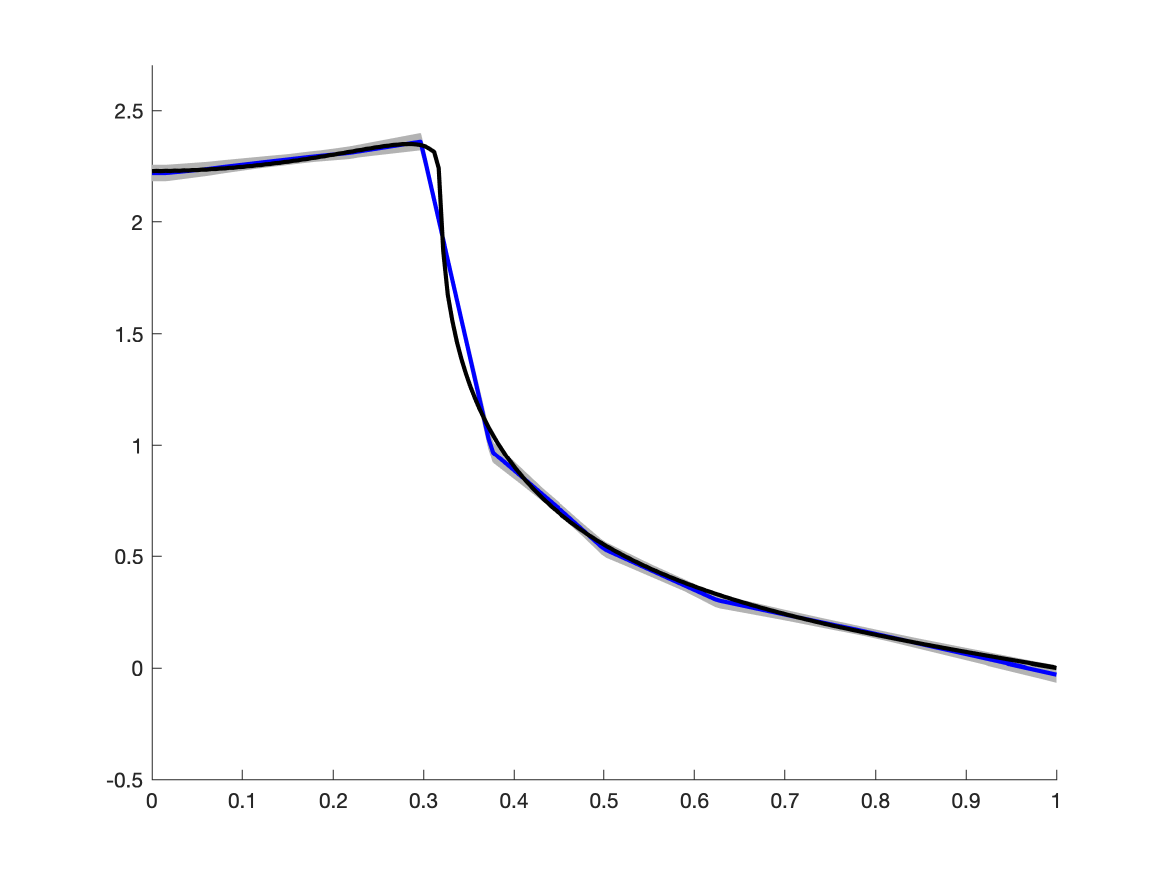}
    \includegraphics[width=0.24\textwidth]{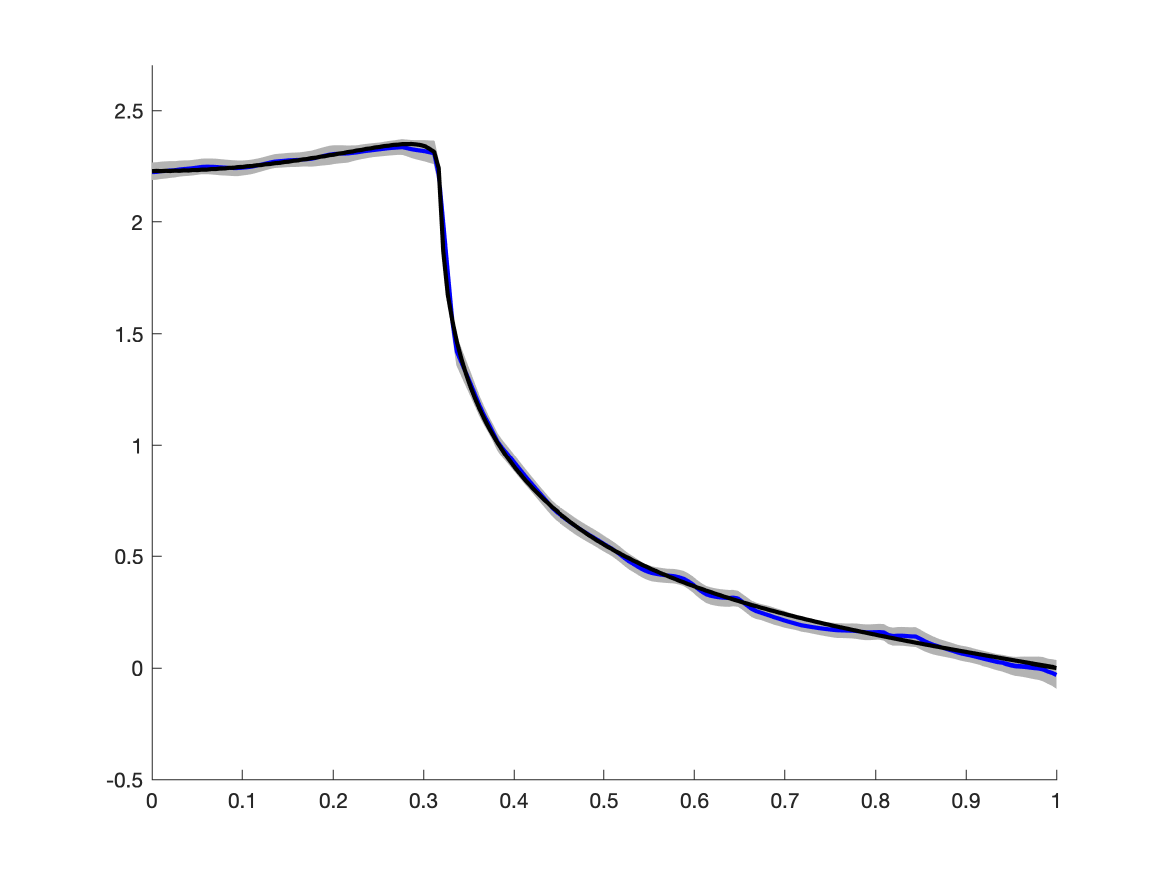}
    \caption{Same setting as Figure \ref{fig-post400} but here with $n=4000$}
    %Random design regression: true function (black), posterior mean (blue), 95\% credible regions (grey), for $n=4000$. Top row: $p$-exponential series priors with $\alpha=2$ and $p=2, 1, 1/2,1/4$ left to right. Second row: series priors with varying $p$-tails as in \eqref{eq:seriesvarp-a} with $\alpha=2$ and as in \eqref{eq:seriesvarp-ot} with $\gamma=1/2$, Cauchy HT($\alpha$) with $\alpha=2$ and Cauchy OT with $\gamma=1/2$, left to right. Third row: shallow neural network priors as in \eqref{SNN-prior} with $\alpha=1/2$, oracle choice of $\sigma_n$ and $p=2,1,1/2,1/4$ left to right. Fourth row: shallow neural network priors as in \eqref{SNN-prior} with $\alpha=1/2$, $\sigma_n=\veps_n^+/N_\alpha$ and $p=2,1,1/2,1/4$ left to right. Bottom row: shallow network prior as in \eqref{SNN-prior} with $\alpha=1/2$, $\sigma=n^{-7/5-0.01}$ and $p_n=1/\log{n}$.}
    \label{fig-post4000}
\end{figure}

\begin{table}[!htbp]
\centering{\footnotesize
\setlength{\tabcolsep}{15pt}
\begin{tabular}{lcccc}
\hline
\textbf{Prior} &
\multicolumn{2}{c}{$n=400$} &
\multicolumn{2}{c}{$n=4000$} \\
\cline{2-3} \cline{4-5}
& \textbf{Err A} & \textbf{Err B} & \textbf{Err A} & \textbf{Err B} \\
\hline
Series $p=2, \;\alpha=2$  & 0.1194 & 0.1226 & 0.0742 & 0.0749\\
Series $p=1, \;\alpha=2$  & 0.0895 & 0.0962 & 0.0477 & 0.0495 \\
Series $p=1/2, \;\alpha=2$  & 0.0814 & 0.0919 & 0.0400 & 0.0431 \\
Series $p=1/4, \;\alpha=2$ & 0.0915 & 0.1050 & 0.0401 & 0.0444 \\
Series varying $p$, $\alpha=2$  & 0.0851 & 0.0984 & 0.0383 & 0.0427 \\
Series varying $p$, $\gamma=1/2$  & 0.0950 & 0.1184 & 0.0398 & 0.0461 \\
Cauchy HT($\alpha$), $\alpha=2$  & 0.0807 & 0.0920 & 0.0394 & 0.0431 \\
Cauchy OT, $\gamma=1/2$  & 0.0707 & 0.0906 & 0.0349 & 0.0397 \\
SNN $p=2, \alpha=1/2$, oracle $\sigma_n$  & 0.2901 & 0.2908 & 0.2113 & 0.2114 \\
SNN $p=1, \alpha=1/2$, oracle $\sigma_n$ & 0.1539 & 0.1559 & 0.1080 & 0.1082 \\
SNN $p=1/2, \alpha=1/2$, oracle $\sigma_n$ & 0.0793 & 0.0837 & 0.0497 & 0.0505 \\
SNN $p=1/4, \alpha=1/2$, oracle $\sigma_n$ & 0.0664 & 0.0718 & 0.0410 & 0.0423 \\
SNN $p=2, \alpha=1/2$, $\sigma_n=\veps_n^+/N_\alpha$  & 0.8569 & 0.8570 & 0.9018 & 0.9018 \\
SNN $p=1, \alpha=1/2$, $\sigma_n=\veps_n^+/N_\alpha$ & 0.3467 & 0.3472 & 0.3473 & 0.3474 \\
SNN $p=1/2, \alpha=1/2$, $\sigma_n=\veps_n^+/N_\alpha$ & 0.0999 & 0.1031 & 0.0797 & 0.0801 \\
SNN $p=1/4, \alpha=1/2$, $\sigma_n=\veps_n^+/N_\alpha$ & 0.0675 & 0.0727 & 0.0498 & 0.0509\\
SNN varying $p$, $\alpha=1/2$ & 0.0729 & 0.0776 & 0.0481 & 0.0492\\
SNN varying $p$, $\alpha=0$, $\gamma=1/2$ & 0.0681 & 0.0750 & 0.0265 & 0.0290\\
\hline
\end{tabular}\caption{$L_2$ average errors of posterior mean (Error A) and $L_2$ contraction-type errors (Error B) for the considered priors, with number of observations $n=400$ and $n=4000$.}\label{tab:prior_errors} }
\end{table}

\section{Discussion}\label{sec : disc}

This work shows that equipping the coefficients of a function (on a basis, or on a ReLU dictionary) with $p$-exponential priors leads to near (small $p$) or full (regime $p\to 0$) adaptation to smoothness. The presented results are obtained for simplicity in white noise (for series priors) and random design regression (for neural network priors). However, they hold much more generally for $\rho$--posteriors in other statistical models, see Section \ref{app:othermodels} of the supplement. 

We now comment on the overparameterized $p$--exponential shallow neural network prior. The prior architecture we have chosen originates from the piecewise affine approximation result Lemma \ref{lem : approx shallow}. To recover $f_0$, the prior randomly draws  weights in the hidden layer and the weights of the input layer are fixed on a certain grid. This specific architecture allows us to define a ``simplest possible" overparameterized SNN, having much more active neurons in the hidden layer compared to the oracle network in Lemma \ref{lem : approx shallow}. This choice also allows us to easily compare performances of different choices of $p$--exponential prior distributions on the weights. In particular, we show (e.g. in Theorem \ref{thm:SNN}) that on this architecture heavier tailed priors perform better (than Gaussian for example). We now informally link this $p$--exponential shallow prior to deeper Gaussian priors. The recent works \cite{ pmlr-v97-vladimirova19a, Zavatone, noci2021precise} have shown that in a deep NN with i.i.d. Gaussian weights, the distribution of the output of a layer at depth $L \ge 1$ is a random vector whose coordinates (conditionally on the input $\mathbf{x}$ of the NN) are {\em marginally} $2/L$--exponential, in the sense of equation \eqref{condt} with $q = 2/L$ (see for instance Theorem~3.1 in \cite{pmlr-v97-vladimirova19a}). The $p$--SNN prior in \eqref{SNN-prior} can then be viewed as a `summary' of an $i.i.d.$ deep Gaussian prior with depth $L = 2/p$. Indeed, we take for prior on the single hidden layer a distribution whose marginals are the same as the output of an $L$--layered Gaussian DNN and removing dependencies. 
 %(of course, taking an i.i.d. $p$-exponential prior we cannot possibly recover the dependencies, making this analogy purely informal).
  With this analogy, the improved performances of heavier priors (e.g. in Theorem \ref{thm:SNN}) can be related to improved performances of DNN priors as the depth increases. Particularly looking at Theorem \ref{thm:SNNvarp}, we show that taking $p = 2/\log n$ provides a smoothness-adaptive procedure. Such result can be related to those of \cite{ce25} who showed that overparameterized heavy-tailed (with heavier polynomial tails rather than $p$--exponential as in the present work) DNN priors of depth $L \gtrsim \log n =2/p$  are fully adaptive. Our results provide a first link between the theoretical results for overparameterized heavy-tailed priors and the more practically employed overparameterized  Gaussian DNN prior (we are not aware of any existing posterior contraction results for such Gaussian priors). The case of purely Gaussian weights for deeper networks will be addressed in future work.  

Finally, % we briefly comment on the fact that, 
for simplicity, we restricted to results to fractional posteriors with parameter $\rho<1$. The main technical reason is that one does not need to build sieve sets and control their entropy, as is usually the case for classical posteriors ($\rho=1$) when using the generic posterior contraction theory \cite{ggv00}. Still, we believe all our results go through also for $\rho=1$: this can be formally proved in the white noise model (see Section \ref{secapp:ubpost} in the Supplement for an example of such a result), but for more general models such as random design regression or classification, it is currently an open question; we refer to the discussion in \cite{AC} for more details on this point.

%\section{Proof of the main results.}\label{sec : proof}

\section{Proof of Theorem \ref{thm : conc series}} \label{sec : proof}
\label{proof : conc series}
Throughout this section, the notation $a_n \lesssim b_n$ denotes inequalities up to a positive constant. 
\begin{proof}
%We give the proof of the (arguably most interesting) case $alpha>beta$
First, thanks to Lemmas \ref{lem : renyi from prior mass} and \ref{lem : vois computations}, in the Gaussian white noise model \eqref{def : gwn}, to obtain $\rho$--posterior contraction in $L_2$--distance, it is enough to show, for some constant $C,D > 0$,
\[ \Pi \left[ ||f-f_0||_2 \le D \veps_n \right] \ge \exp \left\{-C n \veps_n^2 \right\}. \]
%\pe{Doing so you might loose a constant $C \vee D$ in front of the rate but we dont really care }
Let $(\delta_k)_{k\ge 1}$ be a ($n$--dependent) sequence of positive numbers, to be chosen below, that verifies, for a large enough constant $D\ge 1$,
\begin{equation} \label{conddelta}
\sum_{k=1}^{\infty} \delta_k^2 \le (D\veps_n)^2.
\end{equation}
By combining Parseval's identity and \eqref{conddelta}, the event $\{f=(f_k):\  (f_k-f_{0,k})^2\le \delta_k^2\ \text{for all }k\ge 1\}$ is included in $\{f:\ \|f-f_0\|_2\le D\veps_n\}$. Using that the variables $f_k=\si_k\zeta_k$ are independent, one deduces that it is enough to verify the inequality
\begin{equation}\label{ineqgoal}
\prod_{k\ge 1} \Pi\left[ |f_k-f_{0,k}|\le \delta_k \right] \ge \exp \left\{-Cn\veps_n^2 \right\}. 
\end{equation}
We now bound from below the individual probabilities in the last product, that is% the quantities 
\begin{equation}\label{inter}
\Pi\left[ |f_k-f_{0,k}|\le \delta_k \right] = P\left[ |\si_k\zeta_k-f_{0,k}|\le 
\delta_k \right] \ge 
c_0 \int_{(f_{0,k}-\delta_k)/\si_k}^{(f_{0,k}+\delta_k)/\si_k} e^{-c_1 |x|^p} dx,
\end{equation}
for all $k\ge 1$. Since the last integrand is symmetric, we can assume without loss of generality that $f_{0,k}\ge 0$ in the bounds to follow. 

We distinguish the two cases $\al> \be$ and $\al\le \be$. First suppose $\al>\be$. Recall the definitions of $\gamma $ and $N_\ga$ in \eqref{def : gamma}--\eqref{def : Ncutoff} and let us choose $\delta_k$ as follows%, for $k\ge 1$,
\begin{equation} \label{defdel}
\delta_k := 
\begin{cases}
\, 1/\sqrt{n},& \quad 1\le k\le N_\ga,\\
\, 2Lk^{-1/2-\be},& \quad k>N_\ga
\end{cases}.
\end{equation}
Let us check that this choice satisfies the constraint \eqref{conddelta}. By definition, for some % constant 
$C_\be>0$,
\[ 
\sum_{k=1}^{\infty} \delta_k^2 \le \frac{N_\ga}{n} + 4L^2 C_\be N_\ga^{-2\be}.
\]
Recall \eqref{def : rate} the definition of $\veps_n$, since $\al>\be$, one has $\ga>\be$, so that $N_\ga\le N_\be$ and $\veps_n^2\ge n^{-2\be/(2\be+1)}\ge N_\be/n\ge N_{\ga}/n$.
Also, $N_\ga^{-2\be}\le 2\veps_n^2$ for large enough $n$ (using $\lfloor x \rfloor^{-2\be}\le 2x^{-2\be}$ for large $x$), so that the last display is bounded from above by $(1+8 L^2 C_\be)\veps_n^2$.

To bound \eqref{inter} from below, let us first consider the case of indices $1\le k\le N_\ga$. In that case we bound from below the integrand in \eqref{inter} by its smallest value, attained at $x=(f_{0,k}+\delta_k)/\si_k$ (since $f_{0,k}\ge 0$), so that, using $(a+b)^p\le \ka_p(a^p+b^p)$ ($\ka_p=1$ for $p\in(0,1]$),
\begin{align*}
\Pi&\left[ |f_k-f_{0,k}|\le \delta_k \right]  \ge 
\frac{2c_0\delta_k}{\si_k} \exp\left\{-\frac{c_1}{\si_k^p} (f_{0,k}+\delta_k)^p\right\} \\
& \ge \frac{2c_0\delta_k}{\si_k} \exp\left\{-\frac{c_1}{\si_k^p}\ka_p (f_{0,k}^p+n^{-p/2})\right\}  \ge \frac{2c_0\delta_k}{\si_k} \exp\left\{-\frac{c_2}{\si_k^p} 2\ka_p (Lk^{-1/2-\be})^p\right\} ,
\end{align*}
using the definition of $\delta_k$, the regularity condition on $f_0$ and that $Lk^{-1/2-\be}\geqa 1/\sqrt{n}$ for all $1\le k\le N_\ga$ for large enough $n$ (since $N_\ga\le N_\be$ for $\al>\be$). Deduce, for such $k$'s and $n$'s,
%$1\le k\le N_\ga$,
\[ \Pi\left[ |f_k-f_{0,k}|\le \delta_k \right]  \ge  
\frac{2c_0}{\sqrt{n}\si_k} \exp\left\{-2\ka_p c_2(Lk^{\al-\be})^p\right\}. 
%\ge \frac{C_0}{\sqrt{n}} \exp\left\{-C_1k^{p(\al-\be)}\right\}.
\]
In order to bound the product of these probabilities from below in \eqref{ineqgoal}, we distinguish two cases. If $p\le 2$, one notes that $\ga\le \al$ by definition, so that $\sigma_k^{-1}\ge k^{1/2+\ga}$.  Lemma \ref{lemlog} (applied with $\al$ therein replaced by $\ga$) then implies
\[ \prod_{k=1}^{N_\ga} 
\frac{2c_0}{\sqrt{n}\si_k} 
\ge e^{-(1/2+\ga-\log(2c_0))N_\ga}\ge e^{-C_0N_\ga},
\]
for some $C_0>0$. Then using the bound $\sum_{k=1}^{N} k^a\leqa  N^{a+1}$ for any $a>0$ and integer $N$, 
\begin{align*}
 \prod_{k=1}^{N_\ga} \Pi\left[ |f_k-f_{0,k}|\le \delta_k \right] 
 \ge  \exp\left\{ -C_0 N_\ga  - C_2 N_\ga^{p(\al-\be)+1}\right\}
 \ge \exp\left\{ - C_3 N_\ga^{p(\al-\be)+1}\right\} \ge \exp\left\{ - C_3 n \vep^2\right\}, 
\end{align*}
noticing that $N_\ga^{p(\al-\be)+1}\le n\veps_n^2$ follows from the definitions of $N_\ga$ and $\veps_n$. In the case that $p\ge 2$, one instead applies Lemma \ref{lemlogbis}. Noting that the lower bound in that Lemma is itself bounded from below by $\exp(-C_4N_\ga^{p(\al-\be)+1})$, one obtains
\[ \prod_{k=1}^{N_\ga} 
\frac{2c_0}{\sqrt{n}\si_k} 
\ge \exp\left\{-\log(2c_0)N_\ga-C_4N_\ga^{p(\al-\be)+1}\right\}\ge e^{-C_5N_\ga^{p(\al-\be)+1}},
\]
so that the product in the last but one display can be bounded from below by $e^{-C_6N_\ga^{p(\al-\be)+1}}$ and same argument as for the case $p\le 2$ can be used. 
%\sbl{[technical note: here I did a `fixed $\al$' bound; it can probably be made uniform in $\al>\be$ by using the more precise argument used belo in case $\al<\be$ by keeping the $\si_k$ factor, which enables one to remove the log term above]}

To bound \eqref{inter} from below for indices $k>N_\ga$,  the choice of $\delta_k$ in \eqref{defdel} ensures that%the inclusion
\[ \left[ f_{0,k}-\delta_k , f_{0,k}+\delta_k\right] \supset  
\left[ -Lk^{-1/2-\be},Lk^{-1/2-\be} \right]
\]
holds, using \eqref{defdel} and the regularity condition on $f_0$. Further bounding the probability in \eqref{inter} from below gives, in this case
\[ 
\Pi\left[ |f_k-f_{0,k}|\le \delta_k \right]  \ge   \Pi\left[ |f_k|\le Lk^{-1/2-\be} \right]
\ge  \Pi\left[ |\zeta_k|\le Lk^{\al-\be} \right].
 \]
Since $\zeta_k$'s have density $h$ and survival function $\overline{H}$, one has $\Pi\left[ |\zeta_k|\le Lk^{\al-\be} \right]=1-2\overline{H}(Lk^{\al-\be})$. Gathering these bounds and combining with Condition \eqref{condu} gives 
\begin{align*}
\prod_{k>N_\ga} & \Pi\left[ |f_k-f_{0,k}|\le \delta_k \right] 
 \ge \prod_{k>N_\ga} \left( 1-2\overline{H}(Lk^{\al-\be}) \right) \\
& \ge \exp\left\{ \sum_{k>N_\ga} \log \left( 1-2d_0e^{-d_1(Lk^{\al-\be})^q}\right) \right\}
 \ge \exp\left\{ -4d_0\sum_{k>N_\ga} e^{-d_1(Lk^{\al-\be})^q} \right\},
\end{align*}
where we have used the inequality $\log(1-2x)\ge -4x$, valid for $x\in[0,1/4]$.  Since the series $\sum_k e^{-c k^\delta}$ converges for any given constants $c, \delta>0$, one deduces that the last display converges to $1$ as $n\to\infty$ and in particular is bounded from below by $1/2$ for $n$ large enough. Gathering the previous bounds gives
\[ \prod_{k\ge 1} \Pi\left[ |f_k-f_{0,k}|\le \delta_k \right] \ge \exp \left\{ -C n\veps_n^2 \right\}, \]
so that \eqref{ineqgoal} is satisfied for large enough $C$, concluding the proof in the case $\al>\be$.
\end{proof}

%\medskip
\textbf{\large Supplementary Material.} The appendix contains the remaining proofs of the main results and some additional results referenced in the main text.

\bibliography{ptails}

\begin{thebibliography}{54}
\providecommand{\natexlab}[1]{#1}
\providecommand{\url}[1]{\texttt{#1}}
\expandafter\ifx\csname urlstyle\endcsname\relax
  \providecommand{\doi}[1]{doi: #1}\else
  \providecommand{\doi}{doi: \begingroup \urlstyle{rm}\Url}\fi

\bibitem[Abraham and Deo(2023)]{abraham2023deep}
K.~Abraham and N.~Deo.
\newblock Deep {G}aussian process priors for {B}ayesian inference in nonlinear inverse problems.
\newblock \emph{arXiv preprint arXiv:2312.14294}, 2023.

\bibitem[Agapiou and Castillo(2024)]{AC}
S.~Agapiou and I.~Castillo.
\newblock Heavy-tailed {B}ayesian nonparametric adaptation.
\newblock \emph{Ann. Statist.}, 52\penalty0 (4):\penalty0 1433--1459, 2024.
\newblock ISSN 0090-5364,2168-8966.

\bibitem[Agapiou and Savva(2024)]{AgapiouSavva}
S.~Agapiou and A.~Savva.
\newblock {Adaptive inference over {B}esov spaces in the white noise model using p-exponential priors}.
\newblock \emph{Bernoulli}, 30\penalty0 (3):\penalty0 2275 -- 2300, 2024.

\bibitem[Agapiou and Wang(2024)]{awLapInv}
S.~Agapiou and S.~Wang.
\newblock {Laplace priors and spatial inhomogeneity in {B}ayesian inverse problems}.
\newblock \emph{Bernoulli}, 30\penalty0 (2):\penalty0 878 -- 910, 2024.

\bibitem[Agapiou et~al.(2014)Agapiou, Bardsley, Papaspiliopoulos, and Stuart]{abps14}
S.~Agapiou, J.~M. Bardsley, O.~Papaspiliopoulos, and A.~M. Stuart.
\newblock Analysis of the {G}ibbs sampler for hierarchical inverse problems.
\newblock \emph{SIAM/ASA J. Uncertain. Quantif.}, 2\penalty0 (1):\penalty0 511--544, 2014.

\bibitem[Agapiou et~al.(2021)Agapiou, Dashti, and Helin]{AgapiouPEXP}
S.~Agapiou, M.~Dashti, and T.~Helin.
\newblock {Rates of contraction of posterior distributions based on p-exponential priors}.
\newblock \emph{Bernoulli}, 27\penalty0 (3):\penalty0 1616 -- 1642, 2021.

\bibitem[Agapiou et~al.(2026)Agapiou, Castillo, and Egels]{ace}
S.~Agapiou, I.~Castillo, and P.~Egels.
\newblock Heavy-tailed and horseshoe priors for regression and sparse {B}esov rates.
\newblock \emph{Bernoulli}, 2026.
\newblock To appear, arXiv:2505.15543.

\bibitem[Arbel et~al.(2026)Arbel, Pitas, Vladimirova, and Fortuin]{arbel2023primer}
J.~Arbel, K.~Pitas, M.~Vladimirova, and V.~Fortuin.
\newblock {A Primer on {B}ayesian Neural Networks: Review and Debates}.
\newblock \emph{Statistical Science}, 41\penalty0 (2):\penalty0 316 -- 353, 2026.

\bibitem[Bai et~al.(2020)Bai, Song, and Cheng]{baietal20}
J.~Bai, Q.~Song, and G.~Cheng.
\newblock Efficient variational inference for sparse deep learning with theoretical guarantee.
\newblock In \emph{Advances in Neural Information Processing Systems}, 2020.

\bibitem[Barron et~al.(1999)Barron, Birg\'{e}, and Massart]{bbm99}
A.~Barron, L.~Birg\'{e}, and P.~Massart.
\newblock Risk bounds for model selection via penalization.
\newblock \emph{Probab. Theory Related Fields}, 113\penalty0 (3):\penalty0 301--413, 1999.
\newblock ISSN 0178-8051.

\bibitem[Berenfeld et~al.(2024)Berenfeld, Rosa, and Rousseau]{RosaAdaptive}
C.~Berenfeld, P.~Rosa, and J.~Rousseau.
\newblock {Estimating a density near an unknown manifold: A {B}ayesian nonparametric approach}.
\newblock \emph{The Annals of Statistics}, 52\penalty0 (5):\penalty0 2081 -- 2111, 2024.

\bibitem[Castillo(2008)]{ic08}
I.~Castillo.
\newblock Lower bounds for posterior rates with {G}aussian process priors.
\newblock \emph{Electron. J. Stat.}, 2:\penalty0 1281--1299, 2008.

\bibitem[Castillo(2024)]{cstf}
I.~Castillo.
\newblock \emph{{B}ayesian nonparametric statistics}, volume 2358 of \emph{Lecture Notes in Mathematics}.
\newblock Springer, Cham, 2024.
\newblock Saint-Flour Probability Summer School LI---2023.

\bibitem[Castillo and Egels(2025)]{ce25}
I.~Castillo and P.~Egels.
\newblock Posterior and variational inference for deep neural networks with heavy-tailed weights.
\newblock \emph{Journal of Machine Learning Research}, 26\penalty0 (122):\penalty0 1--58, 2025.

\bibitem[Castillo and Randrianarisoa(2025)]{cr25}
I.~Castillo and T.~Randrianarisoa.
\newblock {Deep horseshoe {G}aussian processes}.
\newblock \emph{The Annals of Statistics}, 53\penalty0 (5):\penalty0 1886 -- 1912, 2025.

\bibitem[Castillo et~al.(2014)Castillo, Kerkyacharian, and Picard]{ckp14}
I.~Castillo, G.~Kerkyacharian, and D.~Picard.
\newblock Thomas {B}ayes' walk on manifolds.
\newblock \emph{Probab. Theory Related Fields}, 158\penalty0 (3-4):\penalty0 665--710, 2014.
\newblock ISSN 0178-8051,1432-2064.

\bibitem[Chen et~al.(2018)Chen, Dunlop, Papaspiliopoulos, and Stuart]{cdps18}
V.~Chen, M.~M. Dunlop, O.~Papaspiliopoulos, and A.~M. Stuart.
\newblock Dimension-robust {MCMC} in {B}ayesian inverse problems.
\newblock 2018.
\newblock arXiv preprint 1803.03344.

\bibitem[Ch{\'e}rief-Abdellatif(2020)]{cherief-abdellatif20a}
B.-E. Ch{\'e}rief-Abdellatif.
\newblock Convergence rates of variational inference in sparse deep learning.
\newblock In \emph{Proceedings of the 37th International Conference on Machine Learning}, 2020.

\bibitem[Cotter et~al.(2013)Cotter, Roberts, Stuart, and White]{crsw13}
S.~Cotter, G.~Roberts, A.~Stuart, and D.~White.
\newblock {{MCMC}} methods for functions: Modifying old algorithms to make them faster.
\newblock \emph{Statistical Science}, 28\penalty0 (3):\penalty0 424--446, 2013.

\bibitem[Damianou and Lawrence(2013)]{damianou2013deep}
A.~Damianou and N.~D. Lawrence.
\newblock Deep {G}aussian processes.
\newblock In \emph{Artificial intelligence and statistics}, pages 207--215. PMLR, 2013.

\bibitem[Dolera et~al.(2024)Dolera, Favaro, and Giordano]{dolera2024strong}
E.~Dolera, S.~Favaro, and M.~Giordano.
\newblock On strong posterior contraction rates for {B}esov-{L}aplace priors in the white noise model.
\newblock \emph{arXiv preprint arXiv:2411.06981}, 2024.

\bibitem[Donoho et~al.(1995)Donoho, Johnstone, Kerkyacharian, and Picard]{djkp95}
D.~L. Donoho, I.~M. Johnstone, G.~Kerkyacharian, and D.~Picard.
\newblock Wavelet shrinkage: asymptopia?
\newblock \emph{J. Roy. Statist. Soc. Ser. B}, 57\penalty0 (2):\penalty0 301--369, 1995.
\newblock ISSN 0035-9246.

\bibitem[Finocchio and Schmidt-Hieber(2023)]{fsh23}
G.~Finocchio and J.~Schmidt-Hieber.
\newblock Posterior contraction for deep {G}aussian process priors.
\newblock \emph{Journal of Machine Learning Research}, 24\penalty0 (66):\penalty0 1--49, 2023.

\bibitem[Ghosal and van~der Vaart(2017)]{gvbook}
S.~Ghosal and A.~van~der Vaart.
\newblock \emph{Fundamentals of nonparametric {B}ayesian inference}.
\newblock Cambridge University Press, Cambridge, 2017.

\bibitem[Ghosal et~al.(2000)Ghosal, Ghosh, and van~der Vaart]{ggv00}
S.~Ghosal, J.~K. Ghosh, and A.~W. van~der Vaart.
\newblock Convergence rates of posterior distributions.
\newblock \emph{Ann. Statist.}, 28\penalty0 (2):\penalty0 500--531, 2000.

\bibitem[Gin{\'e} and Nickl(2015)]{GineNickl}
E.~Gin{\'e} and R.~Nickl.
\newblock \emph{Mathematical foundations of infinite-dimensional statistical models}, volume~40.
\newblock Cambridge University Press, 2015.

\bibitem[Giordano(2023)]{GiordanoLapDens}
M.~Giordano.
\newblock {{B}esov-{L}aplace priors in density estimation: optimal posterior contraction rates and adaptation}.
\newblock \emph{Electronic Journal of Statistics}, 17\penalty0 (2):\penalty0 2210 -- 2249, 2023.

\bibitem[Giordano et~al.(2022)Giordano, Ray, and Schmidt-Hieber]{JSHGPcomp}
M.~Giordano, K.~Ray, and J.~Schmidt-Hieber.
\newblock On the inability of {G}aussian process regression to optimally learn compositional functions.
\newblock In \emph{Advances in Neural Information Processing Systems}, volume~35, 2022.

\bibitem[Knapik et~al.(2016)Knapik, Szab\'{o}, van~der Vaart, and van Zanten]{ksvv16}
B.~T. Knapik, B.~T. Szab\'{o}, A.~W. van~der Vaart, and J.~H. van Zanten.
\newblock {B}ayes procedures for adaptive inference in inverse problems for the white noise model.
\newblock \emph{Probab. Theory Related Fields}, 164\penalty0 (3-4):\penalty0 771--813, 2016.

\bibitem[Kohler and Langer(2021)]{kl21}
M.~Kohler and S.~Langer.
\newblock {On the rate of convergence of fully connected deep neural network regression estimates}.
\newblock \emph{The Annals of Statistics}, 49\penalty0 (4):\penalty0 2231 -- 2249, 2021.

\bibitem[Kong and Kim(2025)]{KimKong}
I.~Kong and Y.~Kim.
\newblock Posterior concentrations of fully-connected {B}ayesian neural networks with general priors on the weights.
\newblock \emph{Journal of Machine Learning Research}, 26\penalty0 (94), 2025.

\bibitem[Kong et~al.(2023)Kong, Yang, Lee, Ohn, Baek, and Kim]{kong2023masked}
I.~Kong, D.~Yang, J.~Lee, I.~Ohn, G.~Baek, and Y.~Kim.
\newblock Masked {B}ayesian neural networks: Theoretical guarantee and its posterior inference.
\newblock In \emph{International Conference on Machine Learning}. PMLR, 2023.

\bibitem[Lee and Lee(2022)]{lee2022asymptotic}
K.~Lee and J.~Lee.
\newblock Asymptotic properties for {B}ayesian neural network in {B}esov space.
\newblock In \emph{Advances in Neural Information Processing Systems}, volume~35, pages 5641--5653, 2022.

\bibitem[Lepski\u{\i}(1990)]{lepski90}
O.~V. Lepski\u{\i}.
\newblock A problem of adaptive estimation in {G}aussian white noise.
\newblock \emph{Teor. Veroyatnost. i Primenen.}, 35\penalty0 (3):\penalty0 459--470, 1990.

\bibitem[Lepski\u{\i}(1991)]{lepski91}
O.~V. Lepski\u{\i}.
\newblock Asymptotically minimax adaptive estimation. {I}. {U}pper bounds. {O}ptimally adaptive estimates.
\newblock \emph{Teor. Veroyatnost. i Primenen.}, 36\penalty0 (4):\penalty0 645--659, 1991.

\bibitem[L'Huillier et~al.(2023)L'Huillier, Travis, Castillo, and Ray]{ltcr23}
A.~L'Huillier, L.~Travis, I.~Castillo, and K.~Ray.
\newblock Semiparametric inference using fractional posteriors.
\newblock \emph{Journal of Machine Learning Research}, 24\penalty0 (389):\penalty0 1--61, 2023.

\bibitem[Liu et~al.(2020)Liu, Ong, Shen, and Cai]{liu2020gaussian}
H.~Liu, Y.-S. Ong, X.~Shen, and J.~Cai.
\newblock When {G}aussian process meets big data: A review of scalable {GP}s.
\newblock \emph{IEEE transactions on neural networks and learning systems}, 31\penalty0 (11):\penalty0 4405--4423, 2020.

\bibitem[Nakada and Imaizumi(2020)]{nakada2020adaptive}
R.~Nakada and M.~Imaizumi.
\newblock Adaptive approximation and generalization of deep neural network with intrinsic dimensionality.
\newblock \emph{JMLR}, 21\penalty0 (174):\penalty0 1--38, 2020.

\bibitem[Nieman and Szab{\'o}(2025)]{nieman2025adaptive}
D.~Nieman and B.~Szab{\'o}.
\newblock {Adaptive Sparse Variational Approximations for Gaussian Process Regression}.
\newblock \emph{Bayesian Analysis}, pages 1 -- 20, 2025.

\bibitem[Noci et~al.(2021)Noci, Bachmann, Roth, Nowozin, and Hofmann]{noci2021precise}
L.~Noci, G.~Bachmann, K.~Roth, S.~Nowozin, and T.~Hofmann.
\newblock Precise characterization of the prior predictive distribution of deep {R}e{LU} networks.
\newblock \emph{Advances in Neural Information Processing Systems}, 34:\penalty0 20851--20862, 2021.

\bibitem[Ohn and Lin(2024)]{ohn2024adaptive}
I.~Ohn and L.~Lin.
\newblock Adaptive variational {B}ayes: Optimality, computation and applications.
\newblock \emph{The Annals of Statistics}, 52\penalty0 (1):\penalty0 335--363, 2024.

\bibitem[Polson and Ro{\v{c}}kov{\'a}(2018)]{polson2018posterior}
N.~G. Polson and V.~Ro{\v{c}}kov{\'a}.
\newblock Posterior concentration for sparse deep learning.
\newblock \emph{Advances in Neural Information Processing Systems}, 31, 2018.

\bibitem[Rasmussen and Williams(2006)]{rw06}
C.~E. Rasmussen and C.~K.~I. Williams.
\newblock \emph{{G}aussian processes for machine learning}.
\newblock Adaptive Computation and Machine Learning. MIT Press, Cambridge, MA, 2006.

\bibitem[Rousseau and Szabo(2017)]{RS17}
J.~Rousseau and B.~Szabo.
\newblock {Asymptotic behaviour of the empirical {B}ayes posteriors associated to maximum marginal likelihood estimator}.
\newblock \emph{The Annals of Statistics}, 45\penalty0 (2), 2017.

\bibitem[Schmidt-Hieber(2020)]{jsh20}
J.~Schmidt-Hieber.
\newblock {Nonparametric regression using deep neural networks with {R}e{LU} activation function}.
\newblock \emph{The Annals of Statistics}, 48\penalty0 (4):\penalty0 1875 -- 1897, 2020.

\bibitem[Szab{\'o} and Zhu(2026)]{szabo2025vecchiagaussianprocessesprobabilistic}
B.~Szab{\'o} and Y.~Zhu.
\newblock Vecchia {G}aussian processes: Probabilistic properties, minimax rates and methodological developments, 2026.
\newblock arXiv preprint 2410.10649.

\bibitem[Szab{\'o} et~al.(2013)Szab{\'o}, van~der Vaart, and van Zanten]{Svz13}
B.~T. Szab{\'o}, A.~W. van~der Vaart, and J.~H. van Zanten.
\newblock {Empirical {B}ayes scaling of {G}aussian priors in the white noise model}.
\newblock \emph{Electronic Journal of Statistics}, 7, 2013.

\bibitem[Tang et~al.(2026)Tang, Wu, Cheng, and Dunson]{tang2024adaptive}
T.~Tang, N.~Wu, X.~Cheng, and D.~Dunson.
\newblock {Adaptive Bayesian regression on data with low intrinsic dimensionality}.
\newblock \emph{The Annals of Statistics}, 54\penalty0 (2):\penalty0 1080 -- 1099, 2026.

\bibitem[van~der Vaart and van Zanten(2008)]{vz08}
A.~W. van~der Vaart and J.~H. van Zanten.
\newblock Rates of contraction of posterior distributions based on {G}aussian process priors.
\newblock \emph{Ann. Statist.}, 36\penalty0 (3):\penalty0 1435--1463, 2008.
\newblock ISSN 0090-5364.

\bibitem[van~der Vaart and van Zanten(2009)]{vz09}
A.~W. van~der Vaart and J.~H. van Zanten.
\newblock Adaptive {B}ayesian estimation using a {G}aussian random field with inverse gamma bandwidth.
\newblock \emph{Ann. Statist.}, 37\penalty0 (5B):\penalty0 2655--2675, 2009.

\bibitem[Vladimirova et~al.(2019)Vladimirova, Verbeek, Mesejo, and Arbel]{pmlr-v97-vladimirova19a}
M.~Vladimirova, J.~Verbeek, P.~Mesejo, and J.~Arbel.
\newblock Understanding priors in {B}ayesian neural networks at the unit level.
\newblock In \emph{Proceedings of the 36th International Conference on Machine Learning}, volume~97, pages 6458--6467. PMLR, 2019.

\bibitem[Vladimirova et~al.(2020)Vladimirova, Girard, Nguyen, and Arbel]{vladimirova2020sub}
M.~Vladimirova, S.~Girard, H.~Nguyen, and J.~Arbel.
\newblock Sub-{W}eibull distributions: Generalizing sub-{G}aussian and sub-exponential properties to heavier tailed distributions.
\newblock \emph{Stat}, 9\penalty0 (1):\penalty0 e318, 2020.

\bibitem[Yang and Dunson(2016)]{YangDunsonManifold}
Y.~Yang and D.~B. Dunson.
\newblock {{B}ayesian manifold regression}.
\newblock \emph{The Annals of Statistics}, 44\penalty0 (2):\penalty0 876 -- 905, 2016.

\bibitem[Zavatone-Veth and Pehlevan(2021)]{Zavatone}
J.~Zavatone-Veth and C.~Pehlevan.
\newblock Exact marginal prior distributions of finite {B}ayesian neural networks.
\newblock In \emph{Advances in Neural Information Processing Systems}, volume~34, pages 3364--3375, 2021.

\end{thebibliography}

\appendix

\section*{Supplementary material}

This supplement is structured along three Sections. Section \ref{secapp:proofs} contains 
the remaining proofs of the results stated in the main paper, along with, at the end of the Section, a few remarks on allowing slightly different parameter choices. Section \ref{secapp:add} starts with a few comments on how to extend the results to other statistical models. It then also contains the statement and proof of a few additional results: the case of Sobolev truths; the proof of the lower bound of Theorem \ref{thm :  lower bound} in the special (and easier) case $p=1$; an example of statement of an upper-bound rates for the classical posterior (case $\rho=1$, in contrast to the case $\rho<1$ considered in the main paper); an extension of Theorem \ref{thm:SNN} to the case of light tails $p>1$.
Finally, Section \ref{sec: technical res} gathers a few technical lemmas used along the proofs.

\section{Remaining proofs} \label{secapp:proofs}

\subsection{End of the proof of Theorem \ref{thm : conc series}}\label{proof:thm1case2}

Let us now focus on the case $\al\le \be$. The constants $C_i$, $i=1, \dots, 6$ might change from the previous case. Again, we split along indices $k$, this time separating $1\le k \le N_\al$ and $k>N_\al$. Noticing that $N_{\al}/n\le \veps_n^2=n^{2\al/(2\al+1)}$, we use the inclusion
\begin{align*}
\bigcap_{k=1}^{N_\al} & \left\{ f \, : \, (f_k-f_{0,k})^2 \le \frac{D^2}{2n}\right\} \, \cap \,
\left\{ f \,:\, \sum_{k>N_{\al}} (f_k-f_{0,k})^2 \le \frac{(D\veps_n)^2}{2} \right\} \subset \, \left\{ f \,:\, \|f-f_0\|_2^2 \le (D\veps_n)^2 \right\}.
\end{align*}
 First consider the case of indices $1\le k \le N_\al$. One reproduces the same argument as for the case $\al>\be$ above: setting $\delta_k :=D/\sqrt{2n}$,
\begin{align*}
\Pi\left[ |f_k-f_{0,k}|\le \delta_k \right]  \ge \frac{2c_0\delta_k}{\si_k} \exp\left\{-\frac{c_1}{\si_k^p}\ka_p (f_{0,k}^p+n^{-p/2})\right\} \ge \frac{2c_0\delta_k}{\si_k} \exp\left\{-c_1\ka_p(L^p + \si_k^{-p}n^{-p/2})\right\},
\end{align*}
where we use that $|f_{0,k}|\le Lk^{-1/2-\be}\le Lk^{-1/2-\al}=L\sigma_k$ for $\al\le \be$. Using further that $\sigma_k^{-1}\le \sqrt{n}$ for $k\le N_\al$, one obtains, with $\si_k=k^{-1/2-\al}$ by definition,
\[ \Pi\left[ |f_k-f_{0,k}|\le \delta_k \right] 
\ge C_2D\frac{k^{1/2+\al}}{\sqrt{n}} \exp\left\{ - C_3 \right\}.
\]
Taking the product over $1\le k \le N_\al$ and using Lemma \ref{lemlog} 
leads to
\[ \prod_{k=1}^{N_\al}  \Pi\left[ |f_k-f_{0,k}|\le \delta_k \right] 
\ge
\exp\left\{ -(1/2+\al)N_\al - (C_4+\log{D}) N_\al \right\}
\ge
\exp\left\{ - C_5  n \vep^2 \right\},
 \]
where $C_5>0$ is a large enough constant and we used $N_\al\leqa n\veps_n^2$ in the last inequality.

Second let us deal with indices $k > N_\al$. One has $\sum_{k>N_\al} f_{0,k}^2 \le L^2 N_\al^{-2\al}\leqa \veps_n^2$. So for large enough $D$, it holds
\begin{align*}
\Pi\left[ \, \sum_{k>N_{\al}} (f_k-f_{0,k})^2 \le \frac{(D\veps_n)^2}{2} \right] 
& \ge  
\Pi\left[ \, \sum_{k>N_{\al}} f_k^2 \le \frac{(D\veps_n)^2}{4} \right] \ge 
\Pi\left[ \, \sum_{k>N_{\al}} \left(f_k^2 - \sigma_k^2 E[\zeta_k^2]\right) \le \frac{(D\veps_n)^2}{8} \right],
\end{align*}
where for the last inequality we have used that $E[\zeta_k^2]=E[\zeta_1^2]$ is a fixed constant and $\sum_{k>N_\al} \sigma_k^2\leqa \veps_n^2$ by definition of $\si_k$ and $ N_\al$ and once again taking $D$ large enough. Looking now at the complement, Markov's inequality gives
\begin{align*}
\Pi&\left[ \, \sum_{k>N_{\al}} \left(f_k^2 - \sigma_k^2 E[\zeta_k^2]\right) > (D\veps_n)^2 / 8 \,\right] 
= P\left[ \, \sum_{k>N_{\al}} \sigma_k^2 \left(\zeta_k^2 -  E[\zeta_k^2]\right) > (D\veps_n)^2 / 8 \,\right]\\
& \qquad \le \frac{64}{(D\veps_n)^4}
\operatorname{Var}\left[ \, \sum_{k>N_{\al}} \sigma_k^2 \zeta_k^2  \,\right]
\le \frac{64}{(D\veps_n)^4} \operatorname{Var}\left[\zeta_1^2\right] \sum_{k>N_\al} \sigma_k^4\le \frac{C_7}{(D\veps_n)^4}N_{\al}^{-1-4\al}.
\end{align*}
Since $N_{\al}^{-1-4\al}\leqa N_\al^{-1}\veps_n^4$, one obtains that the prior mass in the last display goes to $0$ as $n\to\infty$. In particular, the last but one display is bounded from below by $1/2$ for large enough $n$. Putting together the above bounds in both regimes of $k$'s leads to
\[ \Pi\left[ \|f-f_0\|_2 \le D\veps_n\right]
\ge \frac12 \exp\left\{ - C_5 n\veps_n^2 \right\} 
\ge \exp\left\{ - C n \veps_n^2 \right\},
\]
which concludes the proof of the theorem.

\subsection{Proof of Theorem \ref{thm : lower bound}}\label{proof : lower bound}

\begin{proof}
To simplify the notation, we give the proof first for the standard posterior $\rho=1$. We also focus on the (harder) case $p<1$: the proof for $p=1$ is similar, though easier  (the maximiser $\mu_k$ of the function $h_k$
 below is completely explicit in that case): for completeness we give it explicitly below, see Section \ref{prooflbp1}. 
 
Let us choose $f_0$ as the function in $\cF^{\beta}(L)$ defined through its basis coefficients by $f_{0,k}=L k^{-1/2-\be}$ for $\be>0, L>0$. 
Let us recall the definitions of $\gamma$ in \eqref{def : gamma} and  $N_\ga := \lfloor n^{1/(2 \ga +1)} \rfloor$, as well as the target rate $\veps_n=\veps_n(p,\al,\be)=N_\ga^{-\be}$. Denoting $\|g\|_{N_\ga}^2 :=\sum_{k=1}^{N_\ga} g_k^2$ for any square-integrable function $g$, it is enough to prove, for small $m>0$ to be chosen and $n_\ga=d N_\ga$ for some small enough constant $d$ to be chosen below, that, as $n \to \infty$,
\[ E_{f_0}\Pi[\|f-f_0\|_{n_\ga} \ge m\veps_n \given X] \to 1. \]
% Let us define the sequence $\mu=(\mu_k)$ by, for any $k\ge 1$,
% \begin{equation}\label{defmu}
% \mu_k := X_k - \frac{X_k^{p-1}}{n \sigma_k^p}.
% \end{equation}
For a sequence $\mu =(\mu_k)$ to be defined below (that will correspond to the mode of the posterior over the corresponding coordinates), the triangle inequality gives $\|f_0- \mu\|_{n_\ga}
\le \|f_0-f\|_{n_\ga}+\|f- \mu\|_{n_\ga}$. This implies
\begin{equation}\label{eq:indicator-rhs} 
\Pi[\|f-f_0\|_{n_\ga} \ge m\veps_n \given X] 
\ge \Pi[\|f-\mu\|_{n_\ga} \le m\veps_n \given X]\cdot \1\{ \|f_0- \mu\|_{n_\ga} \ge 2m\veps_n \}. 
 \end{equation}
It now suffices to show that each term of the product of the right hand side of the last display goes to $1$ in probability under $P_{f_0}^{(n)}$.

We start by showing $\Pi[\|f-\mu\|_{n_\ga} \le m\veps_n \given X]$ goes to $1$ in probability under $P_{f_0}^{(n)}$. Consider the event  \[ \cB_n := \left\{ |\xi_k| \leq \sqrt{2 \log n}, \quad \text{for all } k=1,\dots,n_\ga \right\}.\] Since $P_{f_0}^{(n)}(\cB_n) \to 1$ (Lemma \ref{lem : bounds on the event lower bound}), it is sufficient to work on the event $\cB_n$ defined above and show 
$E_{f_0} \Pi[\|f-\mu\|_{n_\ga} > m\veps_n \given X] \mathbf{1}_{\cB_n} \to 0.$ 
Below are the steps to obtain such result.

\bigskip

{\em Expression of the posterior.}
Since $f_k = \sigma_k \zeta_k$, where $\zeta_k$ is $p$--exp, one can express the posterior on the coefficient $f_k$ as 
\[ d\Pi_k(\theta \, | \, X^n) = \frac{ e^{-\frac n2{(X_k - \theta)^2 - \frac{|\theta|^p}{p \sigma_k^p}}} \, d\theta }{\int e^{-\frac n2{(X_k - \theta)^2 - \frac{|\theta|^p}{p \sigma_k^p}}} \, d\theta }. %=: \frac{\tilde{g_k}(\theta) \, d\theta}{\int \tilde{g_k}(\theta) \, d\theta},
\]
Let use denote  $\tilde{g}_k^+(\theta) := %\tilde{g_k}(\theta)
e^{-\frac n2{(X_k - \theta)^2 - \frac{|\theta|^p}{p \sigma_k^p}}}
\mathbf{1}_{\theta \geq 0}$ and $\tilde{g}_k^-(\theta) := 
%\tilde{g_k}(\theta) 
e^{-\frac n2{(X_k - \theta)^2 - \frac{|\theta|^p}{p \sigma_k^p}}}
\mathbf{1}_{\theta < 0}$, as well as
\[G_k^+ := \int \tilde{g}_k^+(\theta)\, d\theta, \qquad G_k^- := \int \tilde{g}_k^-(\theta) \, d\theta \qquad \text{and} \qquad w_k^+ := \frac{G_k^+}{G_k^+ + G_k^-}.\]
Re-normalizing $g_k^+ := \tilde{g}_k^+ / G_k^+$ and $g_k^- := \tilde{g}_k^- / G_k^-$, one obtains the decomposition
\[ d\Pi_k(\theta \, | \, X^n) = w_k^+ g_k^+(\theta) \, d\theta + (1-w_k^+)g_k^-(\theta) \, d\theta. \]
%To conclude the proof, it suffices to check that $\Pi[\|f-\mu\|_{n_\ga} \le m\veps_n \given X]$ goes to $1$ in probability under $P_{f_0}$. %By Lemma \ref{lemdob}, denoting $\Pi_+:=\cN(\mu_k,1/n)_+$, it suffices to check that $\Pi_+[\|f-\mu\|_{n_\ga} \le m\veps_n \given X]$ goes to $1$ in probability 
By Markov's inequality, for any $\mu = (\mu_k)$ to be chosen below,
\begin{align} 
\Pi[& \|f-\mu\|_{n_\ga} > m\veps_n \given X] \mathbf{1}_{\cB_n} \le \frac{1}{(m\veps_n)^2} \int \|f-\mu\|_{n_\ga}^2 d\Pi(f\given X)  \mathbf{1}_{\cB_n} \nonumber\\
 & \le \frac{ \mathbf{1}_{\cB_n}}{(m\veps_n)^2} \sum_{k=1}^{n_\ga} \left[w_k^+ \int (f_k-\mu_k)^2 g_k^+(f_k)df_k
 +  (1-w_k^+) \int (f_k-\mu_k)^2 g_k^-(f_k)df_k \right]. 
 \label{eqimp}
 %&  \le \frac{ \1_{\cA_n}}{(m\veps_n)^2} \sum_{k=1}^{n_\ga} \left[ \int (f_k-\mu_k)^2d\cN(\mu_k,1/n)_+(f_k)
% +  c_1e^{-c_2n^{\frac{\al-\be}{2\ga+1}}} \int (f_k-\mu_k)^2d\cN(\nu_k,1/n)_{-}(f_k) \right],
\end{align}
%where for the last inequality we use $w_k^+\le 1$ and the uniform bound on $(1-w_k^+)$ obtained in Lemma \ref{lemdob}, and $2\ga=\al+\be$. 

{\em Control on the event $\cB_n$.} In order to control the various terms that will appear below, we set a constant $M=M(p)$, such that
\begin{equation}\label{proof : def M lower bound}
    M(p) \ge 10^{2-p} \vee \Big(1 + \frac{2^p}{2p(1-p)}\Big)^{2-p} \vee \frac{8}{p} \left(\frac32 \right)^p.
\end{equation} 
Applying Lemma \ref{lem : bounds on the event lower bound} shows that we can take $d>0$ small enough in $n_\gamma = d N_\gamma$, such that, for all $k \leq n_\gamma$ and $n$ large enough, on $\cB_n$ and under $P_{f_0}^{(n)}$, we have
\begin{equation}\label{eq:Xsq-lb} n X_k^2 \ge M \left( \frac{X_k}{\si_k} \right)^{p}.\end{equation}
Also note that, thanks to Lemma \ref{lem : bounds on the event lower bound}, for all $k \le n_\gamma$, on $\cB_n$, we have
\[ X_k \ge f_{0,k}/2 > 0.\]

{\em Laplace-type bound on the positive part, choice of $\mu_k$}. Since $X_k >0$, provided we choose $\mu_k >0$ on $\cB_n$, we will see below that most of the posterior mass will come from the `positive part'. Since $w_k^+\le 1$, in order to control the term involving $g_k^+$ in \eqref{eqimp}, it suffices to bound from above, on the event $\cB_n$,
\[ \sum_{k \leq n_\gamma} \int (\theta - \mu_k)^2 g_k^+(\theta) \, d\theta.\]
Let us define 
\begin{figure}[h]
    \centering
    \includegraphics[width=0.7\linewidth]{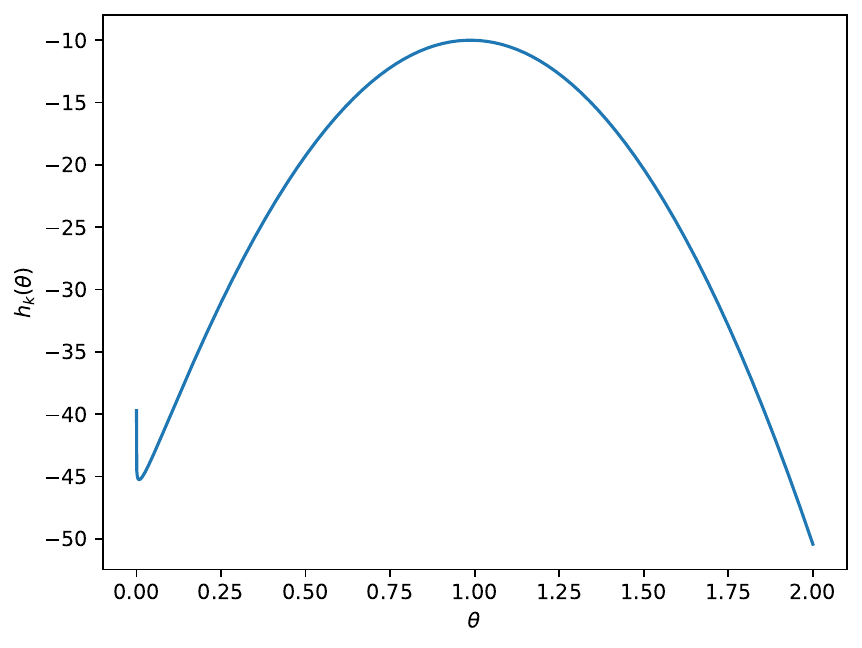}
    \caption{Function $h_k$ with $p = 0.1$, $X_k = \si_k =1$ and $n = 80 > 10^{2-p}$}
    \label{fig:placeholder}
\end{figure}
\begin{equation}\label{eq : h} h_k(\theta) := -\frac n2{(X_k - \theta)^2 - \frac{\theta^p}{p \sigma_k^p}}, \qquad k = 1 , \dots, n_\gamma\end{equation}
for $\theta>0$.
Simple algebraic computation show that $h_k''' < 0$ and
\begin{align}
    \label{eq : first deriv h} h_k'(\theta) &= n(X_k- \theta) - \frac{\theta^{p-1}}{\sigma_k^p}, \\
    \label{eq : sec deriv h} h_k''(\theta) &= - n + (1-p)\frac{\theta^{p-2}}{\sigma_k^p}.
\end{align}
Hence, as $\theta>0$ grows, $h_k''$ strictly decreases from $+\infty$ to $-n$ and vanishes at a single inflection point $\theta_2^*$, such that $h_k''(\theta_2^*)=0$, where
\begin{equation}\label{eq : inflexpoint}
    \theta_2^* := \left(\frac{n\sigma_k^p}{1-p} \right)^{\frac{1}{p-2}}.
\end{equation}
Denote $\psi(u) := 4u^{1/(1+u)} + 4^{-u}u^{-u/(1+u)}$, for $u\in(0,1)$. Using $u \log \frac1u \leq 1 +u$, available for all $u \in (0,1)$ we obtain $\psi(u) \le 4+e < 7$. We can then compute
\begin{align*}
    h_k'(\theta_2^*) \geq h_k'(4\theta_2^*) &= n \left(X_k -4 \left(\frac{n\sigma_k^p}{1-p} \right)^{\frac{1}{p-2}} - \frac{4^{p-1}}{n \sigma_k^p}\left(\frac{n\sigma_k^p}{1-p} \right)^{\frac{p-1}{p-2}}  \right) \\
    &=n \left( X_k - (n\sigma_k^p)^{\frac{1}{p-2}} \psi(1-p) \right) \\
    &\ge n( X_k - 7(n\sigma_k^p)^{\frac{1}{p-2}}).
\end{align*} 
Since $M > 7^{2-p}$ in \eqref{proof : def M lower bound} and thanks to our choice of $d >0$, by \eqref{eq:Xsq-lb} we have $nX_k^2 > 7^{2-p}(X_k/\sigma_k)^p$ and thus $h_k'(\theta_2^*) \geq h_k'(4\theta_2^*) >0$.

Thus, $h_k'$ starts from $- \infty$ at $0$, ends at $-\infty$ at $+\infty$ and takes a positive value at $\theta_2^*$, therefore $h_k'$ vanishes at exactly two critical points $\theta_m^*$ and $\theta_M^*$, such that $\theta_m^* < \theta_2^* < \theta_M^*$ (these critical points obviously depend on $k$; we do not make this explicit to avoid overloading the notation). In particular, at $\theta_M^*$ there is a local maximum of $h_k$ (because $\theta_M^* > \theta_2^*$ implies $h_k''(\theta_M^*) < h_k''(\theta_2^*) =0$). 

First notice that $\lim_{\theta \to \infty}h_k(\theta) = - \infty$. Second, recalling \eqref{eq : inflexpoint} the definition of $\theta_2^*$, we have $h_k(2\theta_2^*) > h_k(0)$, this shows that the global maximum of $h_k$ is attained at $\theta_M^*$ and not at the boundary. Indeed
\begin{align*}
    h_k(2\theta_2^*) -h_k(0) &= 2nX_k\theta_2^* - 2n(\theta_2^*)^2 - \frac{2^p}{p \sigma_k^p}(\theta_2^*)^p \\
    &= 2n \theta_2^* \left(X_k - (1+\frac{2^p}{2p(1-p)})\theta_2^* \right) >0,
\end{align*} 
where the last inequality comes from choice of $M$ large enough in \eqref{proof : def M lower bound}.

We finally set, for any $k \le n_\ga$, $\mu_k$ as the global maximizer of $h_k$,
\[ \mu_k := \theta_M^*.\]
We apply now a Laplace's method--type argument, to upper bound, on the event $\cB_n$,
\[  \int (\theta - \mu_k)^2 g_k^+(\theta) \, d\theta := \frac{\int_0^\infty (\theta - {\theta_M^*})^2 e^{h_k(\theta)} \, d\theta}{\int_0^\infty  e^{h_k(\theta)} \, d\theta}.\]
Taylor's formula at $\theta_M^* = \mu_k$, provides $\xi_\theta \in [\theta \wedge\theta_M^*, \theta \vee\theta_M^*]$, such that
\begin{equation}\label{eq : taylor on h}
    h_k(\theta) = h_k(\theta_M^*) + \frac12(\theta-\theta_M^*)^2 h_k''(\xi_\theta).
\end{equation} 
Recall \eqref{eq : sec deriv h} the expression of $h_k''$. Plugging the simple bound $h_k'' \geq -n$ in \eqref{eq : taylor on h}, allows us to bound the denominator
\[ \int_0^\infty  e^{h_k(\theta)} \, d\theta \ge e^{h_k(\theta_M^*)}\int_0^\infty e^{- \frac{n}{2}(\theta - \theta_M^*)^2} \, d\theta \geq \sqrt{\frac{\pi}{2n}}e^{h_k(\theta_M^*)}.\]
For the numerator, recall \eqref{eq : inflexpoint} the definition of $\theta_2^*$ and recall that we showed $h_k'(4\theta_2^*) >0$, thus $2\theta_2^* \leq 4\theta_2^* \leq \theta_M^*$. We cut the integral at $2\theta_2^*$ and obtain 
\[ \int_0^\infty (\theta - {\theta_M^*})^2 e^{h_k(\theta)} \, d\theta = \underbrace{\int_0^{2\theta_2^*} (\theta - {\theta_M^*})^2 e^{h_k(\theta)} \, d\theta}_{I_1} +\underbrace{\int_{2\theta_2^*}^\infty (\theta - {\theta_M^*})^2 e^{h_k(\theta)} \, d\theta}_{I_2}.\]
Bounding $I_2$ first, for $\theta\in[2\theta_2^*, \infty)$, we have $\xi_\theta\ge 2\theta_2^*$, hence $h_k''(\xi_\theta) \leq h_k''(2\theta_2^*) = -n (1-2^{p-2}) \leq -n/2$ and thus 
\[ I_2 \leq e^{h_k(\theta_M^*)} \int_{2\theta_2^*}^{\infty}(\theta- {\theta_M^*})^2 e^{-\frac{n}{4}(\theta - \theta_M^*)^2} \, d\theta \leq {e^{h_k(\theta_M^*)}\frac{4 \sqrt{\pi}}{n\sqrt{n}}} .\]
Now for $I_1$, we know $h_k$ decreases from $0$ to $\theta_m^*$ and then increases up until $\theta_M^* \geq 2 \theta_2^*$, thus
\[I_1 \leq \big(e^{h_k(0)} \vee e^{h_k(2\theta_2^*)}\big) \int_0^{2\theta_2^*}(\theta - {\theta_M^*})^2 \, d \theta \leq 4 \theta_2^* [(\theta_2^*)^2 + ({\theta_M^*})^2 ] (e^{h_k(0)} \vee e^{h_k(2\theta_2^*)}).\]
Considering the maximum on the right hand side, we showed in the study of $h_k$ above, that, on the event $\cB_n$, we have $h_k(2\theta_2^*) > h_k(0)$. Along with ${\theta_M^*} \le X_k$ (from the critical equation $h_k'(\theta_M^*)=0$ in \eqref{eq : first deriv h}), this provides,
\[I_1 \le 4 \theta_2^* [(\theta_2^*)^2 + X_k^2 ]  e^{h_k(2\theta_2^*)}.\]
Along with the denominator bound, we obtain 
\[ \frac{\int_0^\infty (\theta - {\theta_M^*})^2 e^{h_k(\theta)} \, d\theta}{\int_0^\infty  e^{h_k(\theta)} \, d\theta}
 \lesssim \sqrt{n}\theta_2^* [(\theta_2^*)^2 + X_k^2 ] e^{h_k(2\theta_2^*) - h_k(\theta_M^*)} + {\frac1{n}}. \]
We showed above that $4 \theta_2^* \leq \theta_M^*$, so that $2 \theta_2^* - \theta_M^* \leq -\theta_M^*/2$. Combined with an application of the mean-value theorem, and noticing that $h_k'$ is decreasing in $[2\theta_2^*, \theta_M^*]$ (since it vanishes at $\theta_m^*$ and $\theta_M^*$, while $h_k'(\theta_2^*)>0$, where $\theta_m^*<2\theta_2^*<\theta_M^*$), we have
\[h_k(2\theta_2^*) - h_k(\theta_M^*) \leq  h_k'(2\theta_2^*)[2\theta_2^* - \theta_M^*] \leq - \frac12 \theta_M^*h_k'(2\theta_2^*).\]
Doing similar computations as for the bound of $h_k'(4\theta_2^*)$ above and using $n X_k^2 \geq 10^{2-p} (X_k/\sigma_k)^p$ (again $M \ge 10^{2-p}$ in \eqref{proof : def M lower bound}), we have (from $2 +e \le 5$)
\[ h_k'(2\theta_2^*) \ge n( X_k - 5(n\sigma_k^p)^{\frac{1}{p-2}}) \ge n X_k/2. \]
These inequalities combined provide $e^{h_k(2\theta_2^*) - h_k(\theta_M^*)} \leq e^{- \frac14 n X_k \theta_M^*}$ and thus 
\[ \frac{\int_0^\infty (\theta - {\theta_M^*})^2 e^{h_k(\theta)} \, d\theta}{\int_0^\infty  e^{h_k(\theta)} \, d\theta}
 \lesssim \sqrt{n}\theta_2^* [(\theta_2^*)^2 + X_k^2 ] e^{- \frac14 n X_k \theta_M^*} + \frac1n. \]
Recall $\theta_2^* \leq \theta_M^*$. Using  the definition of $\theta_2^*$ in \eqref{eq : inflexpoint}, the fact that $h_k'(\theta_M^*)=0$ where $h_k'$ is given in \eqref{eq : first deriv h}, as well as $(1-p)^{-\frac{1-p}{2-p}} \leq 2$, we obtain
\begin{equation}\label{eq : thetaM bounded by X_k/2}
    \theta_M^* = X_k - \frac{(\theta_M^*)^{p-1}}{n \sigma_k^p} \ge X_k - (1-p)^{-\frac{1-p}{2-p}} (n\sigma_k^p)^{\frac{p-1}{p-2}-1} \ge X_k - 2(n\sigma_k^p)^{\frac{1}{p-2}}.
\end{equation} 
Since $M \ge 4^{2-p}$ in \eqref{proof : def M lower bound}, we have $nX_k^2 \geq 4^{2-p}(X_k/\sigma_k)^p$ and the last display is further lower bounded from below by $X_k/2$, so that
\[ \frac{\int_0^\infty (\theta - {\theta_M^*})^2 e^{h_k(\theta)} \, d\theta}{\int_0^\infty  e^{h_k(\theta)} \, d\theta}
 \lesssim \sqrt{n}\theta_2^* [(\theta_2^*)^2 + X_k^2 ] e^{- \frac18 n X_k^2 } + \frac1n. \]
%  Finally, we employ Lemma \ref{lem : err max and mu} below to bound $(\theta_M^* - \mu_k)^2 \lesssim (\theta_2^*)^2$, such that
% \[ \frac{\int_0^\infty (\theta - \mu_k)^2 e^{h_k(\theta)} \, d\theta}{\int_0^\infty  e^{h_k(\theta)} \, d\theta}
%  \lesssim \sqrt{n}\theta_2^* [(\theta_2^*)^2 + X_k^2 ] e^{- \frac18 n X_k^2 } + (\theta_2^*)^2 + \frac2n. \]
Summing up over $k \leq n_\gamma$ gives two terms on the right hand side. The first term is smaller than any polynomial power of $n$. Indeed, thanks to Lemma \ref{lem : bounds on the event lower bound}, on the event $\cB_n$, we have $f_{0,k}/2 \le X_k \le 3f_{0,k}/2$. There exist constants $c_1,c_2$, such that, recalling $\alpha>\beta$,
\[ \sqrt{n}\sum_{k \le n_\gamma}\theta_2^* [(\theta_2^*)^2 + X_k^2 ] e^{- \frac18 n X_k^2 } \lesssim \sqrt{n}e^{-c_2 n n_\gamma^{-2\beta -1}} \sum_{k \le n_\gamma}(n\sigma_k^p)^{\frac{1}{p-2}}[(n\sigma_k^p)^{\frac{2}{p-2}} +f_{0,k}^2] \lesssim n^{c_1} e^{-c_2n^{\frac{p(\al- \be)}{1 +2\be + p(\al - \be)}}}. \]
The last displayed bound is $o(n_\ga^{-2\be})$ as $n \to \infty$. For the third term, simply notice that $n_\gamma/n = o(n_\gamma^{-2\beta})$, so that it holds 
\begin{equation}\label{eq : bound pos part lowerp}
    E_{f_0} \sum_{k \leq n_\gamma} w_k^+\int (\theta - \mu_k)^2 g_k^+(\theta) \, d\theta \mathbf{1_{\cB_n}} = o(n_\gamma^{-2\beta}).
\end{equation}
{\em Control on the weight of the negative part.} We are only left to control
\[(1-w_k^+) \int (f_k-\mu_k)^2 g_k^-(f_k)df_k,\]
on the event $\cB_n$.
We start by bounding the weight $1 - w_k^+= (1+G_k^+/G_k^-)^{-1} \leq G_k^-/G_k^+$. For $\theta <0$ and $X_k \geq 0$, we have $(X_k - \theta)^2 \geq X_k^2$, so that
\[ G_k^- = \int_{-\infty}^0e^{-\frac n2{(X_k - \theta)^2 - \frac{|\theta|^p}{p \sigma_k^p}}} \, d\theta  \leq e^{-\frac{n X_k^2}{2}} \int_0^{\infty} e^{ - \frac{\theta^p}{p \sigma_k^p}} \, d\theta. \]
The change of variable $\theta = \sigma_k (p u)^{1/p}$ yields
\[ G_k^- \leq p^{\frac1p-1}\Gamma \left( 1/p \right) \sigma_ke^{-\frac{nX_k^2}{2}}.\]
For $G_k^+$, restricting the integral on $[X_k/2, 3X_k/2]$, so that $(X_k - \theta)^2 \leq X_k^2/4$, provides
\[ G_k^+ \geq \int_{X_k/2}^{3X_k/2} e^{-\frac n2{(X_k - \theta)^2 - \frac{\theta^p}{p \sigma_k^p}}} \, d\theta \geq e^{- \frac{nX_k^2}{8} }\int_{X_k/2}^{3X_k/2}e^{- \frac{\theta^p}{p \sigma_k^p}} \, d\theta \geq {X_k}e^{- \frac{nX_k^2}{8} }e^{- \frac1p \left( \frac32 \right)^p \left( \frac{X_k}{\sigma_k}\right)^p}.  \]
Combining, we obtain
\[ 1-w_k^+ \lesssim \frac{\sigma_k}{X_k}\exp \left\{- \frac38 n X_k^2 + \frac1p \left( \frac32 \right)^p \left( \frac{X_k}{\sigma_k}\right)^p \right\}, \]
where the constant depends only on $p$. We now employ Lemma \ref{lem : bounds on the event lower bound}, on the event $\cB_n$ and thanks to our choice of $M$ in \eqref{proof : def M lower bound}, we have for all $k \le n_\ga$ and large enough $n$, 
\[  \frac{L}{2} \si_k \le \frac{f_{0,k}}{2}\le X_k \qquad \text{and} \qquad nX_k^2 \geq  \frac8p \left( \frac32\right)^p \frac{X_k^p}{\sigma_k^p}.\]
This leads to the bound, on the event $\cB_n$ and under $P_{f_0}^{(n)}$,
\begin{equation}\label{eq : borne poids negatif}
  1 - w_k^+ \lesssim  \exp \{ - \frac1{16}n f_{0,k}^2 \}. 
\end{equation}
We are left to control
\[ \int (\theta - \mu_k)^2 g_k^-(\theta) \, d\theta = \frac{\int_0^\infty(\theta + {\theta_M^*})^2 e^{-\frac n2 (X_k + \theta)^2 - \frac{\theta^p}{p \sigma_k^p}}\, d\theta}{\int_0^\infty e^{-\frac n2 (X_k + \theta)^2 - \frac{\theta^p}{p \sigma_k^p}}\, d\theta}.\]
First use $(\theta + {\theta_M^*})^2 \leq 2 \theta^2 + 2 ({\theta_M^*})^2$ and upper bound similarly as before
\[ \int_0^\infty\theta^2 e^{-\frac n2 (X_k + \theta)^2 - \frac{\theta^p}{p \sigma_k^p}}\, d\theta \leq e^{-\frac{nX_k^2}{2}} \int_0^\infty \theta^2e^{- \frac{\theta^p}{p \sigma_k^p}}\, d\theta \lesssim e^{-\frac{nX_k^2}{2}}p^{\frac3p-1}\Gamma \left( 3/p \right) \sigma_k ^3 .\]
For the denominator, restricting the integral up to $c \sigma_k$, where $c>0$ is a small enough constant to be chosen below, we obtain
\[\int_0^\infty e^{-\frac n2 (X_k + \theta)^2 - \frac{\theta^p}{p \sigma_k^p}}\, d\theta \ge e^{-\frac{c^p}p} \int_0^{ c\sigma_k} e^{-\frac n2 (X_k + \theta)^2 }\, d\theta \ge c\sigma_k e^{-\frac{c^p}p}e^{-\frac n2 (X_k + c\sigma_k)^2 } .\]
Finally, combining the two previous bounds yields
\[ \int (\theta - \mu_k)^2 g_k^-(\theta) \, d\theta \lesssim  ({\theta_M^*})^2 + \sigma_k^2 e^{- \frac{n}{2}X_k^2 + \frac{n}2 (X_k +c\sigma_k)^2}.\]
By Lemma \ref{lem : bounds on the event lower bound}, on the event $\cB_n$, we have $L\sigma_k \leq 2X_k$ and $2X_k \leq 3f_{0,k}$, thus 
\[ \int (\theta - \mu_k)^2 g_k^-(\theta) \, d\theta \lesssim  ({\theta_M^*})^2 + \sigma_k^2 e^{ 2  \frac{c}{L}(1+ \frac{c}{L})n X_k^2} \lesssim ({\theta_M^*})^2 + \sigma_k^2 e^{ \frac92  \frac{c}{L}(1+ \frac{c}{L})n f_{0,k}^2}.\]
Using the bound \eqref{eq : borne poids negatif} above, we finally have
\[ \sum_{k\leq n_\gamma} (1-w_k^+)\int (\theta - \mu_k)^2 g_k^-(\theta) \, d\theta \lesssim \sum_{k \leq n_{\gamma}} (({\theta_M^*})^2 + \sigma_k^2 e^{ \frac92  \frac{c}{L}(1+ \frac{c}{L})n f_{0,k}^2})e^{- \frac1{16} n f_{0,k}^2}. \]
Using ${\theta_M^*} = X_k -({\theta_M^*})^{p-1}/n\sigma_k^p \leq X_k \leq f_{0,k}/2$ and $L\sigma_k \le 2X_k$, this bound becomes
\[ \sum_{k\leq n_\gamma} (1-w_k^+)\int (\theta - \mu_k)^2 g_k^-(\theta) \, d\theta \lesssim \sum_{k \leq n_{\gamma}} f_{0,k}^2(1+ e^{ \frac92  \frac{c}{L}(1+ \frac{c}{L})n f_{0,k}^2})e^{- \frac1{16} n f_{0,k}^2}. \]
We choose $c>0$ so that $\frac92 \frac{c}L(1+\frac{c}{L}) = \frac{1}{32}$ and we obtain, on $\cB_n$, recalling $\alpha>\beta$ (hence $\gamma>\beta$)
\begin{equation}\label{eq : bound neg part lowerp}
    \sum_{k\leq n_\gamma} (1-w_k^+)\int (\theta - \mu_k)^2 g_k^-(\theta) \, d\theta \lesssim \sum_{k \leq n_{\gamma}} f_{0,k}^2 e^{-\frac{1}{32}n f_{0,k}^2} = o(n^{-1}) = o(\varepsilon_n^2).
\end{equation}

{\em Conclusion of the Proof}. We gather the bounds \eqref{eq : bound pos part lowerp} and \eqref{eq : bound neg part lowerp} into equation \eqref{eqimp} and finally obtain
\[ E_{f_0}\Pi[\|f-\mu\|_{n_\ga} > m\veps_n \given X] \mathbf{1}_{\cB_n} = o(1).\]
We are only left to show that the indicator $\1\{ \|f_0- \mu\|_{n_\ga} \ge 2m\veps_n \}$ in \eqref{eq:indicator-rhs}  goes to $1$ under $P_{f_0}^{(n)}$. Taking the expectation and using $\|\mu-f_0 \|_{n_\ga}\ge  \|\mu-X\|_{n_\ga}- \| X -f_0 \|_{n_\ga}$ by the triangle inequality,
\[ P_{f_0}^{(n)}[ \| \mu-f_0 \|_{n_\ga} \ge 2m\veps_n] \ge P_{f_0}^{(n)}[ \| X -f_0 \|_{n_\ga} \le m\veps_n \, , \, \| \mu-X\|_{n_\ga} \ge 3m\veps_n ]. \] 
Recall the definition of the event $\cB_n$, which satisfies $P_{f_0}^{(n)}(\cB_n) \to 1$ (Lemma \ref{lem : bounds on the event lower bound}),
\[ \cB_n := \left\{ |\xi_k| \leq \sqrt{2 \log n}, \quad \text{for all } k=1,\dots,n_\ga \right\}.\]
On this event and under $P_{f_0}^{(n)}$, using $\alpha > \beta$, we have for $n$ large enough
\[ \| X -f_0 \|_{n_\ga}^2 = \sum_{k \le n_\ga} \frac{1}{n} |\xi_k|^2 \leq \frac{n_\ga}{n} \times 2 \log n = o(\veps_n^2).\]
Second, recall that $\mu_k = \theta_M^*$, defined by $n(X_k- \theta_M^*) = \si_k^{-p}(\theta_M^*)^{p-1}$, satisfying $\theta_M^* \le X_k$ so that, since $p<1$,
\[ \| \mu-X\|_{n_\ga}^2 = \sum_{k \le n_\ga} \frac{\si_k^{-2p}}{n^2} |\theta_M^*|^{2(p-1)} \ge \sum_{k \le n_\ga} \frac{\si_k^{-2p}}{n^2} |X_k|^{2(p-1)} .\]
On the event $\cB_n$, thanks to Lemma \ref{lem : bounds on the event lower bound}, we have $|X_k| \asymp |f_{0,k}|$ and thus
\begin{align*}
    \| \mu-X\|_{n_\ga}^2 \gtrsim \sum_{k \le n_\ga} \frac{\si_k^{-2p}}{n^2}  \left|f_{0,k}\right|^{2(p-1)} \gtrsim \frac{n_\ga^{2(1 + \beta + p(\al - \be))}}{n^2} \gtrsim \veps_n^2.
\end{align*}
Combining the two previous facts shows that for a small enough constant $m$, for $n$ large enough,
\[ P_{f_0}^{(n)}[ \| X -f_0 \|_{n_\ga} \le m\veps_n \, , \, \| \mu-X\|_{n_\ga} \ge 3m\veps_n ] \ge P_{f_0}^{(n)}(\cB_n), \]
where $P_{f_0}(\cB_n) \to 1$. We have thus shown that the indicator in \eqref{eq:indicator-rhs} goes to 1 in $E_{f_0}$-expectation, hence also in $P_{f_0}^{(n)}$-probability, which concludes the proof.
\end{proof}

\subsection{Proof of Theorem \ref{thm : p to zero}}\label{proof : p to zero}

\begin{proof}

Recall $N_\beta= \lfloor n^{1/(2 \beta +1)} \rfloor$ is  the usual frequency cutoff. Using independence of the $f_k$ drawn from the prior, it is sufficient to show
\[\prod_{k \geq 1} \Pi[|f_k - f_{0,k}| \leq \delta_k] \geq \exp( - n \varepsilon_n^2) ,\]
where 
\begin{equation} 
\delta_k := 
\begin{cases}
\, 1/\sqrt{n},& \quad 1\le k\le N_\beta,\\
\, 2Lk^{-1/2-\beta},& \quad k>N_\beta
\end{cases}.
\end{equation}
Indeed, recalling $\varepsilon_n = n^{-\beta/(2\beta +1)} \log^{\eta'} n = N_\beta^{-\beta} \log^{\eta'}n = \sqrt{N_\beta /n} \log^{\eta'} n$, we have, for $n$ large enough,
\[\sum_{k} \delta_k^2 = \sum_{k \leq N_\beta} \frac1n+ 4L^2 \sum_{k \geq N_\beta} k^{-1-2\beta} \leq \frac{N_\beta}{n} + 4L^2 N_\beta^{-2\beta} \leq \varepsilon_n^2.\]
The following inclusions hold,
\[ \left\{f \, : \, \forall k\ge 1, \quad |f_k - f_{0,k}| \leq \delta_k\right\} \subset \left\{ f \, : \, \sum_{k \ge 1} |f_k - f_{0,k}|^2 \leq \sum_{k \ge 1} \delta_k^2 \right\} \subset \{ f \, : \, ||f - f_0||_2 \leq \varepsilon_n \}, \]
establishing the claimed sufficiency of the condition above.
Let us start with the indices $k \leq N_\beta$. We bound from below 
\begin{align*}
\Pi\left[ |f_k-f_{0,k}|\le \delta_k \right]  = P\left[ |\si_k \zeta_k-f_{0,k}|\le 
\frac1{\sqrt{n}} \right] \nonumber =
 \int_{(f_{0,k}-1/{\sqrt{n}})/\si_k}^{(f_{0,k}+1/{\sqrt{n}})/\si_k} h_{p_k}(x) dx,
\label{inter}
\end{align*}
for all $k\ge 1$. Since $h_{p_k}$ the density of $\zeta_k$ is symmetric (see \eqref{def : p-exp dist}), we can assume without loss of generality that $f_{0,k}\ge 0$ in the bounds to follow. We can then bound the previous integral using the smallest value of $h_{p_k}$, which leaves us with
\[ \Pi\left[ |f_k-f_{0,k}|\le \delta_k \right] \geq \frac{2}{\sigma_k\sqrt{n}} h_{p_k}\left(\frac{f_{0,k} + 1/\sqrt{n}}{\si_k}\right).\]
Using the lower bound on $h_p$ given in Lemma \ref{lem : sandwich}, we obtain
\begin{equation}\label{eq : proofptozer bound1}
    \Pi[|f_k - f_{0,k}| \leq \delta_k] \gtrsim \frac{U_{k}}{\sqrt{n}} \exp\left\{ - \frac{1}{p_k}\left( \frac{f_{0,k} + 1/\sqrt{n}}{\sigma_k} \right)^{p_k}\right\},
\end{equation}
where 
\[ U_{k} := \sqrt{p_k}  \exp \left(\frac1{p_k} - \frac{p_k}{12} \right) \sigma_k^{-1}.\]
Using first the smoothness condition on $f_0$, we have $f_{0,k} \leq L k^{-1/2 - \beta}$ and then using $k \le N_\beta$ we obtain $k^{-1/2 - \beta} \geq n^{-\frac{1/2+\beta}{2 \beta +1}} = 1/\sqrt{n}$. Since $p_k \in (0,1]$, we obtain 
\[ (f_{0,k} + 1/\sqrt{n})^{p_k} \leq (L+1)^{p_k} k^{-p_k(1/2 + \beta)} \leq (L+1) k^{-p_k(1/2 + \beta)}.\]
Using $\si_k \le 1$ and $p_k \le 1$, we have $\sqrt{p_k} e^{\frac1{p_k} - \frac{p_k}{12}} \ge 1 \ge \si_k$ and thus $U_k \ge 1$.
%Using once again $p_k \in (0,1]$, we have, $U_k \geq e^{-1/12} \sqrt{p_k} e^{1/p_k}\sigma_k^{-1} \geq e^{-1/12}$ thanks to condition \eqref{eq : compatib cond 1}.
Finally plugging these two inequalities in \eqref{eq : proofptozer bound1}, we obtain,
\[\Pi[|f_k - f_{0,k}| \leq \delta_k] \gtrsim \frac1{\sqrt{n}} \exp\left\{ -  \frac{L+1}{p_k}\left( \frac{  k^{-\beta - 1/2}}{\sigma_k} \right)^{p_k}\right\} = \frac{1}{\sqrt{n}}\exp\{-(L+1)z_k \}. \]
From this bound, it follows
\[ \prod_{k \leq N_\beta}\Pi[|f_k - f_{0,k}| \leq \delta_k] \gtrsim \exp \left \{ - \sum_{k \leq N_\beta} (\log \sqrt{n} + (L+1) z_k )\right\}. \]
Employing the summability condition \eqref{eq : sum cond} on $z_k$, we obtain $C>0$ and $\eta > 1$, such that 
\[ \prod_{k \leq N_\beta}\Pi[|f_k - f_{0,k}| \leq \delta_k] \geq e^{- C N_\beta \log^{\eta} n} \geq e^{-n \varepsilon_n^2}.\]
For the other part of the product, indexed by $k > N_\beta$, we have $\delta_k = 2L k^{-1/2 - \beta}$, such that
\[ [f_{0,k} - \delta_k , f_{0,k} + \delta_k] \supset [- Lk^{-1/2-\beta}, Lk^{-1/2 - \beta}].\]
Recalling $r_k = \sigma_k^{-1}k^{-1/2-\be}$, we can then bound the probability of interest as
\[ 
\Pi\left[ |f_k-f_{0,k}|\le \delta_k \right]  \ge   \Pi\left[ |f_k|\le Lk^{-1/2-\be} \right]=  P\left[ |\zeta_k|\le L \sigma_k^{-1}k^{-1/2-\be}  \right] =P\left[ |\zeta_k|\le Lr_k  \right].
 \]
Recalling that $\zeta_k$ has symmetric density $h_{p_k}$ and survival function $ \overline{H}_{p_k}$, we get
\[\Pi\left[ |f_k-f_{0,k}|\le \delta_k \right]  \ge 1 - 2\overline{H}_{p_k}(Lr_k). \]
Taking now the product over $ k > N_\beta$, using $L \ge 1$ and the monotonicity of $\overline{H}_{p_k}$, we obtain
\[  \prod_{k > N_\beta} \Pi[|f_k - f_{0,k}| \leq \delta_k] \geq \exp \left\{ \sum_{k > N_\beta} \log \left( 1 - 2 \overline{H}_{p_k}(r_k)\right)\right\} .\]
Thanks to the condition \eqref{eq : compatib cond 2}, we have $r_k \geq 1$. Using Lemma \ref{lem : sandwich}, we get
\[ \overline{H}_{p_k}(r_k) \lesssim\frac{e^{\frac1{p_k}}}{\sqrt{p_k}} r_k e^{- r_k^{p_k}/p_k} \lesssim \exp \left\{ \frac12 \log\frac1{p_k} + \frac1{p_k} + \log r_k - \frac{r_k^{p_k}}{p_k} \right\}.\]
Using first $\frac12 \log\frac1{p_k} \leq \frac1{p_k}$ and then $\log r_k = \frac1{p_k}\log r_k^{p_k} \leq \frac1{2p_k}r_k^{p_k}$, we get 
\[ \overline{H}_{p_k}(r_k) \lesssim \exp \left\{ \frac2{p_k} -\frac{r_k^{p_k}}{2p_k} \right\} \lesssim \exp \left\{ -\frac{r_k^{p_k}}{4p_k} \right\} = \exp \left\{- \frac{z_k}{4} \right\},\]
where the assumption $r_k^{p_k} \geq 8$ was used in the last inequality. As the general term of a convergent series (thanks to condition \eqref{eq : sum cond}) the previous bound goes to zero. Thus, for $n$ large enough and $k > N_\beta$, we have $\overline{H}_{p_k}(r_k) \leq 1/4$. We then use the inequality $\log(1-2x) \geq -4x$, available for $x \leq 1/4$, to obtain
\[ \prod_{k > N_\beta} \Pi[|f_k - f_{0,k}| \leq \delta_k] \geq \exp \left\{ -4 \sum_{k > N_\beta} \overline{H}_{p_k}(r_k)\right\} \geq \exp\left\{ -C \sum_{k > N_\beta} \exp \left\{- \frac{z_k}{4} \right\}\right\}.  \]
This last sum goes to $0$ when $n \to \infty$, as the remainder of a converging series, such that 
\[ \prod_{k > N_\beta} \Pi[|f_k - f_{0,k}| \leq \delta_k] \geq e^{-n \varepsilon_n^2} ,\]
for any $n \varepsilon_n^2 \to \infty$.
\end{proof}

\subsection{Proof of Corollary \ref{thm : cor p to zero}} \label{proof : cor p to zero}
\begin{proof}
 Let us look at the first case where $\sigma_k = k^{-1/2 - \alpha}$ and $p_k = \log \log k / \log k$ for $k \ge 3$. % We check the compatibility condition \eqref{eq : compatib cond 1} trivially holds for $k=1,2$. Using the inequality $\log \log u - \log u + u +2u/\log u > 0$ available for all $u > 1$, we get, for $x >e$ and $u := \log x >1$, 
 % \[ \sqrt{\frac{\log u}{u}} \exp( u / \log u)  \geq \exp(-u/2). \]
 % Therefore, for all $k \geq 3$, we obtain the first compatibility condition \[\sqrt{p_k} \exp\{ 1/p_k\} \geq 1/\sqrt{k} \geq \sigma_k.\] 
 We check the compatibility condition \eqref{eq : compatib cond 2}, we compute, for all $ k > N_\beta$,
 \[ r_k^{p_k} = \left(\frac{ k^{-1/2 - \beta}}{\sigma_k}\right)^{p_k} = \left( k^{\alpha-\beta}\right)^{\frac{\log \log k}{\log k}} =  (\log k)^{\alpha-\beta}.\]
 Since $\alpha > \beta$, we have $(\log k)^{\alpha-\beta} \to \infty$, so that $r_k^{p_k} \geq 8$ for $k > N_\beta$ and $n$ large enough. Finally to check the summability conditions \eqref{eq : sum cond}, we compute
 \[ z_k = \frac{r_k^{p_k}}{p_k} = \frac{(\log k)^{1 +\alpha-\beta}}{\log \log k} ,\]
 with the appropriate adaptations whenever $k = 1,2$. Such that, taking $\eta =1 + \al - \beta >1$, we obtain
 \[ \sum_{k \leq N_{\beta}} z_k \leq C + \frac{(\log N_\beta)^{\alpha-\beta}}{\log\log3} \sum_{k =3}^{N_\beta} \log k \lesssim N_\beta (\log N_\beta)^{1 + \alpha -\beta} = N_\beta \log^\eta n.\]
 We also have that
 \[ \exp\{ -z_k/4 \} = \exp \left\{ - \frac14 \frac{(\log k)^{1 +\alpha-\beta}}{\log \log k}  \right\}\]
 is the general term of a summable series since $\alpha > \beta$.

Now looking at the second case where $\sigma_k = \exp\{-\log^{1 + \gamma}k\}$ and $p_k = c / \log^{1 + \ga}k$.  We check the compatibility condition \eqref{eq : compatib cond 2}, where
\[ r_k^{p_k} = \left(\frac{ k^{-1/2 - \beta}}{\sigma_k}\right)^{p_k} = \left({ k^{-1/2 - \beta }e^{\log^{1+\gamma}k}}\right)^{\frac{c}{\log^{1+\gamma} k}} = e^{-(1/2 + \beta)\frac{c}{\log^{\gamma}k}+c} \]
  which, provided $c > 2.1 >\log 8$, satisfies $r_k^{p_k} \geq 8$, for $k > N_{\beta}$ and $n$ large enough. Finally, we compute
 \[ z_k = \frac{r_k^{p_k}}{p_k} = e^{-(1/2 + \beta)\frac{c}{\log^{\gamma}k} +c}\,  \log^{1+\gamma}k.\]
 Therefore, taking $\eta = 1 + \gamma$, we have, for large enough $n$,
 \[ \sum_{k =3}^{N_\beta}z_k \lesssim N_\beta e^{c(1 -\frac{1/2 + \beta}{\log^{\gamma}N_\beta})}(\log N_\beta)^{1 + \gamma } \lesssim N_\beta \log^\eta n.\]
 To check the rest of Condition \eqref{eq : sum cond}, notice that, for large enough $k$, we have $\tilde{c} >0$, so that
 \[ \exp\{ -z_k /4 \} = \exp \left\{ - \frac1{4c}e^{c(1-\frac{1/2 + \beta}{\log^{\gamma}k})}\log^{1+\gamma} k\right\} \leq \exp \left\{- \tilde{c} \log^{1+\ga}k \right\},\]
 which is the general term of a converging series. 
 % \[ r_k^{p_k} = \left(\frac{ k^{-1/2 - \beta}}{\sigma_k}\right)^{p_k} = \left({ k^{-1/2 - \beta }e^{\log^{1+\gamma}k}}\right)^{\frac{\log \log k}{\log^{1+\gamma} k}} = e^{-(1/2 + \beta)\frac{\log \log k}{\log^{\gamma}k}}\  \log k\]
 % which satisfies $r_k^{p_k} \geq 8$, for $k > N_{\beta}$ and $n$ large enough. Finally, we compute
 % \[ z_k = \frac{r_k^{p_k}}{p_k} =e^{-(1/2 + \beta)\frac{\log \log k}{\log^{\gamma}k}}\,  \frac{\log^{2+\gamma}k}{\log \log k}.\]
 % Therefore, taking $\eta = 2 + \gamma$, we have, for large enough $n$,
 % \[ \sum_{k =3}^{N_\beta}z_k \lesssim N_\beta (\log 3)^{- \frac{1/2 + \beta}{\log^{\gamma}N_\beta}}(\log N_\beta)^{2 + \gamma } \lesssim N_\beta \log^\eta n.\]
 % To check the rest of condition \eqref{eq : sum cond}, notice that, for large enough $k$, we have 
 % \[ \exp\{ -z_k /4 \} = \exp \left\{ - \frac14e^{-(1/2 + \beta)\frac{\log \log k}{\log^{\gamma}k}}\frac{\log^{2+\gamma} k}{\log \log k}\right\} \leq \exp \left\{- \frac18 \frac{\log^{2 +\gamma}k}{\log\log k} \right\},\]
 % which is the general term of a converging series. 
\end{proof}

\subsection{Proof of Theorem \ref{thm:SNN} and Corollary \ref{cor:SNN}} \label{proof:SNN}
\begin{proof}
Corollary \ref{cor:SNN} is directly obtained by applying (the second part of) Lemma \ref{lem : vois computations} to lower bound $D_\rho(P_f^n,P_{f_0}^n)/n$ in Theorem \ref{thm:SNN}. We now prove Theorem \ref{thm:SNN}. In the random design regression setting of \eqref{def : random design}, Lemma \ref{lem : renyi from prior mass} and \ref{lem : vois computations} show that it is sufficient to get, for some constant $C >0$,
\begin{equation}\label{eq:stareq}
\Pi[ ||f - f_0||_\infty \le \veps_n ] \ge e^{-C n \veps_n^2}.
\end{equation}
Recall $f^\star_{N_\be}$ is the approximating network in \eqref{approxSNN} obtained by Lemma \ref{lem : approx shallow} and satisfies $|| f -f^\star_{N_\beta}||_\infty \le 2L N_\be^{-\be}$. For the considered $\varepsilon_n \gtrsim N_\beta^{-\beta} = \veps_n^*$, the triangle inequality implies, for some constant $d >0$ small enough, that $\Pi[ ||f - f_0||_\infty \le \veps_n ] \ge \Pi[ ||f -f^\star_{N_\be}||_\infty \le d\veps_n ].$

We now relabel the approximating network $f^\star_{N_\be}$, so that it has larger width $N_\alpha$ (corresponding to the prior network). To this end, we define the following partition of the index set $\{0,\dots,N_\alpha-1\}$ into
\begin{equation}\label{def : set SnTn}
    S_n=\Big\{k=l\frac{N_\alpha}{N_\beta}, \;l=0,1,\dots, N_\beta-1\Big\} \quad \text{and} \quad T_n=\{0,\dots,N_\alpha-1\} \setminus S_n
\end{equation}

We then let
\[f^\ast(x)=b^\ast+\sum_{k=0}^{N_\alpha-1}w_k^\ast(x-a_k^\ast)_+,\]
where $b^\ast=f_0(0)$,
\[
w_k^\ast =
\left\{
\begin{array}{ll}
0, & \text{if } k \in T_n, \\
w_{0;l},\ \text{for}\ l=\dfrac{kN_\beta}{N_\alpha},
& \text{if } k\in S_n.
\end{array}
\right.
\]
and
\[
a_k^\ast =
\left\{
\begin{array}{ll}
0, & \text{if } k\in T_n, \\[2pt]
l/N_\beta,\ \text{for } l=\dfrac{kN_\beta}{N_\alpha},
& \text{if } k\in S_n.
\end{array}
\right.
\]
It is straightforward to check that $f^\ast=f^\star_{N_\beta}$ (notice the different `star notation'), so that this is indeed a relabeling of the approximating shallow network, where only the $N_\beta$  weights indexed by $k\in S_n$ are nonzero; for these indices, note that $a_k^\ast=k/N_\alpha$. According to Lemma \ref{lem : approx shallow} and after the relabeling, the nonzero weights still satisfy $|w_{0}^\ast|\leq LN_{\beta}^{(1-\beta)_+}$ and $|w_{k}^\ast|\leq 2LN_\beta^{1-\beta}$ for $k=1,\dots, N_\alpha-1$   and $|b^\ast|\leq L$. With this relabeling, we now have 
\[ \Pi[ ||f - f_0||_\infty \le \veps_n ] \ge \Pi[ ||f -f^*||_\infty \le d\veps_n ].\]
Recalling $a_k=k/N_\alpha$ in the definition of the prior \eqref{SNN-prior}, notice that whenever $w_k^\ast\neq0$, it holds $a_k^\ast=a_k$, so that
\(\sum_{k=1}^{N_\alpha-1}w_k^\ast\big((x-a_k)_+-(x-a_k^\ast)_+\big)=0\)
and thus
\begin{align*}
    f-f^* = \sum_{k=0}^{N_\alpha-1}w_k(x-a_k)_+-\sum_{k=1}^{N_\alpha-1}w_k^\ast(x-a_k^\ast)_++(b-b^\ast) = \sum_{k=0}^{N_\alpha-1}(w_k-w_k^\ast)(x-a_k)_++(b-b^\ast).
\end{align*}
From the triangle inequality, follows
\[ \Pi[ ||f - f_0||_\infty \le \veps_n ] \ge \Pi \left( \big\|\sum_{k=0}^{N_\alpha-1}(w_k-w_k^\ast)(x-a_k)_+\big\|_\infty\leq d\veps_n/2 \,, \,|b-b^\ast|\leq d\veps_n/2 \right).\]
Since $\sup_{x \in [0,1]} (x - a_k)_+ \le 1$, the probability displayed above is further lower bounded by 
\[ \Pi \left( \sum_{k=0}^{N_\alpha-1}|w_k-w_k^\ast|\leq d\veps_n/2 \,, \,|b-b^\ast|\leq d\veps_n/2 \right).\]
Using independence, this probability can be split in the following product (recall $w_k^* = 0$ for $k\in T_n$), where we set $c=d/2$
    \begin{equation}\label{eq:decomposition} \underbrace{\prod_{k\in S_n}\Pi( |w_k-w_k^*| \leq c\varepsilon_n / N_\alpha)}_{I}\times \underbrace{\prod_{k\in T_n} \Pi( |w_k| \leq c\varepsilon_n / N_\alpha)}_{II}\times\underbrace{ \Pi( |b-b^\ast| \leq c\veps_n)}_{III} .\end{equation}

We first study term III, which (denoting the density of $\pi_b$ also by $\pi_b$), by symmetry, positivity and continuity, is lower bounded by $2rc\veps_n$ for $r=\min_{x\in[-2L,2L]}\pi_b(x)>0$. The latter lower bound vanishes much slower than $\exp(-n\veps_n^2)$, as long as $\veps_n\gtrsim n^{-s}$ for some $s<1/2$, as is the case for the considered $\veps_n\gtrsim \veps_n^\ast$.

We then study term II. Assuming $\veps_n/(\sigma_nN_\alpha)\to\infty$, the symmetry and tail assumptions \eqref{conds} and \eqref{condu} on $h$, allow us to lower bound as follows
\[ \Pi( |w_k| \le c\varepsilon_n / N_\alpha) =1-2\overline{H}(\frac{c\varepsilon_n}{\sigma_nN_\alpha})\ge  1 - 2d_0\exp\left\{- d_1(\frac{c\veps_n}{\sigma_nN_\alpha})^q\right\}.\] 
We can then lower bound the product {over indices in $T_n \subset\{0 , \dots , N_\al -1 \}$} as
\[II\ge \left(1-2d_0\exp\left\{- d_1'(\frac{\veps_n}{\sigma_nN_\alpha})^q\right\}\right)^{N_\alpha},\]
which, provided $\veps_n/(\sigma_n N_\alpha)\gtrsim \log^{1/q}n$, can be shown to remain bounded away from zero. Indeed, it is equivalent to show that the negative logarithm remains bounded, which can be verified under the previously mentioned condition, using the inequality $\log(1-x)\geq -x/\sqrt{1-x}$, for $0\le x<1$.

Finally for the term $I$, assuming without loss of generality (due to the symmetry of $h$) that $w_k^\ast>0$, using \eqref{condt} we have
\[\Pi( |w_k-w_k^*| \leq c\varepsilon_n / N_\alpha) {\ge }\frac{1}{\sigma_n}\int_{w_k^\ast}^{w_k^\ast+\frac{c\veps_n}{N_\alpha}}h(x/\sigma_n)dx\ge cc_0\frac{\veps_n}{\sigma_nN_\alpha}\exp(-c_1|w_k^\ast+c\veps_n/N_\alpha|^p\sigma_n^{-p}),\]
so that, since the cardinality of $S_n$ is $N_\beta$, 
\[ I = \prod_{k \in S_n}\Pi( |w_k-w_k^*| \le c\varepsilon_n / N_\alpha) \geq \left(c_0'\frac{\veps_n}{N_\alpha\sigma_n}\right)^{N_\beta}\exp\left\{- c_1  \sigma_n^{-p} \sum_{k\in S_n}|w_k^* + c\varepsilon_n / N_\alpha|^p \right\}.\]
We study the sum in the exponent. Using the available bounds for $w_k^\ast$,
combined with $|a+b|^p\leq|a|^p+|b|^p$ valid for any $p\in(0,1]$, we can bound the sum as
\begin{align}\label{sumbound}
\sum_{k\in S_n}|w_k^* + c\varepsilon_n / N_\alpha|^p&\lesssim |w_0|^p+\sum_{k\in S_n\setminus\{0\}}|w_k^\ast|^p+N_\beta\frac{\veps_n^p}{N_\alpha^{p}}\nonumber\\
&\lesssim N_\beta^{p(1-\beta)_+}+N_\beta^{1+(1-\beta)p}+N_\beta\frac{\veps_n^p}{N_\alpha^{p}}.
\end{align}
We next note that the third term is dominated by the second if and only if $\veps_n/N_\alpha\ll N_\beta^{1-\beta}.$ Since $N_\alpha>N_\beta$ (recall $\alpha>\beta$), for this condition to hold it suffices that $\veps_n\ll N_\beta^{2-\beta}$, which is always the case since $\beta\le2$. Hence the second term always dominates the third. Recalling $p\in(0,1]$, we compare the first and second terms to find that:
\begin{enumerate}
    \item {for $\beta\in(1,2]$}, we have $1+p-\beta p\ge (2-\beta)p\ge0$, hence the second term dominates the first and overall in the right hand side of the bound;
    \item {for $\beta\in(0,1]$}, we have $1+p-\beta p>p(1-\beta)_+$, hence again the second term dominates the first and overall in the right hand side of the bound.
\end{enumerate} 
For any $\beta\in(0,2]$ we thus get that 
\[I\ge\exp(-N_\beta(c_2+\log(\frac{N_\alpha \sigma_n}{\veps_n}))-c_3\frac{N_\beta^{1+(1-\beta)p}}{\sigma_n^p})\]
and combining with the bounds for the previous terms, we obtain that under the assumption 
\begin{equation}\label{bcond1}\veps_n/(\sigma_n N_\alpha)\gtrsim \log^{1/q}n,\end{equation}
\[\Pi( || f - f_0||_\infty \leq  \varepsilon_n) \geq\exp\left(-N_\beta\big(c_2+\log(\frac{N_\alpha \sigma_n}{\veps_n})\big)-c_3\frac{N_\beta^{1+(1-\beta)p}}{\sigma_n^p}-{c_4}-n\veps_n^2\right).\]
The latter lower bound, is in turn lower bounded by $\exp(-c_5n\veps_n^2)$ for a large enough constant $c_5>0$, provided the following hold
\begin{enumerate}
\item $N_\beta\lesssim n\veps_n^2$, or equivalently $\veps_n\gtrsim \veps_n^\ast$ which always holds;
\item $N_\beta\log(N_\alpha\sigma_n/\veps_n)\lesssim n\veps_n^2$ or equivalently 
\[\log(N_\alpha\sigma_n/\veps_n)\lesssim (\veps_n/\veps_n^\ast)^2,\]
which for $\sigma_n\le1$, $\veps_n\gtrsim n^{-1/2}$ and $N_\alpha$ as defined above, holds if,  for some $\delta>0$,
\begin{equation}\label{bcondaux}\veps_n\gtrsim \veps_n^\ast\log^{1/2+\delta}(n),\end{equation} 
\item $N_\beta^{1+(1-\beta)p}\lesssim \sigma_n^pn\veps_n^2$, or equivalently 
\begin{equation}\label{bcond2}\veps_n\gtrsim \veps_n^\ast\Big(\frac{N_\beta^{1-\beta}}{\sigma_n}\Big)^{p/2}.\end{equation}
\end{enumerate}
Hence, it suffices that $\veps_n$ satisfies \eqref{bcond1} and \eqref{bcond2} to obtain the prior mass bound \eqref{eq:stareq}.
For the oracle choice of $\sigma_n$: we optimize the choice of $\sigma_n$  based on \eqref{bcond1},\eqref{bcond2}, and then check that \eqref{bcondaux} also holds. Since \eqref{bcond1},\eqref{bcond2} imply that 
\begin{equation}\label{eq:oracle}\veps_n\gtrsim \{\sigma_nN_\alpha \log^{1/q}n\}\vee \{\veps_n^\ast\Big(\frac{N_\beta^{1-\beta}}{\sigma_n}\Big)^{p/2}\},\end{equation}
where the first term in the maximum improves with a faster decay of $\sigma_n$ while the second deteriorates, we choose $\sigma_n$ to balance the two terms, resulting in 
\[\sigma_n\asymp N_\alpha^{-\frac2{2+p}}N_\beta^{\frac{p}{2+p}-\beta}\log^{-\frac2{q(2+p)}}n.\]
This results in
\[\veps_n\gtrsim N_\alpha^{\frac{p}{2+p}}N_\beta^{\frac{p}{2+p}-\beta}\log^{\frac{p}{q(2+p)}}n=\veps_n^\ast(N_\alpha N_\beta)^{\frac{p}{2+p}}\log^{\frac{p}{q(2+p)}}n.\]
This $\veps_n$ also satisfies \eqref{bcondaux}, as required.

For the non-oracle choice of $\sigma_n$, choosing $\sigma_n=\veps^+_n/N_\alpha$, where $\veps^+_n=n^{-2/5}$ is the minimax rate for $\beta=2$ (that is, for the highest considered smoothness), we have that the inequality arising from the first term in the maximum \eqref{eq:oracle} becomes trivial and the admissible choices of $\veps_n$ are determined solely by the second term in the maximum. The resulting constraint is
\[\veps_n\ge \veps_n^\ast n^{\frac{p}2(\frac{1-\beta}{1+2\beta}+\frac25+\frac1{1+2\alpha})},\]
which as $p$ becomes smaller, approaches $\veps_n^\ast$.
\end{proof}

\subsection{Proof of Theorem \ref{thm:SNNvarp}}\label{proof:SNNvarp}
%\ma{Will we put this proof (and remarks that follow) to the supplement? If not, the remarks that follow should be placed in a new section B.4 in the supplement with title "SNN prior: adaptation with smaller widths and weaker conditions on the scalings"}
\begin{proof}
The contraction in $L_2(P_X)$--loss is directly obtained from the contraction in R\'enyi loss using Lemma \ref{lem : vois computations}. For the R\'enyi contraction result, the proof proceeds similarly to the proof of Theorem \ref{thm:SNN}, but with careful tracking of the dependence of the constants on $p$ using a simplification of the techniques employed for series priors with varying--$p$ tails. 

Indeed, up to \eqref{eq:decomposition} the proof is identical to the one of Theorem \ref{thm:SNN}, and so is the handling of term $III$. 

We study term $II$, the product over indices in $T_n$ (defined in \eqref{def : set SnTn}) where $w_k^* =0$. Let us recall that $\alpha=0$, so that $N_\alpha= N_0 =\kappa n$, $\kappa\in[1/\sqrt{2},\sqrt{2}]$ (recall that the definition of $N_0$ is $2^{m_0}$, where $m_0$ is the closest integer solution to $2^{m_0}=n$). We need to lower bound
\begin{align*}
    \Pi(|w_k|\leq c \varepsilon_n/n)=1-2\overline{H}_{p_n}(\frac{c\varepsilon_n}{\sigma_n N_0}),
\end{align*}
where $\overline{H}_{p_n}$ is the survival function of the $p_n$-exponential distribution defined in \eqref{def : p-exp dist}.
For $\sigma_n\lesssim n^{-s}, s>7/5$, it holds that for large enough $n$, $x_n:=\frac{c\varepsilon_n}{\sigma_n N_0} \ge\frac{c\varepsilon_n}{\sqrt{2}\sigma_n n}\ge1$ for any $\varepsilon_n\ge \varepsilon_n^+=n^{-2/5}$, where the latter is the minimax rate for the highest considered regularity of the truth $\beta=2$. In particular, for large enough $n$, $x_n\ge 1$ for $\varepsilon_n$ as in the statement. This allows us to use the bound on the $p_n$-exponential cumulative distribution function from Lemma \ref{lem : sandwich}. In addition, our choice $\sigma_n=n^{-t}$ with $t>2.5>7/5+\log(8)/2$ combined with the choice $p_n=2/\log{n}$, secure that for large $n$ it holds $x_n^{p_n}\ge8$, since $x_n\geq n^{\tau},$ with $\tau>\log{(8)}/2$. These considerations give that  
\begin{align}\label{eq:survbd} \overline{H}_{p_n}(x_n) &\lesssim\frac{e^{\frac1{p_n}}}{\sqrt{p_n}} x_n e^{- x_n^{p_n}/p_n} \lesssim \exp \left\{ \frac12 \log\frac1{p_n} + \frac1{p_n} + \log x_n - \frac{x_n^{p_n}}{p_n} \right\}\nonumber\\&\lesssim\exp \left\{ \frac2{p_n} -\frac{x_n^{p_n}}{2p_n} \right\} \lesssim \exp \left\{ -\frac{x_n^{p_n}}{4p_n} \right\},\end{align}
where in the top line we used the bound from Lemma \ref{lem : sandwich} and for the second line we first used $\frac12 \log\frac1{p_n} \leq \frac1{p_n}$ and $\log x_n = \frac1{p_n}\log x_n^{p_n} \leq \frac1{2p_n}x_n^{p_n}$, and then $x_n^{p_n}\ge8$.
Using $\log(1-2y)\ge -4y$, available for $y\le1/4$, and since based on the last bound $\overline{H}_{p_n}(x_n)\to0$, for large enough $n$, we get
\[1-2\overline{H}_{p_n}(x_n)=\exp\left(\log\left(1-2\overline{H}_{p_n}(x_n)\right)\right)\ge \exp\left(-c'\exp\left(-\frac{x_n^{p_n}}{4p_n}\right)\right),\] for some constant $c'>0$, which in turn gives
\begin{align*}
  II&\ge \prod_{k\in T_n}\exp\left(-c'\exp\left(-\frac{x_n^{p_n}}{4p_n}\right)\right)=\exp\left(-c'\sum_{k\in T_n}\exp\left(-\frac{x_n^{p_n}}{4p_n}\right)\right)\\
  &\ge\exp\left(-c'(\sqrt{2}n-N_\beta)\exp\left(-\frac{x_n^{p_n}}{4p_n}\right)\right).
\end{align*}
The latter remains bounded away from zero, since, by $x_n^{p_n}\ge8,$ we have $\frac{x_n^{p_n}}{4p_n}\ge \frac{2}{p_n}=\log{n}$.

We next study term $I$, the product over indices in $S_n$ (defined in \eqref{def : set SnTn}). Assume without loss of generality that $w_k^\ast>0$. The lower bound on the $p_n$--exponential density in Lemma $\ref{lem : sandwich}$ provides
\begin{equation*}
    \Pi \left(|w_k-w_k^\ast|\leq \frac{c \veps_n}{N_0}\right)\ge\frac1{\sigma_n}\int_{w_k^\ast}^{w_k^\ast+ \frac{c}{\sqrt{2}} \veps_n/n}h_{p_n} \left(\frac{t}{\si_n}\right)\,dt \gtrsim \frac{\sqrt{p_n}}{\sigma_n}e^{\frac1{p_n} - \frac{p_n}{12}}\int_{w_k^\ast}^{w_k^\ast+ \frac{c}{\sqrt{2}}\veps_n/n}e^{- \frac1{p_n}\frac{|t|^{p_n}}{\sigma_n^{p_n}}} \,dt.
\end{equation*}
Noticing that $1/p_n-p_n/12\ge c_0/p_n$ with $c_0=11/12$, we further bound
\begin{equation*}
    \Pi\left(|w_k-w_k^\ast|\leq \frac{c \veps_n}{N_0}\right)\gtrsim e^{c_0/p_n}\frac{\sqrt{p_n}\veps_n}{\sigma_n n}\exp \left\{-\frac{\sigma_n^{-p_n}}{p_n} \left|w_k^\ast+\frac{c \veps_n}{\sqrt{2}n}\right|^{p_n}\right\},
\end{equation*}
so that, using $|S_n| = N_\beta$, we get for some constant $c_1>0$
\[I = \prod_{k \in S_n} \Pi\left(|w_k-w_k^\ast|\leq \frac{c \veps_n}{N_0}\right) \gtrsim \Big(c_1\frac{\sqrt{p_n}e^{c_0/p_n}\veps_n}{\sigma_n n}\Big)^{N_\beta} \exp \left\{-\frac{\sigma_n^{-p_n}}{p_n} \sum_{k \in S_n}\left|w_k^\ast+\frac{c \veps_n}{\sqrt{2}n}\right|^{p_n}\right\}. \]
We study the sum, using $|a+b|^p\le |a|^p+|b|^p,$ for all $a, b\in\RR$ and the bounds on the weights $w_k^\ast$ established in the proof of Theorem \ref{thm:SNN}:
\begin{align*}
\sum_{k \in S_n}\left|w_k^\ast+\frac{c \veps_n}{\sqrt{2}n}\right|^{p_n}&\lesssim |w_0|^{p_n}+\sum_{k\in S_n\setminus\{0\}}|w_k^\ast|^{p_n}+N_\beta\frac{\veps_n^{p_n}}{n^{p_n}}\\
&\lesssim L^{p_n} N_\beta^{p_n(1-\beta)_+}+(2L)^{p_n}N_\beta^{p_n(1-\beta)+1}+\frac{N_\beta \veps_n^{p_n}}{n^{p_n}}\\
&\lesssim  N_\beta^{p_n(1-\beta)_+}+N_\beta^{p_n(1-\beta)+1}+\frac{N_\beta \veps_n^{p_n}}{n^{p_n}}.
\end{align*}
As in the proof of Theorem \ref{thm:SNN}, it is straightforward to check that the second term dominates in the right hand side, using $p_n\le1$ and $\beta\le2$. Hence, for some $c''>0$,
\begin{align*}I&\ge \exp\Bigg(-c''N_\beta\Big(1+\log\big(\frac{e^{-c_0/p_n}}{\sqrt{p_n}}\big)+\log\big(\sigma_n n/\veps_n\big)+\frac{\sigma_n^{-p_n}}{p_n}N_\beta^{(1-\beta)p_n}\Big)\Bigg)\\&\ge\exp\Bigg(-c''(N_\beta+\frac{\sigma_n^{-p_n}}{p_n}N_\beta^{1+(1-\beta)p_n})\Bigg),\end{align*}
where for the last bound we used that, under our assumptions on $p_n, \sigma_n, \veps_n$, the two logarithms in the exponent in the top line are negative. 

Combining the bounds on terms I, II and III, we get, for some constant $c'''>0$ and for sufficiently large $n$,
\[\Pi(||f-f_0||_{\infty}\leq \veps_n) \ge \exp\Bigg(-c'''(N_\beta+\frac{\sigma_n^{-p_n}}{p_n}N_\beta^{1+(1-\beta)p_n})\Bigg), \]
which in turn is lower bounded by $\exp(-Cn\veps_n^2)$ for large enough $C>0$, provided
\begin{enumerate}
    \item $N_\beta\lesssim n\veps_n^2$ or equivalently $\veps_n\gtrsim \veps_n^\ast$;
    \item it holds \[\frac{\sigma_n^{-p_n}}{p_n}N_\beta^{1+(1-\beta)p_n}\lesssim n\veps_n^2\] or equivalently
    \[\veps_n\gtrsim \veps_n^\ast \frac{n^{(\frac{1-\beta}{1+2\beta}+t)p_n/2}}{\sqrt{p_n}}.\]
\end{enumerate}
The rate $\veps_n\ge\veps_n^\ast\sqrt{\log{n}}$ in the statement trivially satisfies the first condition, while for the second one, under our assumptions  $n^{(\frac{1-\beta}{1+2\beta}+t)p_n/2}$ is bounded, so it is again satisfied.
\end{proof}
\begin{rmk} \label{rem:SNN-other}
It is easy to verify that the proof of Theorem \ref{thm:SNNvarp} goes through as well with minor modifications with a pair $(\sigma_n,p_n)$ similar in spirit to the second choice in Corollary~\ref{thm : cor p to zero}, namely $\sigma_n=\exp(-a(\log{n})^{1+\ga})$ and $p_n=b/(\log{n})^{1+\ga}$, for $\ga>0$ and $a,b>0$ such that $ab>\log{8}$. Indeed, the condition $1/p_n-p_n/12\ge c_0/p_n$ used in the proof above is also satisfied for this choice. The only difference with the proof above is that term I gives a slightly different condition: the final rate $\veps_n$ should be chosen such that $\veps_n\geqa \veps_n^* (\log{n})^{(1+\ga)/2}$, which again gives the minimax rate up to a (slightly different) logarithmic factor. 
\end{rmk}
\begin{rmk}\label{rem:SNN-varp}
Theorem \ref{thm:SNNvarp} is in fact proved for $\sigma_n=n^{-t}$ under a slightly weaker condition on $t$, namely $t>7/5+\log{(8)}/2$. While this result is very attractive from the theoretical point of view since it leads to full adaptation (up to $\sqrt{\log{n}}$) of pseudo-posteriors over $\beta$-H\"older spaces with $\beta\in(0,2]$, in practice, the use of a neural network with width $n$ can become expensive for large sample sizes. To address this, one can consider networks of width $N_\alpha$, for some fixed $0\le\alpha<2$ (recall $N_\alpha\asymp n^{\frac{1}{1+2\alpha}}$, hence the choice $\alpha=1/2$ would lead to $\sqrt{n}$ width). In that case, with the same proof techniques, one can show that for scalings $\sigma_n=n^{-t}, \;t>2/5+\frac1{1+2\alpha}+\log(8)/2$ and tails $p_n=1$ for $n\in\{1,\dots, \lceil e^{2(1+2\alpha)}\rceil-1\}$ and  $p_n=2(1+2\alpha)/\log{n}$ for $n\ge \lceil e^{2(1+2\alpha)}\rceil$, the prior mass bound underlying the proof is satisfied for $\beta\in[\alpha,2]$, thus (near) adaptation is again achieved, however in this $\alpha$-restricted range. Aside from the computational benefits due to the smaller width, the scalings $\sigma_n$ and the tails $p_n$ are larger compared to the ones in Theorem \ref{thm:SNNvarp}, which can also be advantageous during posterior sampling. {In addition, the constant $\log(8)/2$ appearing in the scaling $\sigma_n$ results in scalings which for moderately large sample sizes $n$ (such as 400 and 4000 that we use in our simulations in  Section \ref{sec:simulations}) are too small and lead to too smooth posteriors. An inspection of the proof, shows that one can use a smaller constant, say $\tau>0$, in the condition on $t$, $t>2/5+\frac1{1+2\alpha}+\tau$, which needs to be counteracted by a larger constant say $g>0$ in the numerator of $p_n$, $p_n=g/\log{n}$, without this affecting the rate. In fact, choosing $g=g_n=\log\log{n}$ and $\tau=\tau_n\to0$ so that $g_n\tau_n\to\infty$ and $g_n/e^{\tau_n g_n}\to0$, leads to the same rates up to logarithmic terms. For further details on these considerations see  Remark \ref{rem:SNN-varptau} below. In the simulations Section \ref{sec:simulations} %which includes a simulation study, 
we used $\tau=0$ and $g=1$, which seems to work well for all considered sample sizes.}
\end{rmk}

\begin{rmk}\label{rem:SNN-varptau}
%\ma{S: more careful calculations: can make the log(8)/2-term in the scaling $\sigma_n$ arbitrarily small at the expense of increasing the constant 2 in the definition of $p_n=2/\log(n)$. In fact can choose $\sigma_n=n^{-t}$, with $t>7/5+\tau_n$ for say $\tau_n=1/\sqrt{\log\log{n}}$ and $p_n=g_n/\log{n}$. For the last line of the proof to remain valid, $g_n$ needs to be of $\log\log{n}$ type (i.e. not $\log^s{n}$ for $s<1$), and there will be some extra logs in the final rate $\veps_n$. This perhaps explains why even for relatively large $p$ performance is not bad?}
We provide some details on the modifications to the proof of Theorem \ref{thm:SNNvarp}, needed to enable the use of smaller exponents $t$ in $\sigma_n=n^{-t}$, as discussed in Remark~\ref{rem:SNN-varp}.
Set $p_n=g/\log{n}$ and $\sigma_n=n^{-t}$ for $t>7/5+\tau$, $\tau>0$. Then $x_n\ge n^{-7/5+t}=n^{\tau}$ and $x_n^{p_n}\ge n^{\frac{\tau g}{\log{n}}}= e^{\tau g}$. To bound the survival function as in \eqref{eq:survbd}, it suffices to establish a bound of the form $2+\log{x_n^{p_n}}-x_n^{p_n}\leq -m x_n^{p_n}$ for some $m>0$ which is such that $\exp(-mx_n^{p_n}/p_n)\le n^{-4/5}$, so that term $II$ is lower bounded by something of lower (or same) order as $\exp{(-n\veps_n^2)}$, for all $\veps_n\ge \veps_n^+=n^{-2/5}$. %\ma{(we in fact only need this to hold for slightly larger $\veps_n$, due to the logarithm in our rates, but there is no real benefit from this extra restriction)}. 
Equivalently, for $m$ we need $m\ge 4g/(5e^{\tau g})$. To get a bound of the form above, we can use the inequality $\log{x}\leq \frac{\log{\lambda}}{\lambda} x$ valid for $x\ge\lambda\ge e$, to get that it suffices that 
\[2+(\log{(\lambda})/\lambda-1)x_n^{p_n}\leq -mx_n^{p_n}\]
or equivalently \begin{equation}\label{eq:etg}e^{\tau g}\ge \frac{2}{1-\frac{\log{\lambda}}{\lambda}-m},\end{equation}
where $e^{\tau g}\ge \lambda$, $1-\log{\lambda}/\lambda> m\ge 4g/(5e^{\tau g})$.
For example, choosing $\tau, g$ so that $e^{\tau g}=8$, we can get a bound with $g/10\le m\le 3/4-\log{8}/8\approx 0.49$ so that any $g\le 4.9$ works. The choice of $g$ then determines $\tau$ via $\tau=\log{8}/g$. This does not affect the final rate, since $\frac{n^{(\frac{1-\beta}{1+2\beta}+t)p_n/2}}{\sqrt{p_n}}$ remains of order $\sqrt{\log{n}}$ as in the end of the proof of Theorem \ref{thm:SNNvarp}.

Another possibility is to choose $g=g_n=\log\log{n}$ and $\tau=\tau_n\to0$ such that $g_n\tau_n\to\infty$ and $g_n/e^{\tau_n g_n
}\to0$, in which case for any $\lambda>0$ for sufficiently large $n$ we have $x_n^{p_n}\ge e^{\tau_n g_n}\ge \lambda$, and it is straightforward to see that any $m<1$ works, since \eqref{eq:etg} is satisfied as long as $m<1-\log{\lambda}/{\lambda}$, where $\lambda$ can be chosen as large as we wish, and $m\ge 4g_n/(5e^{\tau_n g_n})$ is trivial by the condition $g_n/e^{\tau_ng_n}\to0$. With these choices the final rate is affected, since at the end of the proof there is the requirement
\[\veps_n\gtrsim \veps_n^\ast \frac{n^{(\frac{1-\beta}{1+2\beta}+t)p_n/2}}{\sqrt{p_n}},\]
where ${n^{(\frac{1-\beta}{1+2\beta}+t)p_n/2}}\le e^{r g_n}=(\log{n})^r$, for some $r\ge (\frac{1-\beta}{1+2\beta}+t)/2$, recalling that $t\to 7/5$ since $\tau_n\to0$. Thus $\veps_n$ needs to satisfy \[\veps_n\ge \veps_n^\ast (\log{n})^{1+r}/\log\log{n}.\]

Finally, other possibilities such as $p_n=2/\log^s{n}$ and $\tau_n=\log^{s-1}n$ can be treated similarly.
\end{rmk}

\section{Additional results} \label{secapp:add}

\subsection{Handling other statistical models}
\label{app:othermodels}

We briefly explain how to derive, from the prior mass results obtained in the present paper, contraction rates in other statistical models (above we restricted for simplicity to regression models).  
Suppose for instance to fix ideas that one wishes to  derive a posterior contraction rate for the neural network (log)--priors of Section \ref{sec : SNN} in density estimation. This is done in a similar way as for the models investigated in \cite{AC}. For instance, results for nonparametric binary classification can be derived similarly as well, or in any model where one can appropriately link Kullback--Leibler neighborhoods to $\|\cdot\|_\infty$--neighborhoods.

In the case of density estimation, one observes iid data $X_1,\ldots,X_n$ of unknown density $f$ on $[0,1]$. Given one of the neural network priors of Section \ref{sec : SNN}, one can use the exponential transform $g \to e^{g}/\int_0^1 e^{g}$ to induce a prior $\Pi$ on density functions. 

 Using the lower bound on the prior mass in the $\|\cdot\|_\infty$--norm investigated in the proofs of Theorems \ref{thm:SNN}--\ref{thm:SNNvarp}, one can use a generic result on $\rho$--posterior contraction such as Theorem 3 in \cite{ltcr23}. The latter requires to bound from below the prior mass of a Kullback--Leibler type neighborhood: in density estimation for log-density priors this mass can be bounded from below by the prior mass of a $\|\cdot\|_\infty$--neighborhood (see e.g. Lemma C.2 in \cite{AC}). Theorem 4.1 in \cite{ltcr23} then gives posterior contraction in terms of the $\rho$--R\'enyi divergence: in density estimation the latter is bounded from below by a constant times the $L^1$--norm: this gives posterior contraction in the $L^1$--norm in the density estimation model as desired. %We refer to \cite{AC} for more details in this and other models. 

\subsection{Series prior: Upper bound for Sobolev truths}
Define the $L_2$--Sobolev--type ball, for any $\beta,L >0$, as
\[\cS^\beta(L) := \left\{ f=(f_k) \, : \, \sum_{k \geq 1} k^{2\beta} |f_k|^2 \leq L^2 \right \}.\]
\begin{thm}\label{thm : conc series sobolev}
    Let $p \in (0,1)$ and let $\al >\be > 0$. Suppose $f_0 \in \cS^\beta(L)$ for some $L > 0$ and assume $X^{(n)} \sim P_{f_0}^{(n)}$ from the model \eqref{def : gwn}. Then, for any $\rho \in (0,1)$, starting from the prior $\Pi = \Pi(p,\alpha)$ defined in \eqref{def : prior series}--\eqref{def : alphsig}, as $n \to \infty$, we have
    \[ E_{f_0} \Pi_\rho \left[ ||f-f_0||_2 \le M \varepsilon_n \given X\right] \to 1, \]
    where $\veps_n$ is given in \eqref{def : rate} and $M>0$ is a large enough constant.
\end{thm}

\begin{proof}
    As in the proof of Theorem \ref{thm : conc series}, we recall that it suffices to show, for some $C,D >0$
    \[ \Pi\left[ \|f-f_0\|_2\le D\veps_n \right] \ge \exp(-Cn\veps_n^2). \]
    Recall $N_\gamma$ is the closest integer to $n^{1/(1 +2\beta + p(\al - \beta))}$. From $f_0 \in \cS_\beta(L)$, there exists $D>0$, such that, as $n$ gets large enough
    \[ \{ f\, : \, 1 \leq k \leq N_\gamma,\, |f_{0,k}-f_k| \leq 1/\sqrt{n} \} \cap \{f \, : \, \forall k > N_\gamma, \, |f_k| \leq  k^{-\beta - 1/2}\} \subset \{ f \, : \,  ||f-f_0||_2^2 \leq D\varepsilon_n^2 \}.\]
    Indeed, suppose $f$ belongs to the left hand side set displayed just above. Using first Parseval's equality and $(a-b)^2 \leq 2a^2 + 2b^2$, we have a constant $D >0$, such that, for $n$ large enough
    \begin{align*}
        ||f -f_0||_2^2 &\leq \sum_{k =1}^{N_\gamma}|f_k-f_{0,k}|^2 + 2\sum_{k >N_\gamma}|f_k|^2 + 2 \sum_{k >N_\gamma}|f_{0,k}|^2 \\
        &\leq \frac{N_\gamma}{n} + 2 \sum_{k >N_\gamma}k^{-2 \beta -1} + 2\sum_{k >N_\gamma}k^{-2\beta}k^{2\beta}|f_{0,k}|^2 \\
        &\leq \frac{N_\gamma}{n} + \beta^{-1}N_\gamma^{-2\beta} + 2L^2N_\gamma^{-2\beta}
        \leq D\varepsilon_n^2.
    \end{align*}
    Using independence of $(f_k)$ under the prior, leads to
    \[ \Pi[ ||f-f_0||_2^2 \leq D \varepsilon_n^2] \geq \prod_{k=1}^{N_\gamma} \Pi[ |f_k - f_{0,k}| \leq 1/\sqrt{n}] \times \prod_{k \geq N_\gamma} \Pi[ |f_k| \leq k^{-1/2 - \beta}].\]
    For the first product, we lower bound the integrand with its minimal value
    \[ \Pi[ |f_k - f_{0,k}| \leq 1/\sqrt{n}] \geq c_0 \int_{\sigma_k^{-1}(f_{0,k} - 1/\sqrt{n})}^{\sigma_k^{-1}(f_{0,k} + 1/\sqrt{n})} \exp( -c_1 |x|^p) \, dx \geq \frac{2 c_0}{\sigma_k \sqrt{n}}\exp\left(- c_1 \frac{|f_{0,k} + 1/\sqrt{n}|^p}{\sigma_k^p}\right).\]
    Using $|a + b|^p \leq |a|^p + |b|^p$, available for $p \in (0,1)$, we get
    \[ \prod_{k=1}^{N_\gamma} \Pi[ |f_k - f_{0,k}| \leq 1/\sqrt{n}] \geq \prod_{k=1}^{N_\gamma} \frac{2c_0}{\sigma_k \sqrt{n}} \times \exp\left( -c_1 \sum_{k=1}^{N_\gamma}\frac{|f_{0,k}|^p + n^{-p/2}}{\sigma_k^p}\right).\]
    Since $p < 1$, one notes that $\ga\le \al$ by definition, so that $\sigma_k^{-1}\ge k^{1/2+\ga}$.  Lemma \ref{lemlog} (applied with $\al$ therein replaced by $\ga$) then implies
\[ \prod_{k=1}^{N_\ga} 
\frac{2c_0}{\sqrt{n}\si_k} 
\ge e^{-(1/2+\ga-\log(2c_0))N_\ga}\ge e^{-C_0N_\ga}.
\]
Furthermore, noting that $n^{-p/2} \lesssim k^{-1/2 -\beta}$, for $k \leq N_\gamma$ and $\al > \beta$, we get
\begin{align*}
\sum_{k=1}^{N_\gamma}\frac{|f_{0,k}|^p + n^{-p/2}}{\sigma_k^p} &\lesssim \sum_{k=1}^{N_\gamma} k^{p (\al + 1/2)} |f_{0,k}|^p +  \sum_{k=1}^{N_\gamma}k^{p (\al -\beta) }\lesssim \sum_{k=1}^{N_\gamma} k^{p (\al + 1/2)} |f_{0,k}|^p + N_\gamma^{p(\al - \beta) + 1}.
\end{align*}
To take care of the last sum we use Holder's inequality with exponents $a = 2/p > 1$ and $b = 2/(2-p)$ (such that $1/a + 1/b = 1$), as well as $f_0 \in \cS_\beta(L)$, to obtain
\[\sum_{k=1}^{N_\gamma} k^{p (\al + 1/2)} |f_{0,k}|^p \leq N_\gamma^{p(\al - \be +1/2)}  \sum_{k=1}^{N_\gamma} k^{p \be}|f_{0,k}|^p \leq N_\gamma^{p(\al - \be +1/2)} N_\ga^{1-p/2} L^p \lesssim N_\ga^{p(\al -\be) +1}.  \]
Finally, noting that for $n$ large enough, $N_\ga \leq N_\ga^{p(\al-\be)+1} \le n \varepsilon_n^2$, we obtain large enough constants $C_1,\tilde{C}_1 > 0$, such that
\[\prod_{k=1}^{N_\gamma} \Pi[ |f_k - f_{0,k}| \leq 1/\sqrt{n}] \geq \exp( -C_0 N_\ga - C_1 N_\ga^{p(\al - \be) +1}) \geq \exp(\tilde{C}_1 n \varepsilon_n^2). \]
For the second product, since $\zeta_k$'s are symmetric, have density $h$ and survival function $\overline{H}$, 
\[ \prod_{k \geq N_\gamma} \Pi[ |f_k| \leq k^{-1/2 - \beta}] = \prod_{k \geq N_\gamma} \Pi[ |\zeta_k| \leq k^{\al - \beta}] = \prod_{k \geq N_\gamma} (1 - 2 \overline{H}( k^{\al - \beta})).\]
Using Condition \eqref{condu}, and the inequality $\log(1-2x)\ge -4x$, valid for $x\in[0,1/4]$, we obtain
\[ \prod_{k \geq N_\gamma} (1 - 2 \overline{H}( k^{\al - \beta})) \geq \exp \sum_{k>N_\ga} \log \left( 1-2d_0e^{-d_1k^{q(\al-\be)}}\right) \ge \exp \{ -4d_0\sum_{k>N_\ga} e^{-d_1k^{q(\al-\be)}} \}.\]
Since the series $\sum_k e^{-c k^\delta}$ converges for any given constants $c, \delta>0$, one deduces that the last display converges to $1$ as $n\to\infty$ and in particular is bounded from below by $1/2$ for $n$ large enough. Gathering the previous bounds for both products provides a constant $C >0$, such that for $n$ large enough, one obtains
\[ \Pi[ ||f - f_0||_2^2 \le D \varepsilon_n^2] \geq \exp(-C n\varepsilon_n^2). \qedhere\]
\end{proof}

\subsection{Series prior: Lower bound (case $p=1$)} \label{prooflbp1}

The following is the special case of Theorem \ref{thm : lower bound} where $p=1$ (Laplace priors on coefficients).
\begin{thm} \label{lb1}
Let $\alpha>\beta>0$. Suppose the data $X$ follows the white noise model for some true function $f_0$, and let the prior on $f$ as in \eqref{def : prior series} be defined by taking $\zeta_k$s to be drawn iid from a standard Laplace distribution. There exists a function $f_0\in\cF(\be,L)$ such that, if 
\[ \zeta_n = \veps_n(1,\al,\be) = n^{-\frac{\be}{\al+\be+1}}, \]
then for any $\rho\in(0,1]$, for $m>0$ small enough, as $n\to \infty$,
\[ E_{f_0}\Pi_\rho[\|f-f_0\|_2 < m\zeta_n \given X] \to 0.\]
\end{thm}
\begin{proof}
To simplify the notation, we give the proof first for the standard posterior $\rho=1$.  
Let us choose $f_0$ as the function in $\cF(\be,L)$ defined through its basis coefficients by $f_{0,k}=L k^{-1/2-\be}$ for $\be>0$. 
Let us recall the definitions of $\gamma$ in \eqref{def : gamma} and $N_r$ in \eqref{def : Ncutoff} and let us set $p=1$. With this notation we have $\zeta_n=\veps_n(1,\al,\be)=N_\ga^{-\be}$ by definition. Denoting, for any square-integrable function $g$, by $\|g\|_{N_\ga}^2=\sum_{k=1}^{N_\ga} g_k^2$, it is enough to prove, for small $m>0$ to be chosen and $n_\ga=d N_\ga$ for some small enough constant $d$ to be chosen below, that
\[ E_{f_0}\Pi[\|f-f_0\|_{n_\ga} \ge m\zeta_n \given X] \to 1. \]
Under the prior distribution, coefficients $f_k$ have distribution $\sigma_k\text{Lap}(1)$. 
By writing the Laplace distribution as a mixture of two exponential distributions (one for the positive part, one for the negative part), and using Bayes' formula, one can write the posterior distribution $\cL(f_k\given X)$ of the $k$th coefficient $f_k$ as a mixture 
\begin{equation}\label{explicitlap}
\cL(f_k\given X) = w^+_k \cN(\mu_k,1/n)_+ + (1-w^+_k) \cN(\nu_k,1/n)_-,
\end{equation}
where $\cN(\mu,\si^2)_+$ denotes the distribution of $Z\vee 0$ if $Z\sim \cN(\mu,\si^2)$ and $\cN(\mu,\si^2)_-$ the one of $Z\wedge 0$, and where we have set
\begin{align}
\mu_k & = X_k-\frac{1}{n\sigma_k}, \qquad 
\nu_k  = X_k +\frac{1}{n\sigma_k} \label{munu} \\
w_k^+ & = \frac{ e^{n\mu_k^2/2} \Phi(\sqrt{n}\mu_k) }{ 
e^{n\nu_k^2/2} \bar\Phi(\sqrt{n}\nu_k) + e^{n\mu_k^2/2} \Phi(\sqrt{n}\mu_k)},
\label{dobk}
\end{align}
where $\Phi$ is the distribution function of the standard normal distribution and $\bar\Phi=1-\Phi$.

%Write lemma: WHP $w_{k,+}\to 1$ for $k<N_\ga$ and the chosen $f_0$.

The triangle inequality gives, denoting $\mu=(\mu_k)$ for $\mu_k$ as in \eqref{munu}, that $\|f_0- \mu\|_{n_\ga}
\le \|f_0-f\|_{n_\ga}+\|f- \mu\|_{n_\ga}$. This implies
\[ 
\Pi[\|f-f_0\|_{n_\ga} \ge m\zeta_n \given X] 
\ge \Pi[\|f-\mu\|_{n_\ga} \le m\zeta_n \given X]\cdot \1\{ \|f_0- \mu\|_{n_\ga} \ge 2m\zeta_n \}. 
 \]
It now suffices to show that each term of the product of the right hand side of the last display goes to $1$ in probability under $P_{f_0}$. 

Starting with the indicator, its expectation under $P_{f_0}$ equals, denoting $X=(X_k)$, and using $\|\mu-f_0 \|_{n_\ga}\ge  \|\mu-X\|_{n_\ga}- \| X -f_0 \|_{n_\ga}$ by the triangle inequality,
\[ P[ \| \mu-f_0 \|_{n_\ga} \ge 2m\zeta_n ] \ge 
P[ \| X -f_0 \|_{n_\ga} \le m\zeta_n ] \cdot\1\{ \| \mu-X\|_{n_\ga} \ge 3m\zeta_n \}.
\] 
By definitions of $\mu$ and $(\sigma_k)$, we have, for a constant $C_\al>0$ depending only on $\al$,
\[ \| \mu-X\|_{n_\ga}^2 = \sum_{k=1}^{n_\ga} \frac{\sigma_k^{-2}}{n^2} \ge C_\al \frac{n_\ga^{2+2\al}}{n^2}
= C' \zeta_n^2.
\] %\ma{do we need a new constant here because of $d$ in definition $n_\gamma=dN_\gamma$?}
for $C'=C'(\al,d)$ and $d$ the constant such that $n_\gamma=dN_\gamma$. 
Hence for $m$ small enough so that $3m<C'$, the indicator in the last but one display equals $1$.  
Also, $\|X-f_0\|_{n_\ga}^2=\sum_{k=1}^{n_\ga} \veps_k^2/n$. The later quantity has expectation $n_\ga/n=n^{-(\al+\be)/(1+\al+\be)}=o(\zeta_n^2)$ since $\al>\be$ by assumption. A standard concentration argument (e.g. using Tchebychev's inequality, or a more precise exponential concentration bound for the $\chi^2$ distribution) then gives 
$P[ \| X -f_0 \|_{n_\ga} \le m\zeta_n ]=1+o(1)$. 

To conclude the proof, it suffices to check that $\Pi[\|f-\mu\|_{n_\ga} \le m\zeta_n \given X]$ goes to $1$ in probability under $P_{f_0}$. %By Lemma \ref{lemdob}, denoting $\Pi_+:=\cN(\mu_k,1/n)_+$, it suffices to check that $\Pi_+[\|f-\mu\|_{n_\ga} \le m\zeta_n \given X]$ goes to $1$ in probability 
By Markov's inequality, for $\cA_n$ the event as in \eqref{defan},
\begin{align*} 
\Pi[& \|f-\mu\|_{n_\ga} > m\zeta_n \given X] \1_{\cA_n} \le \frac{1}{(m\zeta_n)^2} \int \|f-\mu\|_{n_\ga}^2 d\Pi(f\given X)  \1_{\cA_n}\\
 & \le \frac{ \1_{\cA_n}}{(m\zeta_n)^2} \sum_{k=1}^{n_\ga} \left[w_k^+ \int (f_k-\mu_k)^2d\cN(\mu_k,1/n)_+(f_k)
 +  (1-w_k^+) \int (f_k-\mu_k)^2d\cN(\nu_k,1/n)_{-}(f_k) \right]\\
 &  \le \frac{ \1_{\cA_n}}{(m\zeta_n)^2} \sum_{k=1}^{n_\ga} \left[ \int (f_k-\mu_k)^2d\cN(\mu_k,1/n)_+(f_k)
 +  c_1e^{-c_2n^{\frac{\al-\be}{2\ga+1}}} \int (f_k-\mu_k)^2d\cN(\nu_k,1/n)_{-}(f_k) \right],
\end{align*}
where for the last inequality we use $w_k^+\le 1$ and the uniform bound on $(1-w_k^+)$ obtained in Lemma \ref{lemdob}, and $2\ga=\al+\be$. 
The first integral on the last line can be written 
\[ \int (u-\mu_k)^2 \sqrt{n}\phi(\sqrt{n}(u-\mu_k)) \1\{u\ge 0\} du/\int \sqrt{n}\phi(\sqrt{n}(u-\mu_k)) \1\{u\ge 0\} du. \]% d\cN(\mu_k,1/n)_+(f_k)=E[(Z_+- \]
The denominator equals $\bar\Phi(-\sqrt{n}\mu_k)=\Phi(\sqrt{n}\mu_k)$. By Lemma \ref{lemdob}, on the event $\cA_n$ (see \eqref{defan}) of overwhelming probability, it holds $\sqrt{n}\mu_k\ge \sqrt{n}f_{0,k}/4$ %\ma{should it be $f_{0,k}/4$ on rhs here?}
, which is bounded away from $0$ for $k\le n_\ga$, so that $\Phi(\sqrt{n}\mu_k)\ge 1/2$ for such $k$'s. One can then bound the numerator in the last display from above by 
\[ \int (u-\mu_k)^2 \sqrt{n}\phi(\sqrt{n}(u-\mu_k)) du = \frac{1}{n} \int u^2 \phi(u)du=1/n. \]
We now bound the second integral in the former display on $\cA_n$, using first $(f_k-\mu_k)^2\le 2f_k^2+2\mu_k^2\le 2f_k^2+2\nu_k^2$ using $0\le \mu_k\le\nu_k$ on $\cA_n$  and then 
\[ \int f_k^2d\cN(\nu_k,1/n)_{-}(f_k) \le E[ (Z+\nu_k)_-^2]\le 2E[Z^2]+2\nu_k^2\le 2/n+2\nu_k^2,\]
with $Z\sim \cN(0,1/n)$, where we have used $E[Y_-^2]\le E[Y^2]$ for any variable $Y$ (here $Y=Z+\nu_k$). Also,
\[ E_{f_0}[\nu_k^2] \le 2 E_{f_0}[X_k^2] + 2/(n\sigma_k)^2\le 4f_{0,k}^2+2/n+ 2/(n\sigma_k)^2.\]
This leads to $\sum_{k=1}^{n_\ga} E_{f_0}[\nu_k^2] \leqa C + n_\ga/n + \zeta_n^2$, where we have used that $f_0$ is squared-integrable and that $\sum_{k=1}^{n_\ga}1/(n\sigma_k)^2\leqa \zeta_n^2$. By gathering the previous bounds one obtains
\begin{align*}
E_{f_0}\left[\Pi[ \|f-\mu\|_{n_\ga} > m\zeta_n \given X]\1_{\cA_n} \right]
& \le \frac{1}{(m\zeta_n)^2} \left[2\frac{n_\ga}{n} +  c_1e^{-c_2n^{\frac{\al-\be}{2\ga+1}}} C\left\{ 1 + n_\ga/n+\zeta_n^2 \right\} \right] \\
& \leqa (n_\ga/n)\zeta_n^{-2}+\zeta_n^{-2} e^{-c_2n^{\frac{\al-\be}{2\ga+1}}} =o(1),
\end{align*}
using that $(n_\ga/n)\zeta_n^{-2}=o(1)$ by the definitions of $n_\ga, \zeta_n$ and using $\al>\be$ by assumption.  Since $P_{f_0}[\cA_n]=1+o(1)$ by Lemma \ref{lemdob}, this implies  $E_{f_0}\left[\Pi[ \|f-\mu\|_{n_\ga} > m\zeta_n \given X] \right]=o(1)$ as desired, which concludes the proof for the usual posterior.

For the $\rho$--posterior with $\rho<1$, the proof is mostly the same:  first one notes that 
the $\rho$--posterior on the $k$th coordinate is the mixture distribution
\begin{equation*}%\label{explicitlap}
 w^+_k \cN(\mu_k,1/n')_+ + (1-w^+_k) \cN(\nu_k,1/n')_-,
\end{equation*}
where we have set $n'=n\rho$ and with the updated definitions
\begin{align*}
\mu_k & = X_k-\frac{1}{n' \sigma_k}, \qquad 
\nu_k  = X_k +\frac{1}{n' \sigma_k}  \\
w_k^+ & = \frac{ e^{n'\mu_k^2/2} \Phi(\sqrt{n'}\mu_k) }{ 
e^{n'\nu_k^2/2} \bar\Phi(\sqrt{n'}\nu_k) + e^{n'\mu_k^2/2} \Phi(\sqrt{n'}\mu_k)}.
\end{align*}
The above proof and that of Lemma \ref{lemdob} for $\rho=1$ both go through with $n$ replaced by the updated `effective' sample size $n'=n\rho$; since $\rho\in(0,1]$ is fixed, this only changes the constants in the obtained rates, which concludes the proof.
\end{proof}

\begin{lem} \label{lemdob}
Let $f_0$ be defined by $f_{0,k}=L k^{-1/2-\be}$ for $\be>0$ and let $\sigma_k=k^{-1/2-\al}$ for $\al>0$. 
\begin{enumerate}
\item Let $\cA_n$ be the event defined by, for $\mu_k$ 
 as in \eqref{munu} and $n_\ga=d N_\ga$, 
\begin{equation} \label{defan}
\cA_n=\left\{ \mu_k \ge f_{0,k}/4, \ \ \ \text{for all }\ k=1,\ldots, n_\ga \right\}.
\end{equation}
Then, for a small enough constant $d>0$ above, one has $P_{f_0}[\cA_n]=1+o(1)$ as $n\to\infty$. 

\item There exist constants $c_1, c_2>0$ such that, for $w_k^+$ as in \eqref{dobk}, on the event $\cA_n$ as in \eqref{defan},
\[ \max_{1\le k\le n_\ga} (1-w_k^+) \le c_1e^{-c_2 n^{(\al-\be)/(\al+\be+1)}}. \]
%for any $k\in\{1,\ldots n_\ga\}$ on an event $\cA_n$ such that $P_{f_0}[\cA_n]=1+o(1)$ as $n\to\infty$. 
\end{enumerate}
\end{lem}
\begin{proof}
%Let $\cA_n$ be the event defined by, for $\mu_k$ 
% as in \eqref{munu} and $n_\ga=d N_\ga$, 
%\begin{equation} \label{defan}
%\cA_n=\left\{ \mu_k \ge f_{0,k}/4, \ \ \ \text{for all }\ k=1,\ldots, n_\ga \right\}.
%\end{equation}
One first notes that for small enough $d$, for any $k\le n_\ga=dN_\ga$, one has $f_{0,k}/2\ge (n\sigma_k)^{-1}$, by definition of $f_{0,k}$. Since under $P_{f_0}$ we have $\mu_k=X_k-(n\sigma_k)^{-1}=f_{0,k}+\veps_k/\sqrt{n}-(n\sigma_k)^{-1}$, it holds $\mu_k\ge f_{0,k}/2+\veps_k/\sqrt{n}\ge  f_{0,k}/2-\sqrt{2\log{n}/n}$ on the event \[ \cB_n=\left\{ |\veps_k|\le \sqrt{2\log{n}},\ \text{ for all } k=1,\ldots,n_\ga\right\}.\]
A union bound shows that $P_{f_0}[\cB_n^c]=o(1)$. Also, by definition of $f_{0,k}$ and $n_\ga$, we have $\sqrt{2\log{n}/n}\le f_{0,k}/4$ so that $\cB_n\subset \cA_n$ which implies $P_{f_0}[\cA_n]=1+o(1)$ as $n\to\infty$.

On the other hand, note that $nf_{0,k}^2 \ge n n_\ga^{-1-2\be}\geqa n^{(\al-\be)/(2\ga+1)}$, so that $n\mu_k^2\geqa n^{(\al-\be)/(2\ga+1)}$ on the event $\cA_n$. This shows that $e^{n\mu_k^2/2} \Phi(\sqrt{n}\mu_k)\geqa e^{c_2n^{(\al-\be)/(2\ga+1)}}$ uniformly over $k=1,\ldots,n_\ga$, for some $c_2>0$.

 Also, by definition, $\nu_k\ge \mu_k$ and the latter is positive on $\cA_n$, so 
 %$\nu_k$ goes to infinity for the considered $k$'s uniformly faster than a power of $n$. 
 since $e^{u^2/2}\bar\Phi(u)\leqa 1/u$ for $u>0$, one obtains 
$e^{n\nu_k^2/2} \bar\Phi(\sqrt{n}\nu_k)\le 1/(\sqrt{n}\nu_k)$ uniformly over  $k\le n_\ga$, which is bounded from above by a constant (since $\sqrt{n}\nu_k\ge \sqrt{n}\mu_k$ goes to infinity on $\cA_n$).

Putting together the two previously obtained bounds, one obtains on $\cA_n$
\[ 1- w_k^+ \le  \frac{ e^{n\nu_k^2/2} \bar\Phi(\sqrt{n}\nu_k) }{ 
  e^{n\mu_k^2/2} \Phi(\sqrt{n}\mu_k)} \leqa e^{-c_2n^{(\al-\be)/(2\ga+1)}}/(\sqrt{n}\nu_k) 
  \leqa e^{-c_2n^{(\al-\be)/(2\ga+1)}}, \]
  uniformly over $k=1,\ldots,n_\ga$, which concludes the proof.  
\end{proof}

\subsection{Series prior: Upper bound ($p=1$) for classical posteriors $\rho =1$}
\label{secapp:ubpost}
The next result provides an example of extension of Theorem \ref{thm : conc series} in the main paper (which considers $\rho$--posteriors, $\rho<1$) to classical posteriors ($\rho=1$). It focuses on the case of Laplace tails $p=1$ for simplicity, although a similar result is expected to hold for other $p$'s as well, albeit with more technical proofs, so for clarity we focus on $p=1$ here. 
\begin{thm} \label{ub1}
Let $\alpha>\beta>0$. Suppose the data $X$ follows the white noise model for some true function $f_0$, and let the prior on $f$ be as in \eqref{def : prior series} and taking $\zeta_k$'s to be drawn iid from a standard Laplace distribution. For any function $f_0\in\cF(\be,L)$, if 
\[ \veps_n = \veps_n(1,\al,\be) = n^{-\frac{\be}{\al+\be+1}}, \]
then for $M>0$ large enough, for some $b>0$, as $n\to \infty$,
\[ E_{f_0}\Pi[\|f-f_0\|_2 > M \veps_n \given X] \to 0.\]
\end{thm} 
\begin{proof}
Let $K_n:= DN_\ga$, where $D$ is a large enough constant to be chosen below.  % \sbl{update to $K N_\ga$ for $K>0$ to be chosen below} for $a>0$ to be chosen below and let us set $v_n:=(\log n)^b\veps_n$. 
We write $\|f-f_0\|^2=\sum_{k\ge 1} (f_k-f_{0,k})^2 = (\sum_{k\le K_n}+\sum_{k> K_n})(f_k-f_{0,k})^2 $ and distinguish two cases: $k\leq K_n$ and $k> K_n$.

We deal first with the indices $k>K_n$.  
By definition of $N_\ga, \veps_n$ and for $f_0\in \cF(\be,L)$, it holds
\[ \sum_{k> K_n} f_{0,k}^2\le L^2 K_n^{-2\be}=L^2 N_\ga^{-2\be} \leqa \veps_n^2. \]
As $\|f^{[K_n^c]}-f_0^{[K_n^c]}\|_2\le \|f^{[K_n^c]}\|_2 +\|f_0^{[K_n^c]}\|_2$, provided we choose $M$ sufficiently large with respect to $D, L$, it suffices to show that $E_{f_0}\Pi[\|f^{[K_n^c]}\|_2 > M \veps_n/2 \given X] \to 0$. To do so, by Markov's inequality, it suffices to %bound from above the expectation under $E_{f_0}$ of 
%  $\sum_{k>K_n}\int 
%f_k^2 d\Pi(f_k\given X_k)$. 
%
%% one directly bounds the second moment, as the bias is negligible. 
%Our aim is to 
 show that, on an event of high probability,
\begin{equation} \label{highfr}
\sum_{k>K_n} \int f_k^2 d\Pi(f_k\given X_k)  = o(\veps_n^2). 
\end{equation}
 %By Markov's inequality this readily implies that the contribution of the frequencies $k > K_n$ to the posterior rate (that is, the contribution of the term $\sum_{k>K_n} (f_k-f_{0,k})^2$ to the total squared-norm $\|f-f_0\|_2^2$) is negligible. 
Let us consider the event 
\[ \cA := \bigcup_{j\ge 0}\ \left\{ \max_{jn< k\le (j+1)n} |\veps_k| \le \sqrt{2\log{\{(j+1)^2n}\}} \right\}. \] 
 A union bound argument shows that $P[\cA^c]=o(1)$.  By definition of $K_n, \sigma_k$ and the fact that $f_0\in\cF(\be,L)$, we have, for any $k>K_n$ that $|f_{0,k}|\le (1/2)(n\sigma_k)^{-1}$, provided the constant $D$ is chosen large enough. 
 This implies, on the event $\cA$, that for any $K_n\le k \le n$, 
 \[ \mu_k=X_k-\frac{1}{n\si_k}\le -\frac{1}{2n\si_k}+\sqrt{\frac{2\log{n}}{n}}. \]
 Since $\al>\be$  we have $\sqrt{\log{n}/n}=o(1/(n\si_k))$ for $k>K_n$, so that  for $K_n< k \le n$, on $\cA$,
 \[ \mu_k \le -\frac{1}{4n\si_k}. \]
 Since the growth of $\sigma_k^{-1}$ is polynomial in $k$, this also implies that the inequality in the last display also holds for large enough $n$ and any index $k>n$, so that the inequality holds for all $k>K_n$. Similarly, we have,  on the event $\cA$ and for any $k>K_n$,
 \[ \nu_k \ge \frac{1}{4n\si_k}. \]
 Now bounding $w_k^+$ and $1-w_k^+$ by $1$, one can bound from above, on the event $\cA$,
 \begin{align*}
  \int f_k^2 d\Pi(f_k\given X_k) \le \int f_k^2d\cN(\mu_k,1/n)_+(f_k)
 + \int f_k^2d\cN(\nu_k,1/n)_-(f_k) \le \frac{2}{n^2}\frac{1}{\mu_k^2} + \frac{2}{n^2}\frac{1}{\nu_k^2},
\end{align*}
 where the second inequality follows from Lemma \ref{lemvar} with $\mu=\mu_k, \nu=\nu_k$ and $\sigma^2=1/n$. One deduces that on $\cA$, the last display is bounded from above by $C\sigma_k^2$. Since $\sum_{k>K_n} \sigma_k^2\leqa K_n^{-2\al}=o( K_n^{-2\be})$ since $\al>\be$, and next using that $K_n^{-2\be}=O(\veps_n^2)$, one concludes that $\sum_{k>K_n}\int 
f_k^2 d\Pi(f_k\given X_k)=o_P(\veps_n^2)$ as desired. 
 
It now remains to deal with the indices $k\le K_n$. Here one can follow the final bounds in the lower bound argument in Theorem \ref{lb1} and extend these to any function $f_0\in \cF(\be,L)$. Denoting by $\|\cdot\|_{K_n}$ the $L^2$--norm truncated to the first $K_n$ coefficients, let us define a `centering' function $h=h(X)$ from its basis coefficients $(h_k)$ as follows: $h_k=0$ for $k>K_n$ and, for $k\le K_n$,
\begin{equation}\label{defh}
h_k = 
\begin{cases}
 \, \mu_k:= X_k - 1/(n\sigma_k) & \text{if } \mu_k>(\log{n})/\sqrt{n},\\
 \, \nu_k:= X_k + 1/(n\sigma_k) & \text{if } \nu_k<-(\log{n})/\sqrt{n},\\
 \, 0 & \text{otherwise}.
\end{cases}
\end{equation}
%\sbl{
%Maybe something simpler to avoid random indices below would be 
%\begin{equation}\label{defh}
%h_k = 
%\begin{cases}
% \, \mu_k:= X_k - 1/(n\sigma_k) & \text{if } f_{0,k}-(n\sigma_k)^{-1}>(\log{n})/\sqrt{n},\\
% \, \nu_k:= X_k + 1/(n\sigma_k) & \text{if } f_{0,k}+(n\sigma_k)^{-1}<-(\log{n})/\sqrt{n},\\
% \, 0 & \text{otherwise}.
%\end{cases}
%\end{equation}
%a bit annoying that then it seems we need $M\to \infty$. It may be OK by self-normalisation as we have constructed the exponential term in front so that it kills any polynomial term so it is OK even if we have to handle  $\nu_k^2\le 2\mu_k^2+2(1/n\sigma_k)^2$ (one bounds differently each term).
%}
Note that the first two cases in the above definition are mutually exclusive, as $\mu_k\le \nu_k$ by definition.  Now one can further write, for $M$ a large constant to be chosen below, 
\begin{align*}
\Pi[&\|f-f_0\|_{K_n} \ge M\veps_n\given X] \\
& = \Pi[\|f-f_0\|_{K_n}\ge M\veps_n\given X]\1\{\|f_0-h\|_{K_n}\le M\veps_n/2\} \\
 & \qquad\qquad +\Pi[\|f-f_0\|_{K_n}\ge M\veps_n\given X] \1\{ \|f_0-h\|_{K_n} > M\veps_n/2 \} \\
& \le \Pi[\|f-h\|_{K_n}\ge M\veps_n/2\given X] + \1\{ \|f_0-h\|_{K_n} > M\veps_n/2 \},
\end{align*}
where one uses the triangle inequality and that indicators and probabilities are bounded from above by $1$. It now suffices to show that the expectation under $P_{f_0}^{(n)}$ of the last display goes to $0$. 

Starting first with the indicator, and denoting by $e(\cdot)$ the function with coefficients $e_k=\veps_k$ for $\le K_n$ and $0$ otherwise, applying the triangle inequality gives 
\[ P_{f_0}^{(n)}[\|f_0-h\|_{K_n} > M\veps_n/2] \le  
P[ \|e\|_{K_n}/\sqrt{n} + \| (1/(n\sigma_k)) \|_{K_n} > M\veps_n/2 ]. 
\]
By our choice of $K_n$, we have $\| (1/(n\sigma_k)) \|_{K_n}^2=n^{-2}\sum_{k=1}^{KN_\ga} \sigma_k^{-2}\leqa \veps_n^2$. Hence for $M$ large enough the last norm in the above display is less than $M\veps_n/4$. It now suffices to bound 
$P[ \|e\|_{K_n}/\sqrt{n} > M\veps_n/4 ]$. By Markov's inequality, this is bounded by a multiple of $(n\veps_n^2)^{-1}E[  \|e\|_{K_n}^2]=(n\veps_n^2)^{-1} K_n^2 = o(1)$, by definition of $\veps_n, K_n$. 

Finally, it now remains to deal with the term  $E_{f_0}\Pi[\|f-h\|_{K_n}\ge M\veps_n/2\given X]$. By Markov's inequality, to show that this term is a $o(1)$, it suffices to check that $E_{f_0} \int \|f- h\|_{K_n}^2 d\Pi(f\given X)=o(\veps_n^2)$, which is done in Lemma \ref{lemtech} below. This concludes the proof of Theorem \ref{ub1}.
\end{proof}

Let us recall that $\cN(\mu,\si^2)_+$ denotes the distribution of $Z\vee 0$ if $Z\sim \cN(\mu,\si^2)$ and $\cN(\mu,\si^2)_-$ the one of $Z\wedge 0$.
\begin{lem} \label{lemvar}
For any $\mu<0$ and $\si^2>0$, the following bound holds
\[  \int x^2 d\cN(\mu,\si^2)_+(x) \le 2 \frac{\si^4}{\mu^2}.\]
Similarly, for any $\nu<0$ and $\si^2>0$, it holds $\int x^2 d\cN(\nu,\si^2)_-(x) \le 2 \si^4/\nu^2$.
\end{lem}
\begin{proof}
The integral in the display of the lemma equals, with $z=-\mu>0$,
\[ \cI:=\int_0^\infty x^2 \exp\{-(x+z)^2/(2\si^2)\}dx/ \cD, \]
where $\cD=\int_0^\infty  \exp\{-(x+z)^2/(2\si^2)\}dx$. Integrating by part once gives
\[ \cD\cdot \cI = \int_0^\infty \frac{2x}{x+z} \si^2 \exp\{-(x+z)^2/(2\si^2)\}dx, \]
where the bracket term vanishes both at $0$ and infinity. Integrating by part once more,
\[ \cD\cdot \cI = \int_0^\infty \frac{2}{(x+z)^2} \si^4 \exp\{-(x+z)^2/(2\si^2)\}dx. \]
Using $x+z\ge z>0$ and $z^2=\mu^2$, the last display is bounded from above by $2\sigma^4\cD/z^2$, which implies the first bound of the lemma. The second part follows by symmetry.  
\end{proof}

\begin{lem} \label{lemtech}
Let $h$ be the random function defined by \eqref{defh}. Then for any $D>0$ and $K_n:=DN_\ga$, it holds
\[ E_{f_0} \int \|f- h\|_{K_n}^2 d\Pi(f\given X)=o(\veps_n^2), \]
where $\veps_n$ is the rate in the statement of Theorem \ref{ub1}.
\end{lem}
\begin{proof}
We use the expression \eqref{explicitlap} of the posterior  and write the integral in the Lemma 
\[  \cI:= \sum_{k=1}^{K_n} \left[w_k^+ \int (f_k-h_k)^2d\cN(\mu_k,1/n)_+(f_k)
 +  (1-w_k^+) \int (f_k-h_k)^2d\cN(\nu_k,1/n)_{-}(f_k) \right]. \]
Let us distinguish three cases depending on the split of indices in the definition \eqref{defh} of $h$, namely 
$\cC_1=\{k\le K_n:\ \mu_k>(\log{n})/\sqrt{n}\}$ (Case 1),  $\cC_2=\{k\le K_n:\ \nu_k<-(\log{n})/\sqrt{n}$ (Case 2) and  $\cC_3:=\{ k\le K_n:\ \mu_k\le (\log{n})/\sqrt{n},\  \nu_k\ge -(\log{n})/\sqrt{n}\}$ 
(Case 3). Note that the corresponding sets of indices is random, so has random cardinality, but in all cases once we bound the corresponding quantities by convenient upper-bounds, we will eventually just bound the sum over the corresponding $k$'s (which is over a random set) simply by that over all $k\le K_n$. 

Starting with Case 3, since $h_k=0$ in that case and using $0\le w_k^+\le 1$, it is enough to bound from above 
\[  \sum_{k=1}^{K_n} \left[\int f_k^2d\cN(\mu_k,1/n)_+(f_k)
 + \int f_k^2d\cN(\nu_k,1/n)_{-}(f_k) \right]\1\{ k\in \cC_3\}.\]
 By symmetry it is enough to deal with the first integral in the last display, the other being dealt with similarly. We thus focus on bounding from above $I_k:=\int f_k^2d\cN(\mu_k,1/n)_+(f_k)$. The following simple bound always holds $I_k=E[Z_+^2]\le E[Z^2]\le 2\mu_k^2+2/n$. In case $\mu_k\ge 0$, for $k\in\cC_3$ by definition this is then further bounded by $2(\log{n})^2/n+2/n$. In case $\mu_k<0$, on top of the previous bound we can also now use Lemma \ref{lemvar} to get $I_k\le 2/(n\mu_k)^2$. This means that in that case $I_k/2\le \min(\mu_k^2+1/n,1/(n\mu_k)^2)\le 2/n$ (by comparing the bound to the case $\mu_k=1/\sqrt{n}$). Hence 
 \[\sum_{k=1}^{K_n} \int f_k^2d\cN(\mu_k,1/n)_+(f_k) \1\{ k\in \cC_3\}
 \le \sum_{k=1}^{K_n} 4(\log^2{n}) /n = 4K_n (\log^2{n}) /n =o(\veps_n^2),
  \]
where for the last comparison we use $\al>\be$. 

Now dealing with Case 1, we use the expression of the weight $w_k^+$ to get
\[ 1-w_k^+ \le \frac{e^{n\nu_k^2/2}\bar\Phi(\sqrt{n}\nu_k)}{e^{n\mu_k^2/2}\Phi(\sqrt{n}\mu_k)}
\le \frac{2}{\sqrt{n\nu_k}} e^{-n\mu_k^2/2}, \]
by using $\Phi(\sqrt{n}\mu_k)\ge\Phi(0)=1/2$ and the bound $\bar\Phi(y)\le \phi(y)/y$ for $y>0$. Since in Case 1 one both has $\nu_k\ge \mu_k>0$ and $\mu_k >(\log{n})/\sqrt{n}$, one gets, for $k\in\cC_1$, 
\[ 1-w_k^+ \le \frac{2}{\log{n}} e^{-(\log{n})^2/2}.\]  
This implies the bound
\begin{align*}
\cI_1&:=\sum_{k=1}^{K_n} \left[w_k^+ \int (f_k-h_k)^2d\cN(\mu_k,1/n)_+(f_k)
 +  (1-w_k^+) \int (f_k-h_k)^2d\cN(\nu_k,1/n)_{-}(f_k) \right]\1_{k\in \cC_1} \\
& \le \sum_{k=1}^{K_n} \left[ \int (f_k-\mu_k)^2d\cN(\mu_k,1/n)_+(f_k)
 +   C e^{-(\log{n})^2/2} \int (f_k-\mu_k)^2d\cN(\nu_k,1/n)_{-}(f_k) \right]\1_{k\in \cC_1}
\end{align*}
Both terms are now bounded in a similar way as in the proof of Theorem \ref{lb1}. The first integral in the last line is bounded from above by $\int (u-\mu_k)^2\sqrt{n}\phi(\sqrt{n}(u-\mu_k))\1{u\ge 0}du/\Phi(0)$, using $\Phi(\sqrt{n}\mu_k)\ge \Phi(0)=1/2$ since $\sqrt{n}\mu_k>0$, so that the integral is at most $2/n$. One deduces
\[ \cI_1\le \sum_{k=1}^{K_n} \left[ \frac{2}{n} +C e^{-(\log{n})^2/2} \int (f_k-\mu_k)^2d\cN(\nu_k,1/n)_{-}(f_k) \right]. \]
Now note that the expectation under $E_{f_0}$ of the integral in the last display has been bounded from above in the proof of Theorem \ref{lb1} by a constant $C$ ($n_\ga$ therein is replaced by $K_n$, which does not change the bound, up to a multiplicative constant). This implies
\[ \cI_1 \le 2K_n/n + CK_n e^{-(\log{n})^2/2}=o(\veps_n). \]
Finally Case 2 is handled exactly as Case 1, by symmetry. Putting the three obtained bounds together concludes the proof.
\end{proof}

\subsection{SNN: prior with $p>1$}
The following theorem complements Theorem \ref{thm:SNN} in the lighter than Laplace tails case ($p >1$).
\begin{thm}\label{thm:SNNhighb}
Consider the setting of Theorem \ref{thm:SNN} but with $p >1$.  For any $\rho \in (0,1)$, denoting $D_\rho$ the Rényi divergence  \eqref{def : renyi}, there exists a large enough constant $M >0$, such that, as $n \to \infty$,
\[ E_{f_0} \Pi_\rho \left[ \left\{ f \, : \, \frac1n D_\rho(P_f^n,P_{f_0}^n) \ge M \varepsilon_n^2 \right\} \mid X,Y\right] \to 0,\]
where, letting $\veps_n^\ast=n^{-\frac{\beta}{1+2\beta}}$ and $\veps_n^+=n^{-2/5}$ (equal to $\veps_n^\ast$ for $\beta=2$), $\veps_n$ is given as follows:
\begin{enumerate}
\item[a)] if $\beta\in(0,1+\frac{1}{p}]$
\begin{enumerate}
    \item[i)] (Oracle $\sigma_n$) for $\sigma_n=N_\alpha^{-\frac{2}{2+p}}N_\beta^{\frac{2}{2+p}-\beta}\log^{-\frac{2}{q(2+p)}}(n)$,
\[\veps_n=\veps_n^\ast(N_\alpha N_\beta)^{\frac{p}{2+p}}\log^{\frac{p}{q(2+p)}}(n);\]
    \item[ii)] (Non-oracle $\sigma_n$) for $\sigma_n=\veps_n^+/N_\alpha$
    \[\veps_n=\veps_n^\ast n^{\frac{p}2(\frac{1-\beta}{1+2\beta}+\frac25+\frac1{1+2\alpha})};\]
\end{enumerate} 
\item[b)] if $\beta\in(1+\frac{1}{p},2]$ 
\begin{enumerate}
    \item[i)] (Oracle $\sigma_n$) for $\sigma_n=(\veps_n^\ast)^{\frac2{2+p}}N_\alpha^{-\frac{2}{2+p}}N_\beta^{-\frac{1}{2+p}}\log^{-\frac{2}{q(2+p)}}(n)$,
\[\veps_n=\veps_n^\ast N_\alpha^{\frac{p}{2+p}}N_\beta^{\frac{\beta p -1}{2+p}}\log^{\frac{p}{q(2+p)}}n;\]
    \item[ii)] (Non-oracle $\sigma_n$) for $\sigma_n=\veps_n^+/N_\alpha$
    \[\veps_n=\veps_n^\ast n^{\frac{p}5-\frac{1}{2(1+2\beta)}};\]
\end{enumerate} 
\end{enumerate}
\end{thm}

\begin{proof}
The proof is identical to the proof of Theorem \ref{thm:SNN} up to \eqref{sumbound}, which we can still get using $|a+b|^p\leq c(p)(|a|^p+|b|^p)$, valid for any $p>0$. We next note that in the case $p>1$
\begin{enumerate}
    \item {for $\beta\in(0,1]$}, we have $1+p-\beta p>p(1-\beta)_+$, hence again the second term dominates the first in \eqref{sumbound} and overall in the right hand side of the bound;
    \item for $\beta\in(1,2]$, the second term dominates the first for $p\leq (\beta-1)^{-1}$, otherwise the first term dominates in the right hand side of the bound.
 \end{enumerate}   
     For  $\beta\in(0, 1+\frac{1}{p}]$, the remaining of the proof is still identical to the proof of Theorem \ref{thm:SNN}, and we only need to deal with the case $\beta\in(1+\frac{1}{p},2]$, in which the first term dominates in \eqref{sumbound} and, in fact, the sum is bounded by a constant. We hence get, for some constant $c_3$,
\[I\ge\exp\Big(-N_\beta\big(c_2+\log(\frac{N_\alpha \sigma_n}{\veps_n})\big)-c_3\sigma_n^{-p}\Big)\]
and, under assumption \eqref{bcond1}, it holds
\[\Pi( || f - f_0||_\infty \leq  \varepsilon_n) \geq\exp\left(-N_\beta\big(c_2+\log(\frac{N_\alpha \sigma_n}{\veps_n})\big)-c_3\sigma_n^{-p}\right).\] The latter is lower bounded by $\exp(-c_4'n\veps_n^2)$ provided $\veps_n$ satisfies \eqref{bcondaux} and 
\begin{equation}\label{bcond2b}\veps_n\gtrsim n^{-1/2}\sigma_n^{-p/2}.\end{equation}
Combining, to have the desired prior mass bound, it suffices that \eqref{bcond1}, \eqref{bcondaux} and \eqref{bcond2b} hold.

We optimize the choice of $\sigma_n$  based on \eqref{bcond1},\eqref{bcond2b}, and then check that \eqref{bcondaux} also holds. Since \eqref{bcond1},\eqref{bcond2b} imply that 
\begin{equation}\label{highbeta}\veps_n\gtrsim \{\sigma_nN_\alpha \log^{1/q}n\}\vee \{n^{-1/2}\sigma_n^{-p/2}\}\asymp\{\sigma_nN_\alpha \log^{1/q}n\}\vee \{\veps_n^\ast N_{\beta}^{-1/2}\sigma_n^{-p/2}\},\end{equation}
where the first term in the maximum improves with a faster decay of $\sigma_n$ while the second deteriorates, we choose $\sigma_n$ to balance the two terms, resulting in 
\[\sigma_n\asymp (\veps_n^\ast)^{\frac2{2+p}}N_\alpha^{-\frac2{2+p}}N_\beta^{-\frac{1}{2+p}}\log^{-\frac2{q(2+p)}}n.\]
This results in
\[\veps_n\gtrsim (\veps_n^\ast)^{\frac2{2+p}} N_\alpha^{\frac{p}{2+p}}N_\beta^{-\frac{1}{2+p}}\log^{\frac{p}{q(2+p)}}n=\veps_n^\ast N_\alpha^{\frac{p}{2+p}}N_\beta^{\frac{\beta p -1}{2+p}}\log^{\frac{p}{q(2+p)}}n.\]
This $\veps_n$ also satisfies \eqref{bcondaux} (e.g. since $\beta, p >1$, which guarantees that $\veps_n$ polynomially slower than $\veps_n^\ast$), as required.

Choosing $\sigma_n=\veps_n^+/(N_\alpha \log^{1/q}{n})$, for $\veps_n^+=n^{-2/5}$ (the minimax rate for $\beta=2$), again makes the first term trivial and gives rise to the constraint 
$\veps_n\gtrsim \veps_n^\ast n^{\frac{p}5-\frac{1}{2(1+2\beta)}}$. 
\end{proof}

\begin{rmk}
Let $p>0, \beta\in(0,2]$ and $0<\alpha\le \beta$. Displays \eqref{eq:oracle} and \eqref{highbeta}, in the proofs of Theorems \ref{thm:SNN} and \ref{thm:SNNhighb}, respectively, show that for any choice of $\sigma_n$, the prior mass bound holds with
\[\veps_n\geq\{\sigma_nN_\alpha \log^{1/q}n\}\vee \{\veps_n^\ast\Big(\frac{N_\beta^{1-\beta}}{\sigma_n}\Big)^{p/2}\},\] for $\beta\in(0,1+\frac1{1\vee p}]$, and with
\[\veps_n\geq \{\sigma_nN_\alpha \log^{1/q}n\}\vee \{\veps_n^\ast N_{\beta}^{-1/2}\sigma_n^{-p/2}\},\]
for $\beta\in(1+\frac1{1\vee p},2]$ (a regime admissible only for $p>1$), as long as \eqref{bcondaux} holds. In particular, for $\beta\in(0,1+\frac1{1\vee p}]$, the standard choices $\sigma_n=n^{-1/2}$ and $\sigma_n=n^{-1}$, give rise to the constraints
\[\veps_n\ge n^{\frac{1-2\alpha}{2(1+2\alpha)}}\log^{1/q}{n}\vee \veps_n^\ast n^{\frac{p}2\big(\frac{1-\beta}{1+2\beta}+\frac12\big)}\]
and
\[\veps_n\ge n^{\frac{-2\alpha}{1+2\alpha}}\log^{1/q}{n}\vee \veps_n^\ast n^{\frac{p}2\big(\frac{1-\beta}{1+2\beta}+1\big)},\]
respectively, which are limited by the choice of $\alpha$ (influencing the network's width), even when $p$ is small. Given that the second terms in the maxima tend to $\veps_n^\ast$ for small $p>0$, the choice $\sigma_n=n^{-1}$ appears to be better, since it improves the first term which is independent of $p$. Finally, for $p>1$ and $\beta\in(1+\frac1{1\vee p},2]$, the standard choices $\sigma_n=n^{-1/2}$ and $\sigma_n=n^{-1}$, give rise to the constraints
\[\veps_n\ge n^{\frac{1-2\alpha}{2(1+2\alpha)}}\log^{1/q}{n}\vee \veps_n^\ast n^{\frac{p}4-\frac{\beta}{1+2\beta}}\]
and
\[\veps_n\ge n^{\frac{-2\alpha}{1+2\alpha}}\log^{1/q}{n}\vee \veps_n^\ast n^{\frac{p}2-\frac{\beta}{1+2\beta}},\]
respectively,
which are similarly limited by the choice of $\alpha$. 
\end{rmk}

\section{Technical lemmas}\label{sec: technical res}
\subsection{Theory on contraction for $\rho$--posteriors}\label{sec : rho-post}
Given a statistical model $(P_f^{(n)})_{f \in \cF}$ and $f_0 \in \cF$, define the Kullback--Leibler neighborhood of $f_0$
\begin{equation}\label{def : KLvois}
    \cB_n(f_0,\eps ) := \left\{ f \in \cF \, : \, \int \log \frac{dP_{f_0}^{(n)}}{dP_{f}^{(n)}} dP_{f_0}^{(n)} \leq n\veps^2 \, , \, \int \log^2 \frac{dP_{f_0}^{(n)}}{dP_{f}^{(n)}} dP_{f_0}^{(n)} \leq n\veps^2 \right\}. 
\end{equation} 
\begin{lem}[Theorem 4.1 in \cite{ltcr23}]\label{lem : renyi from prior mass}
    Let $(P_f^{(n)})_{f \in \cF}$ be a statistical model and assume data generated as $X^{(n)} \sim P_{f_0}^{(n)}$ for some $f_0 \in \cF$. Let $(\veps_n)$ be a positive sequence such that $\veps_n \to 0$ and $n \veps_n^2 \to \infty$ as $n\to \infty$, suppose $\Pi$ is a prior distribution on $\cF$, satisfying
    \begin{equation} \label{eq : prior mass cond}
        \Pi \left[ \cB_n(f_0,\veps_n ) \right] \ge e^{- n \veps_n^2}.
    \end{equation} 
    Then, for any $\rho \in (0,1)$, there exists a large enough constant $M>0$, such that, as $n \to \infty$,
    \[ E_{f_0} \Pi_\rho \left[ \left\{ f \, : \,  \frac1nD_\rho(P_{f}^{(n)},P_{f_0}^{(n)}) \le M \varepsilon_n^2  \right\}\given X^{(n)}\right] \to 1.\]
\end{lem}

\begin{lem}[See e.g. \cite{cstf} and Lemmas 20,21 in \cite{ce25}]\label{lem : vois computations}
    Let $f_0 \in \cF$, $\cB_n(f_0,\veps)$ the set in \eqref{def : KLvois} and  $(P_f^{(n)})_{f \in \cF}$ be the statistical model given 
    \begin{enumerate}
        \item either by Equation \eqref{def : gwn}, in which case for any $\rho \in (0,1)$,
        \begin{align*}
             \left\{ f \in \cF \, : \, ||f-f_0||_2 \le \veps \right\} \subset \cB_n(f_0,\veps) \qquad \text{and} \qquad
            \frac1nD_\rho(P_{f}^{(n)},P_{f_0}^{(n)}) = \frac{\rho}{2(1-\rho)}||f-f_0||_2^2 .
        \end{align*}
        \item or by Equation \eqref{def : random design}, and then there exists a constant $C>0$ such that
        \begin{align*}
            \left\{ f \in \cF \, : \, ||f-f_0||_\infty \le \veps \right\} \subset \cB_n(f_0,C\veps).
        \end{align*}
        Further assuming $||f||_\infty \vee ||f_0||_\infty \leq F$, we have, for any $\rho \in (0,1)$,
        \[ \frac1nD_\rho(P_{f}^{(n)},P_{f_0}^{(n)}) \ge \frac{\rho}{2} e^{-2 F^2 \rho(1-\rho)} ||f-f_0||_{2,P_X}^2. \]
    \end{enumerate}
\end{lem}
\subsection{Lemmas for series priors}
In this Section we regroup different Lemmas used in the proofs of Section \ref{sec : series}.
\begin{lem} \label{lemlog}
Let $\al>0$ and let $N_\al$ be defined in \eqref{def : Ncutoff}. For $n\ge n_0(\al)$, we have
\[ \prod_{k=1}^{N_\al} \frac{k^{1/2+\al}}{\sqrt{n}} \ge e^{- (1/2+\al)N_\al}. \]
\end{lem}
\begin{proof}
The product in the statement equals $\exp\{ (1/2+\al)\sum_{k=1}^{N_\al} \log{k} - (N_\al/2)\log{n} \}$. Using a comparison series/integral, the partial sum is bounded from below by $\int_1^{N_\al} \log(x)dx=N_\al\log{N_\al}-(N_\al-1)$ by integration by parts. On the other hand,
\[-(N_\al/2)\log{n} = - \frac{1+2\al}{2}N_\al \log\left(n^{\frac{1}{2\al+1}}\right) 
\ge -(\frac12+\al)N_\al\log(N_\al+1).
\]
Combining the previous two bounds gives, using the inequality $\log(1+x)\le x$,
\begin{align*}
 (1/2+\al)\sum_{k=1}^{N_\al} \log{k} - (N_\al/2)\log{n}
& \ge (1/2+\al) \left[ -N_\al\log(1+N_\al^{-1})-(N_\al-1) \right] \\
& \ge -(1/2+\al)N_\al%\ge -(1+2\al)N_\al.
\end{align*}
Taking exponentials on both sides gives the result.
\end{proof}

\begin{lem} \label{lemlogbis}
Let $\al,\be,p>0$ with $p\ge 2$ and $\al\ge\be$. Let $\ga$ and $N_\ga$ be defined in \eqref{def : gamma} and \eqref{def : Ncutoff} respectively. For $n\ge n_0(\al,\be,p)$, we have%, for $c_1=1/2+\al+(\ga-\al)\log{2}$, % and $C=C(\al,\be,p)$, we have
\[ \prod_{k=1}^{N_\ga} \frac{k^{1/2+\al}}{\sqrt{n}} \ge e^{-(\ga-\al)-(1/2+\al)N_\ga - N_\ga^{1+p(\al-\be)}}. \]
\end{lem}
\begin{proof}
Proceeding as in the proof of Lemma \ref{lemlog}, the logarithm of the product in the display of the Lemma is bounded from below by, taking $n$ large enough so that $N_\ga\ge 1$,
\begin{align*}
& (1/2+\al)\{ N_\ga\log{N_\ga} -(N_\ga-1) \} - N_\ga(1/2+\ga)\log(N_\ga+1) \\
& \ \ \ge (1/2+\al)\{ N_\ga\log{N_\ga} -N_\ga \} +1/2+\al- N_\ga(1/2+\ga)\log{N_\ga}
-N_\ga(1/2+\ga)\log(1+N_\ga^{-1}).
\end{align*}
The last term is bounded from below by $-(1/2+\ga)$, using $\log(1+x)\le x$ for $x>0$. Regrouping the terms, the last display is further bounded from below by
%, using $\al\le \ga$ for $p\ge 2$,
\[ -(\ga-\al)-(\ga-\al) N_\ga\log{N_\ga} -(1/2+\al)N_\ga. \]
%Since $\log(N_\ga+1)\le \log{2}+\log(N_\ga)$ for large enough  $n$ (granting $N_\ga\ge 1$) and 
With $(\ga-\al)N_\ga\log(N_\ga)=N_\ga\log(N_\ga^{\ga-\al})\le N_\ga^{1+p(\ga-\al)}$, using $p\ge 2\ge 1$, $\ga\ge \al$ for $p\ge 2$ and $\log(x)\le x$ for $x>0$ the result follows.
\end{proof}
\begin{lem} \label{lem : bounds on the event lower bound}
Let $f_0$ be defined by $f_{0,k}=L k^{-1/2-\be}$ for $\be>0$ and let $\sigma_k=k^{-1/2-\al}$ for $\al > \be$. Let $p < 2$ and recall the definitions of $\gamma$  and $N_\ga$ in \eqref{def : gamma}--\eqref{def : Ncutoff}. Let $n_\ga :=d N_\ga$ for some constant $d>0$. In the (projected) Gaussian white noise model \eqref{def : gwn}, define the event 
\begin{equation}\label{def : eventlowerbound}
     \cB_n := \left\{ |\xi_k| \leq \sqrt{2 \log n}, \quad \text{for all } k=1,\dots,n_\ga \right\}. 
\end{equation}
We have $P_{f_0}^{(n)}(\cB_n) \to 1$, as $n \to \infty$. Also, for all $k= 1, \dots,n_\gamma$ and large enough $n$, the following hold on $\cB_n$: 
\begin{enumerate}
    \item $f_{0,k}/2 \le X_k \le 3f_{0,k}/2,$
    \item $L\sigma_k \le 2X_k$,
    \item for any constant $M = M(p)$, one can choose $d>0$ small enough in $n_\ga := dN_\ga$, such that
    \[nX_k^2 \ge M \left( \frac{X_k}{\si_k}\right)^p .\]
\end{enumerate}
\end{lem}
\begin{proof}
A union bound directly shows that $P_{f_0}^{(n)}[\cB_n^c]=o(1)$. Also, by definition of $f_{0,k}$ and $\si_k$ the second point $L\sigma_k \le 2X_k$ immediately follows from the first. Let us check the latter. For $k \leq n_\gamma$ and $n$ large enough, we have
\[ f_{0,k} = L k^{-\be - \frac12} \ge L n_\ga^{-\be - \frac12} \ge 2 \sqrt{\frac{2 \log n}{n}.}\]
Therefore, on the event $\cB_n$,
\[ \frac{f_{0,k}}{2} \le f_{0,k} - \sqrt{\frac{2 \log n}{n}} \leq X_k \le f_{0,k} + \sqrt{\frac{2 \log n}{n}} \le \frac32 f_{0,k}. \]
We are left to check the last point, since $p <2$, on $\cB_n$ we can lower bound, for $n$ large enough,
\[ X_k^{2-p} \sigma_k^{p} \ge 2^{p-2} f_{0,k}^{2-p}\sigma_k^{p}  = (L/2)^{2-p} k^{-(1+2\be +p(\al - \be))} = (L/2)^{2-p} k^{-1 - 2 \ga}.\]
Recalling $n_\ga = d n^{1/(1+ 2 \ga)}$, for any constant $M = M(p)$ and $k \le n_\ga$, one can choose $d>0$ small enough, such that
\[ X_k^{2-p} \sigma_k^{p} \ge (L/2)^{2-p} d^{-1-2\ga} n^{-1} \ge M n^{-1}. \qedhere \]
\end{proof}

\begin{lem}[Bounds for the $p$--exponential distribution]\label{lem : sandwich}
Let $h_p$ be the density function defined in \eqref{def : p-exp dist} and let $\overline{H}_p : x \mapsto \int_0^x h_p(t) \, dt$ be the associated survival function. Denote $Z_p$ the normalizing constant, such that 
\[ h_p(t) = \frac{1}{Z_p} \exp \left\{ \frac{|t|^p}{p} \right\}. \]
Direct computation shows that, for $\Gamma$ for the usual Gamma function, we have 
\[ Z_p = 2 p^{1/p-1}\Gamma(1/p).\]
    For any $p >0$ and $t \in \RR$, we have
    \[\frac{1}{2 \sqrt{2 \pi}} \sqrt{p} e^{1/p- p/12} e^{-\frac{|t|^p}{p}} \leq h_p(t) \leq \frac{1}{2 \sqrt{2 \pi}} \sqrt{p} e^{1/p} e^{-\frac{|t|^p}{p}}. \]
    Also, provided $0<p<1$, we have, for any $x \geq 1$,
    \[ \overline{H}_p(x) \leq \frac{1}{2 \sqrt{2 \pi}}  \frac{e^{1/p}}{\sqrt{p}} x^{1-p} e^{-\frac{|x|^p}{p}}.\]
\end{lem}
\begin{proof}[Proof of Lemma \ref{lem : sandwich}]
    From the Stirling approximation of the Gamma function, the following inequalities are available for any $z >0$ (see Error bounds and exponential improvements for the asymptotic expansions of the gamma function and its reciprocal)
    \[\sqrt{\frac{2 \pi}{z}}\left(\frac{z}{e}\right)^z \leq \Gamma(z) \leq \sqrt{\frac{2 \pi}{z}}\left(\frac{z}{e}\right)^ze^{\frac{1}{12z}},\]
    direct algebraic manipulation gives inequalities for $Z_p$,
    \begin{equation}\label{eq : sandwich zp}
        2 \sqrt{2 \pi} \frac{e^{-1/p}}{\sqrt{p}} \leq Z_p \leq 2 \sqrt{2 \pi} \frac{e^{p/12-1/p}}{\sqrt{p}},
    \end{equation} 
    which leads to the required bounds on $h_p$.
    For the second part of the Lemma, using the definition of $\overline{H}_p(x)$ and the change of variable $t = (pu)^{1/p}$, we get
    \begin{equation}\label{eq : reec Fp}
        Z_p \cdot \overline{H}_p(x) = p^{1/p-1} \int_{x^p/p}^\infty u^{1/p-1}e^{-u} \, du = p^{1/p-1}I\left(\frac1p,\frac{x^p}p \right),
    \end{equation}
    where for any $z >0$ and $s>0$, we define
    \[ I(s,z) := \int_z^\infty e^{-t}t^{s-1}\, dt.\]
    We prove now that if $s>1$ and $z \geq  s$, we have
    \begin{equation}\label{eq : ineq incomp gamma}
        I(s,z) \leq sz^{s-1}e^{-z}.
    \end{equation}
    Indeed, from the change of variables $t = (u+1)z$, we have
    \[ I(s,z) = e^{-z} z^s \int_0^\infty e^{-uz}(u+1)^{s-1} \, du.\]
    Using in succession the inequalities $(u+1)^{s-1} \leq e^{u(s-1)}$ and $(1-s)s \ge (1-s)z$, available for $s >1$ and $z \geq s$, we get 
    \[ \int_0^\infty e^{-uz}(u+1)^{s-1} \, du \leq \int_0^\infty e^{-uz} e^{u(s-1)} \, du = (z-(s-1))^{-1} \leq sz^{-1}.\]
    Combining this inequality with the previous change of variables yields the desired bound on $I(s,z)$.
    Provided $x\geq 1$, such that $x^p/p \geq 1/p$ inequality \eqref{eq : ineq incomp gamma} applied with $s =1/p >1$ and $z = x^p/p$ in \eqref{eq : reec Fp} leads to
    \[ Z_p \cdot \overline{H}_p(x) \leq p^{-1}p^{1/p-1}(x^p/p)^{1/p-1}e^{-x^p/p}= p^{-1}x^{1-p}e^{-x^p/p},\]
    Using \eqref{eq : sandwich zp} the lower bound on $Z_p$ gives the desired result.
\end{proof}
%\ble{From Lemma \ref{lem : sandwich} we see that the density function $f_p$ satisfies condition \eqref{condt} with $c_1 = 1/p$ and $c_0 = (2 \sqrt{2 \pi})^{-1} \sqrt{p}e^{1/p - p/12}$ and a slightly modified version of condition \eqref{condu} with $M_0 =1$, $d_1 = 1/p$ and $d_0 = (2 \sqrt{2 \pi})^{-1} p^{-1/2}e^{1/p} x^{p-1}$ we can then track every constant in the last proof in terms of $p$ and follow the argument while sending $p$ to $0$.}

\subsection{Piecewise affine approximation by SNN}\label{proof : affine approx}
\begin{lem}\label{lem : approx shallow}
    Let $f \in \cH^\beta(L)$ for some $L >0$ and $\beta \in (0,2]$. For any integer $M \geq1$, consider
    \begin{equation*}
   \begin{cases}
       w_0 := M \left[ f\left(\frac1M\right) - f(0)\right], \\
       w_k := M \left[ f\left(\frac{k+1}M\right) -2 f\left(\frac{k}M\right) + f\left(\frac{k-1}M\right)\right] \text{ for } k=1, \dots,M-1
   \end{cases}
    \end{equation*}
    and
    \[f_M : x \mapsto f(0) + \sum_{k=0}^{M-1} w_k \left( x - \frac{k}{M}\right)_+.\]
    Then $f_M$ is a shallow network as in \eqref{def : shallowNN} (by construction), such that:
    \begin{itemize}
        \item $f_M$ coincides with the piecewise affine function $f^\star_M$ changing slope and interpolating $f$ precisely at $k/M$, $k=0,\dots,M$;
        \item $|w_0|\leq LM^{(1-\beta)_+}$ and $|w_k|\leq 2LM^{1-\beta}$ for $k=1,\dots, M-1$;
        \item $f_M \in N\!N_1(M,\,2LM^{(1-\beta)_+} \vee 1)$;
        \item $f_M$ approximates $f$ uniformly
    \[\underset{x \in [0,1]}{\sup}|f(x) - f_M(x)| \leq 2LM^{-\beta}.\]
        \end{itemize} 
\end{lem}
\begin{proof}
    It is clear from its definition that $f_M$ is a shallow ReLU network of width $M$ as in equation \eqref{def : shallowNN}. We next show that $f_M$ coincides with the piecewise affine function $f^\star_M$ changing slope and interpolating $f$ precisely at $k/M$, $k=0,\dots, M$. Indeed, the latter is such that 
    \begin{align*}f^\star_M(x)&=M\left[f\left(\frac{k+1}M\right)-f\left(\frac{k}M\right)\right]x+(k+1)f\left(\frac{k}M\right)-kf\left(\frac{k+1}M\right),\\
    &={f}\left(\frac{k}M\right)+M\left[f\left(\frac{k+1}M\right)-f\left(\frac{k}M\right)\right]\left(x-\frac{k}M\right)_+,\\& x\in I_k:=\left[\frac{k}M,\frac{k+1}M\right), \quad k=0,\dots,M-1.
    \end{align*}
    We proceed by induction. For $k=1$, that is for $x\in I_0$, it is immediate that 
    \[f^\star_M(x)=f(0)+w_0x_+,\]
    where the right hand side above coincides with $f_M$ for $x\in I_0$. Suppose that $f^\star_M=f_M$ on $I_{n-1}$, $n<M$. Then for $x\in I_{n}=\left[\frac{n}M,\frac{n+1}M\right)$, 
    \begin{align*}f_M(x)&=f(0)+\sum_{k=0}^{M-1}w_k\left(x-\frac{k}M\right)_+=f(0)+\sum_{k=0}^nw_k\left(x-\frac{k}M\right)_+\\&=f(0)+\sum_{k=0}^{n-1}w_k\left(x-\frac{k}M\right)_++w_n\left(x-\frac{n}M\right)_+\\&=f(0)+\sum_{k=0}^{n-1}w_k\left(x-\frac{n-1}M+\frac{n-1}M-\frac{k}M\right)+w_n\left(x-\frac{n}M\right)\\
    &=f(0)+\sum_{k=0}^{n-1}w_k\left(\frac{n-1}M-\frac{k}M\right)+\sum_{k=0}^{n-1}w_k\left(x-\frac{n-1}M\right)+w_n\left(x-\frac{n}M\right)\\
    &=f_M\left(\frac{n-1}M\right)+\sum_{k=0}^{n-1}w_k\left(x-\frac{n-1}M\right)+w_n\left(x-\frac{n}M\right)\\
    &=f\left(\frac{n-1}M\right)+\sum_{k=0}^{n}w_k\left(x-\frac{n}M\right)+\sum_{k=0}^{n-1}\frac{w_k}M,\end{align*}
    where in the last equality we used the induction hypthesis for $x=\frac{n-1}M\in I_{n-1}$  and the fact that $f^\star_M$ interpolates $f$. Noticing that for all $m\ge1$
    \[\sum_{k=0}^{m}\frac{w_k}M=f\left(\frac{m+1}M\right)-f\left(\frac{m}M\right),\] we get 
    \[f_M(x)=f\left(\frac{n}M\right)+M\left[f\left(\frac{n+1}M\right)-f\left(\frac{n}M\right)\right]\left(x-\frac{n}M\right)=f^\star_M(x), \quad\forall x\in I_n,\]
    and the claim is proved.

    The bounds on $w_k$, $k=0,\dots, M-1$ for $\beta\in(0,1]$ follow from the definition of $\cH^\beta(L)$, for $k=0$ directly, while for $k\ge1$ after an application of the triangle inequality. For $\beta\in(1,2]$, using the triangle inequality and the mean value theorem, there exist $c\in\left[0, \frac1{2M}\right]$ and $d\in\left[\frac1{2M},\frac1M\right]$, such that
    \[|w_0|=M\left|f\left(\frac1M\right)-f\left(\frac1{2M}\right)+f\left(\frac1{2M}\right)-f(0)\right|=M\left|\frac{f'(c)}{2M}+\frac{f'(d)}{2M}\right|\leq \frac{L}{2}\leq L,\] where we have used the definition of $\cH^\beta(L)$ to upper bound the derivative. For $k=1,\dots,M-1$, again using the mean value theorem, there exist $c_k\in I_k, d_k\in I_{k-1}$, such that 
    \begin{align*}|w_k|&=M\left|f\left(\frac{k+1}M\right)-f\left(\frac{k}{M}\right)-\left[f\left(\frac{k}{M}\right)-f\left(\frac{k-1}{M}\right)\right]\right|\\&=M\left|\frac{f'(c_k)}{M}-\frac{f'(d_k)}{M}\right|\leq L\left(\frac2M\right)^{\beta-1}\leq 2LM^{1-\beta},\end{align*}
    where in the first upper bound we have used the definition of $\cH^\beta(L)$, while in the second the fact that $\beta\leq2$.
    
    The assumption $f \in \cH_\beta(L)$ also implies $|f(0)| \leq L$, while $|k/M| \leq 1$ for all $k\in 0,\dots, M-1$, hence, combining with the bounds on $|w_k|$ we get $f_M \in N\!N_1(M,2LM^{(1-\beta)_+}\vee1)$, as claimed. 

For the last uniform approximation bound, fix $x\in[0,1]$ and let $k$ be the unique index among $\{0,\dots,M-1\}$ such that $x\in I_k$.  Using that $f_M=f^\star_M$, we have  
\begin{equation}\label{eq:lem2}|f(x)-f_M(x)|=\left|f(x)-f\left(\frac{k}M\right)-M\left[f\left(\frac{k+1}M\right)-f\left(\frac{k}M\right)\right]\left(x-\frac{k}M\right)\right|.\end{equation}
If $\beta\in(0,1]$, using the definition of $\cH_\beta(L)$, we have 
\[|f(x)-f_M(x)|\leq \left|f(x)-f\left(\frac{k}M\right)\right|+M\left|f\left(\frac{k+1}M\right)-f\left(\frac{k}M\right)\right|\left|x-\frac{k}M\right|\leq 2L M^{-\beta},\]
since $|x-k/M|\leq M^{-1}$. For $\beta\in(1,2]$,  using the mean value theorem, there exist $c_k, d_k\in I_k$ 
such that 
\[f(x) = f\left(\frac{k}{M}\right) + f'(c_{k})\left(x - \frac{k}{M} \right) \qquad \text{and} \qquad f\left(\frac{k+1}{M}\right) = f\left(\frac{k}{M}\right) + \frac{1}{M}f'(d_{k}). \]
Plugging the last two identities into \eqref{eq:lem2}, using again the definition of $\cH_\beta(L)$ and the fact that $|x-k/M|\leq M^{-1}$, we again get the bound
    \[|f(x) - f_M(x)|= |f'(c_{k})- f'(d_{k})| \left| x - \frac{k}{M} \right| \leq LM^{1-\beta}\left| x - \frac{k}{M} \right| \leq L M^{-\beta}\leq 2LM^{-\beta}. \]
\end{proof}
\end{document}